\documentclass[12pt,a4paper]{amsart}
\makeatletter
\renewcommand\normalsize{%
   \@setfontsize\normalsize{11.7}{14pt plus .3pt minus .3pt}%
   \abovedisplayskip 10\p@ \@plus4\p@ \@minus4\p@
   \abovedisplayshortskip 6\p@ \@plus2\p@
   \belowdisplayshortskip 6\p@ \@plus2\p@
   \belowdisplayskip \abovedisplayskip}
\renewcommand\small{%
   \@setfontsize\small{9.5}{12\p@ plus .2\p@ minus .2\p@}%
   \abovedisplayskip 8.5\p@ \@plus4\p@ \@minus1\p@
   \belowdisplayskip \abovedisplayskip
   \abovedisplayshortskip \abovedisplayskip
   \belowdisplayshortskip \abovedisplayskip}
\renewcommand\footnotesize{%
   \@setfontsize\footnotesize{8.5}{9.25\p@ plus .1pt minus .1pt}
   \abovedisplayskip 6\p@ \@plus4\p@ \@minus1\p@
   \belowdisplayskip \abovedisplayskip
   \abovedisplayshortskip \abovedisplayskip
   \belowdisplayshortskip \abovedisplayskip}
\ifdefined\pdfpagewidth
\else
\fi
\calclayout
\makeatother

\usepackage{graphicx}

\usepackage[T1]{fontenc}
\usepackage[utf8]{inputenc}
\usepackage[frenchb]{babel}

\usepackage{subcaption}

\usepackage[colorlinks=true]{hyperref}
\hypersetup{urlcolor=green,linkcolor=blue,citecolor=red,colorlinks=true}
\usepackage{color}

\newtheorem*{theoreme*}{Théorème}
\newtheorem{theoreme}{Théorème}[section]
\newtheorem{corollaire}[theoreme]{Corollaire}
\newtheorem{lemme}[theoreme]{Lemme}
\newtheorem{proposition}[theoreme]{Proposition}
\theoremstyle{definition}
\newtheorem{definition}[theoreme]{Définition}

\theoremstyle{remark}
\newtheorem{remarque}[theoreme]{Remarque}

\numberwithin{equation}{section}

\newcommand{\norm}[1]{\left\Vert#1\right\Vert}

\def\C{\mathbf{C}}
\def\R{\mathbf{R}}
\def\Z{\mathbf{Z}}

\def\S{\mathbf{S}}
\def\D{\mathbf{D}}

\def\PUC{\mathbf{P}^1_{\C}}
\def\Ptwo{\mathbf{P}^2_{\C}}
\def\PnC{\mathbf{P}^n_{\C}}

\def\J2{J_2}
\def\Jd{J_d}

\def\FJ2{\mathcal{J}_2}
\def\FJd{\mathcal{J}_d}

\def\FF{\mathcal{F}}
\def\FV{\mathcal{F}_V}
\def\F2{\mathcal{F}_2}

\def\GG{\mathcal{G}}

\def\sing {S} 
\def\reg{\text{rég}}

\def \Ptwostar {{\Ptwo} ^*} 

\def \Ptwostarprime { {\Ptwo} ^{* '} } 

\def\W{W}

\def\Pr{\mathcal{P}_{B}}
\def\Qr{\mathcal{P}_{S}}

\def\FSFo{\mathcal F^{--}}
\def\FSFa{\mathcal F^-} 
\def\FIFo{\mathcal F^{++}}
\def\FIFa{\mathcal F^+}

\def\AF{\text{Aff}^+(\R)}

\def\Att{\text{Att}}
\def\Rep{\text{Rép}}

\def\TF{T \FF} 
\def\TRF{T^\R \FF} 

\def\NF{N{\FF}}
\def\NRF{N^\R  \FF}

\def\NG{N{\GG}}

\def\WF {N_{\FF}W}

\def\gWF {g_{W,\FF}}
\def\hWF {h_{W,\FF}}
\def\hyp {h^H} 

\def\Fatou{F}
\def\Julia{J}

\def\B2{\mathcal{B}_2}
\def\K3{\mathcal{K}_3}

\title{Stabilité structurelle du feuilletage de Jouanolou de degré 2}

\author{Aurélien Alvarez et Bertrand Deroin}

\begin{document}

\begin{abstract}
Nous démontrons que le feuilletage de Jouanolou de degré~2 sur le plan projectif complexe est structurellement stable.
De plus, son ensemble de Fatou est une fibration holomorphe sur la quartique de Klein ayant une structure de fibré lisse localement trivial en disques. 
En particulier, aucune feuille de \(\FJ2\) n'est dense dans \(\Ptwo\).
\end{abstract}

\maketitle

\section{Introduction et énoncé du théorème principal} 

Soit \(\mathbf{F}_d\) l'espace des modules des feuilletages algébriques complexes de degré~\( d\) sur le plan projectif complexe \(\Ptwo\). Mieux comprendre la décomposition de la variété quasi-projective \(\mathbf{F}_d\) suivant les propriétés dynamiques et topologiques des feuilletages algébriques reste un problème largement ouvert dès que \(d\geq 2\). Le présent travail se propose d'y apporter une contribution nouvelle et, en un certain sens, inattendue.

Dans un travail célèbre, présenté lors de l'ICM 1978 et qui repose en partie sur les articles de Hudaj-Verenov \cite{HV} et Mjuller \cite{Mjuller}, Il'yashenko \cite{Ilyashenko} démontre que presque tout (vis-à-vis de la mesure de Lebesgue) feuilletage de \( \mathbf{F}_d\) qui contient une droite projective invariante est structurellement rigide, ergodique et minimal. Ce travail a inspiré de nombreux auteurs, parmi lesquels Scherbakov \cite{Scherbakov}, Cerveau \cite{Cerveau}, Ghys \cite{Ghys}, Nakai \cite{Nakai} et, plus récemment, Loray et Rebelo \cite{Loray-Rebelo}. Ces derniers montrent qu'il existe un ouvert non vide de \(\mathbf{F}_d \) formé de feuilletages structurellement rigides, minimaux, et ergodiques, en s'affranchissant de l'hypothèse sur l'existence d'une droite projective invariante. 

Le résultat que nous présentons ici s'oppose radicalement à tous ces travaux : nous exhibons une composante de stabilité non triviale dans \(\mathbf{F}_2\), c'est-à-dire un ouvert formé de feuilletages tous topologiquement conjugués les uns aux autres; de plus, nous montrons que les feuilletages appartenant à cette composante de stabilité n'ont aucune feuille dense et ne sont pas ergodiques. Plus précisément, on désigne par \(\Jd\) le champ de vecteurs de Jouanolou de degré~\(d\) défini dans les coordonnées cartésiennes \((x, y, z)\) de \(\C^3\) par 
\begin{equation} \label{eq: def Jouanolou} \Jd(x, y, z) = y^d \frac{\partial }{\partial x} +z^d \frac{\partial }{\partial y} +x^d \frac{\partial }{\partial z}. \end{equation}
Ce champ est \textit{homogène} de degré \(d\) et définit donc un feuilletage \( \FJd\) du plan projectif complexe. Jouanolou a montré dans \cite{Jouanolou} que \(\FJd\) n'a pas de feuille algébrique invariante lorsque \(d\geq 2\) et qu'il en est ainsi pour un feuilletage générique de~\(\mathbf{F}_d\).  

\begin{theoreme*}\label{t: stabilité structurelle de Jouanolou en degré 2}
Le feuilletage \(\FJ2\) du plan projectif complexe \(\Ptwo\) est structurellement stable. 
De plus, son ensemble de Fatou est une fibration\footnote{Nous appelons fibration un morphisme surjectif à fibres connexes entre deux variétés complexes.} sur la quartique de Klein ayant une structure de fibré lisse localement trivial en disques. 
En particulier, aucune feuille de \(\FJ2\) n'est dense dans \(\Ptwo\).
\end{theoreme*} 

Dire que le feuilletage \(\FJ2\) est structurellement stable signifie qu'il existe un voisinage \(\mathcal{V}\) de \(\FJ2\) dans l'espace des modules $\mathbf{F}_2$ tel que tout feuilletage dans \(\mathcal{V}\) est topologiquement conjugué à~\(\FJ2\). Le lieu de stabilité dans \(\mathbf{F}_2\) est par définition l'ensemble des feuilletages algébriques structurellement stables. Nous conjecturons que l'application quotient du domaine de Fatou d'un feuilletage de la composante de stabilité de \(\FJ2\) est un revêtement de l'espace des modules des courbes algébriques lisses de genre~3. 

\begin{figure}[ht]
	\centering
	\includegraphics[width=0.35\textwidth]{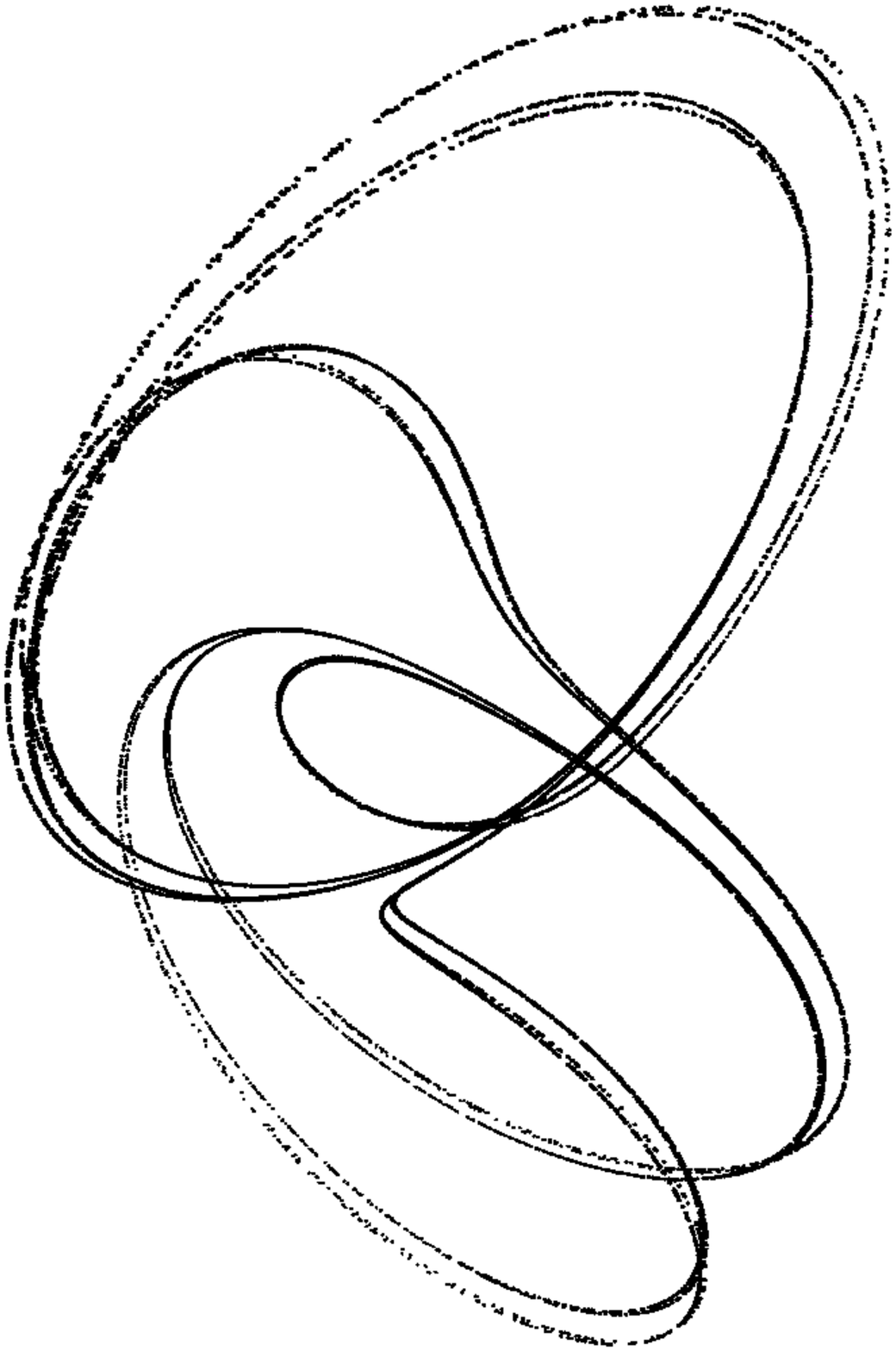}
	\caption{Intersection de l'ensemble de Julia du feuilletage de Jouanolou de degré~2 avec une sphère entourant une singularité}
	\label{img:d2}
\end{figure}

Concernant l'ensemble de Julia, défini comme le complémentaire de l'ensemble de Fatou, nous conjecturons qu'il est de mesure nulle\footnote{En utilisant les propriétés d'hyperbolicité que nous établissons dans cet article, il suffirait de montrer que l'ensemble de Julia est d'intérieur vide pour en déduire qu'il est de mesure nulle.} et transversalement un ensemble de Cantor (figure~\ref{img:d2}). Cette conjecture, étayée par des expérimentations numériques (\(\mathsection\)~\ref{sec:images-exp-num}), impliquerait que le feuilletage de Jouanolou est intégrable en dehors d'un ensemble fermé de mesure nulle.

Nous montrons également que les feuilles du feuilletage \(\mathcal J_2\) sont simplement connexes, sauf un nombre dénombrable d'entre elles qui sont des anneaux (\(\mathsection~\ref{ss: conjecture Anosov}\)).
En vertu du théorème principal ci-dessus, nous sommes donc en mesure de valider la conjecture d'Anosov (qui décrit ainsi la topologie des feuilles pour un feuilletage générique du plan projectif complexe) pour des feuilletages appartenant à un ouvert non vide de \(\mathbf{F} _2\).


\vskip 0.5cm

Nous remercions chaleureusement Serge Cantat, Dominique Cerveau, Yulij Ilyashenko, Étienne Ghys, Alexey Glutsyuk, Xavier Gomez-Mont, Adolfo Guillot, Samuel Lelièvre, Frank Loray, Jorge Pereira, Bruno Sévennec pour les nombreuses discussions que l'on a pu avoir à propos de ce travail.

\tableofcontents

\section{Stratégie pour démontrer le théorème de stabilité structurelle} 

Un feuilletage algébrique complexe \(\FF\) de degré \(d\) du plan projectif complexe est le projectivisé du feuilletage \( \FF_V\) de \( \C^3\) défini par un champ de vecteurs~\(V\) homogène de degré~\(d\) que l'on peut toujours supposer de divergence nulle (lemme~\ref{l: divergence nulle}). De plus, le lieu des points où~\( V\) et~\(R\) sont colinéaires est une union finie de droites vectorielles qui définissent l'ensemble singulier~\(\sing\) de \(\FF\). 

Dans la suite, on considère un feuilletage \(\FF\) défini par un champ homogène~\(V\) de degré~\(d\), de divergence nulle et qui ne s'annule pas sur \({\bf C}^3 \setminus \{0\}\).

Commen\c{c}ons par énoncer un critère qui assure l'existence d'un ensemble errant pour le feuilletage \(\FF\) (\(\mathsection\)~\ref{ss: domaine errant de F}). Sur chaque feuille de \(\FF_V\), considérons la restriction de la fonction \(- \log \norm{\cdot}\), où \(\norm{\cdot}\) est la norme hermitienne standard. Il s'agit d'une fonction strictement sur-harmonique, sauf le long des feuilles radiales de~\( \FF_V\) où elle est harmonique. Ses points critiques définissent l'ensemble algébrique réel~\(\tilde{B}\) donné par l'équation 
\[ \tilde{B}:= \{R\cdot V = 0\},\] 
où \(\cdot\) désigne le produit hermitien. Les points de \(\tilde{B}\) sont non dégénérés d'indice égaux à \(1\) ou \(2\) suivant que  \(| DV(V) \cdot R | >  || V ||^2\) ou \(| DV(V) \cdot R | < || V ||^2\) (lemme \ref{l: singularites longitudinales de \W}). Notre critère met en jeu la propriété suivante.

\vspace{0.2cm}

\noindent
\underline{Propriété \textbf{\(\Pr\)}} (déf.~\ref{propriete-PB}) : \textit{Le champ \(V\) ne s'annule pas sur \(\C ^3\setminus \{0\}\) et les points critiques de la restriction de la fonction \(-\log \norm{\cdot} \) le long des trajectoires de \(V\) dans \(\C ^3 \setminus \{0\} \) sont non dégénérés d'indice égaux à~2. }

\vspace{0.2cm}

Sous l'hypothèse \textbf{\(\Pr\)}, la variété algébrique réelle \(\tilde{B}\subset \C^3 \setminus \{0\} \) est lisse et transverse au feuilletage \(\FF_V\); par conséquent la surface algébrique réelle définie par l'équation 
\[ B:= \Pi( \tilde{B}  ) \]
est une section transverse\footnote{Une section transverse à un feuilletage est une sous-variété réelle transverse au feuilletage et de dimension égale à la codimension réelle du feuilletage.} au feuilletage \(\FF\). Une telle section transverse hérite d'une structure holomorphe induite par la structure holomorphe transverse du feuilletage \(\FF\) (\cite{Bogomolov}).

\begin{proposition}\label{p: fibration en disques}
Si \(V\) satisfait la propriété  \textbf{\(\Pr\)}, alors les trajectoires de \(V\) passant par les points de \( \tilde{B}\) sont des disques proprement plongés dans \(\C^3\) et l'union de ces trajectoires est un fibré lisse localement trivial en disques au-dessus de \(\tilde{B}\).
\end{proposition}

La démonstration de la première partie de cette proposition se trouve au paragraphe \ref{ss: domaine errant de F}.
L'étude de la décomposition Fatou/Julia est faite à la partie~\ref{ss: decomposition Fatou Julia} et celle de la topologie des sections transverses des feuilletages de~\(\Ptwo\) à la partie \ref{s: sections transverses}.

L'un des outils principaux est l'étude du gradient de la fonction \(-\log \norm{\cdot} \) le long des trajectoires du champ \(V\) vis-à-vis d'une métrique hermitienne sur les feuilles de \( \FF_V\) qui est invariante par multiplication par les scalaires (\(\mathsection\)~\ref{ss: metrique sur TF}) : le champ de vecteurs \(\tilde{W}\) ainsi construit sur \(\C^3\) induit un champ de vecteurs \(W\) sur \(\Ptwo\) appelé le \textit{champ réel}. Il s'agit d'un champ de vecteurs analytique sur \(\Ptwo\) (\(\mathsection\)~\ref{ss: definition de W}), qui est par définition tangent à \(\FF\) et qui s'annule sur l'union de l'ensemble~\(B\) et de l'ensemble singulier de \(\FF\). Observons que sous l'hypothèse \textbf{\(\Pr\)}, les points de \(B\) sont des puits pour \(W\) en restriction à chaque feuille de \(\FF\). Ainsi, \(B\) est un attracteur pour le champ réel \(W\). 

Ce champ joue également un rôle central pour établir la stabilité structurelle si  la propriété \textbf{\(\Pr\)} ainsi que la propriété de répulsion \textbf{\(\Qr\)} suivante sont satisfaites :

\vspace{0.2cm}

\noindent
\underline{Propriété \textbf{\(\Qr\)}} (déf.~\ref{propriete-PS}): \textit{Le champ \(V\) ne s'annule pas sur \(\C ^3\setminus \{0\}\), les singularités du feuilletage~\(\FF\) sont hyperboliques et chacune de ces singularités est une source pour le champ réel \(\W\).} 

\begin{proposition}
Un feuilletage algébrique de degré \(d\) de \(\Ptwo\) qui satisfait les propriétés \textbf{\(\Pr\)} et \textbf{\(\Qr\)} est structurellement stable. En d'autres termes, tout feuilletage algébrique \(\FF'\) de degré \(d\) qui est suffisamment proche de \(\FF\) est topologiquement conjugué à \(\FF\).
\end{proposition} 

On commence par démontrer que l'ensemble errant construit à la proposition~\ref{p: fibration en disques} est exactement l'ensemble de Fatou.
La stabilité structurelle de l'ensemble de Fatou découle alors de la proposition \ref{p: fibration en disques} (la propriété \textbf{\(\Pr\)} est stable). 
La stabilité structurelle de l'ensemble de Julia du feuilletage repose sur des propriétés d'hyperbolicité du champ \(W\). Ces propriétés sont établies dans la partie \ref{s: hyperbolicite} : on construit une métrique complète sur \(\Ptwostar =\Ptwo \setminus \left( B\cup \sing \right)  \) pour laquelle le feuilletage stable faible\footnote{Il s'agit d'un feuilletage réel de codimension~1 sur \(\Ptwostar\), dont la distribution tangente est continue, dont les feuilles sont \(C^\infty\), qui est le produit d'un feuilletage de Reeb par une droite au voisinage des points de~\(\sing\) et, localement, un livre ouvert au voisinage des points de \(B\).} de \(W\) est le feuilletage \(\FF\), alors que le feuilletage instable faible est un feuilletage de dimension \(3\) réel transverse à \(\FF\). Nous en déduisons la stabilité structurelle de \(W\) par un théorème de Robinson \cite{Robinson}, ainsi que la stabilité structurelle de l'ensemble de Julia du feuilletage \(\FF\) qui est aussi l'ensemble d'attraction de l'ensemble non errant de \(W\) dans \(\Ptwostar \). 

Il faut ensuite recoller les morceaux pour établir la stabilité structurelle globale. Pour cela, nous construisons une action localement libre du groupe affine sur \(\Ptwostar \) dont les orbites sont les feuilles de la restriction de \(\FF\) à \(\Ptwostar\) (\(\mathsection\)~\ref{ss: structure hyperbolique feuilletee}). La stabilité structurelle est établie à la partie \ref{s: stabilite structurelle}.

Pour conclure, nous montrons que le champ de Jouanolou en degré \(2\) satisfait les propriétés \textbf{\(\Pr\)} et \textbf{\(\Qr\)} puis, en utilisant les symétries du feuilletage, que la surface~\(B\) est biholomorphe à la quartique de Klein (proposition \ref{p: caracterisation de la quartique de Klein}). La propriété  \textbf{\(\Qr\)} est élémentaire et est établie au lemme \ref{l: les singularites de FJ2 sont des sources}. La propiété \textbf{\(\Pr\)} est plus délicate à établir : il s'agit de montrer que, étant donné trois nombres complexes quelconques \(x,y,z\in \C\) non tous nuls, on a l'implication suivante 
\begin{equation}\label{eq: implication} 
x\overline{y}^2 + y\overline{z}^2 + z\overline{x}^2=0 \ \Longrightarrow \ 2 |\bar{x}yz^2 + \bar{y}zx^2 + \bar{z}xy^2| < |x|^{4} + |y|^{4} + |z|^{4}.
\end{equation} 
À la partie \ref{s: numerique}, nous ramenons la démonstration de  \ref{eq: implication} à la vérification d'un nombre fini d'inégalités explicites sur des entiers que l'on a confiée à un ordinateur.


\section{Notations}

\begin{itemize}
\item \( \Ptwo\) plan projectif complexe
\item \(\Pi : \C^3\setminus \{0\} \rightarrow \Ptwo\) application quotient
\item\((x,y,z)\cdot (x',y',z')= x\overline{x'} +y \overline{y'} + z\overline{z'}\) produit hermitien standard sur \(\C^{3}\)
\item \( V\) champ de vecteurs homogène sur  \(\C^{3}\)
\item \(\J2\) champ de Jouanolou sur \(\C ^3\) (éq. \ref{eq: def Jouanolou})
\item \( \FV\) feuilletage induit par \( V\)  
\item \(\FF\) quotient de \(\FV\)   sur \(\Ptwo\) 
\item \( \FJ2\)  feuilletage de \(\Ptwo\) induit par \(\J2\)
\item \( \TF\) fibré tangent holomorphe à \(\mathcal F\)
\item \( \TRF\) fibré tangent réel à \(\FF\) (défini sur la partie régulière \(\reg (\FF)\) de \(\FF\))   
\item \( \NF\) fibré normal holomorphe à \(\mathcal F\)
\item \( \NRF\) fibré normal réel à \(\FF\) (défini sur \(\reg (\FF)\))   
\item \(g\) métrique hermitienne sur \( T\mathcal F\) (\(\mathsection\)  \ref{ss: metrique sur TF})
\item \( \tilde{g}\) métrique relevée sur \( T\FV\) (éq. \ref{eq: metrique relevee})
\item \( \tilde{\W}\) gradient le long des trajectoires de \( \FV\) de la fonction \( -\log ||\cdot || \)
\item \( \W\) champ de vecteurs de \(\Ptwo\) obtenu comme projection de \(\tilde{\W}\), appelé champ réel
\item  \(\{ \Phi_{\W}^t\}_{t\in \R}\) flot induit par \(\W\)
\item \(\WF:= \TRF / \R \W\)
\item \( B\) section transverse (éq. \ref{eq: section transverse})
\item \( U_B\) voisinage tubulaire de \( B\) 
\item \(\sing \) ensemble singulier de \(\FF\)
\item \( U_S\) voisinage tubulaire de \(\sing\)
\item \(\Ptwostar = \Ptwo \setminus \left(B\cup \sing\right) \) 
\item \(\Pr\) : voir déf.~\ref{propriete-PB}
\item \( \Qr\) : voir déf.~\ref{propriete-PS}
\item \(\FF_{A}\) : restriction de \(\FF\) à la sous-variété \(A\), i.e. défini par la distribution \( TA\cap T^{\R}\FF\)
\item \( \mathcal R\) feuilletage de \(\partial U_S \) transverse à \(\FF_{\partial U_S}\) 
\item \(h\) métrique riemannienne sur \( \Ptwostar \) adaptée à \(W\)
\item \( \FSFo _W, \FSFa _W , \FIFo  _W, \FIFa  _W \) feuilletages stable fort, stable faible, instable fort, instable faible du champ \(W\) (\(\mathsection\)~\ref{ss: feuilletages stables et instables})
\item \( D \) saturé de \(B\) par le feuilletage \(\FF\) appelé ensemble errant de \(\FF\) (\(\mathsection\)  \ref{ss: domaine errant de F})
\item \(\Fatou(\FF), \Julia(\FF)\) ensembles de Fatou et de Julia de \(\FF\) (\(\mathsection\) \ref{ss: decomposition Fatou Julia})
\item \(\text{Rép} _W (\square) = \{p\ |\ \lim _{t\rightarrow -\infty} d(\Phi_{\W}^t (p), \square ) =0\} \)~: ensemble de répulsion de \(\square\) pour le champ \(W\) 
\item \( \text{Att}_W (\square) = \{p\ |\ \lim _{t\rightarrow +\infty} d(\Phi_{\W}^t (p), \square ) =0\} \)~: ensemble d'attraction de \(\square\) pour le champ \(W\)
\item \(K= \Ptwo \setminus \left( \Rep_\W (\sing) \cup \Att_\W (B) \right) \): ensemble hyperbolique maximal (\(\mathsection\)~\ref{eq: ensemble hyperbolique})
\item \( \W_R\) reparamétrage du champ \(\W\) (éq. \ref{eq: reparametrisation de W})
\item \(\hyp\) métrique hyperbolique sur \(\TRF _{|\Ptwostar} \) (éq. \ref{eq: metrique hyperbolique})
\item \(\pi : \AF\times \Ptwostar \rightarrow \Ptwostar\) action du groupe affine (éq \ref{eq: def pi})
\item \(\pi_p=\pi (\cdot , p ) \) paramétrage des feuilles par \( \AF\) (éq \ref{eq: coordonnees R^2})
\item \( V'\) perturbation de \(V\), \( \FF'\) feuilletage induit par \(V'\), \(\W'\) champ réel induit par \(V'\) (\(\mathsection\)~\ref{ss: stabilite Julia})
\item \( \psi : \Ptwostar \rightarrow \Ptwostarprime \) conjugaison topologique orbitale de \(\W\) à \(\W'\) (prop.\ref{p: stabilite structurelle W})
\item \( \Psi : \Ptwostar\rightarrow \Ptwostarprime \) modification de \(\psi\)
\item \(\beta\subset F\) bande (i.e. composante connexe de \( F\setminus \text{Att}_{\W} (K)\)) (\(\mathsection\) \ref{ss: bandes})
\item \(l: B\rightarrow B'\) homéomorphisme proche de l'identité (éq. \ref{eq: homeomorphisme entre B et B'})
\end{itemize}

\section{Le champ réel et les propriétés \textbf{\(\Pr\)} et \textbf{\(\Qr\)}} 

\subsection{Feuilletages de \(\Ptwo\)} 

Le plan projectif complexe \(\Ptwo\) admet un fibré tangent holomorphe \(T\Ptwo\) dont les sections locales sont les champs de vecteurs holomorphes locaux sur \( \Ptwo\). Il admet également un fibré tangent réel \(T^\R \Ptwo\) dont les sections sont les champs de vecteurs réels.  Un champ de vecteurs holomorphe s'étend naturellement en une dérivation agissant sur les fonctions lisses à valeurs complexes.
La partie réelle de cette extension définit une dérivation réelle, c'est-à-dire un champ de vecteurs réel. Le flot induit par la partie réelle d'un champ de vecteurs holomorphe s'obtient par restriction du flot induit par ce dernier aux temps réels, modulo le facteur multiplicatif~\(1/2\).
On a un isomorphisme réel naturel entre le fibré tangent holomorphe et le fibré tangent réel, induit par l'application qui à un champ de vecteurs holomorphe associe sa partie réelle.

Un feuilletage algébrique complexe de \(\Ptwo\) est la donnée d'un morphisme \(m: \TF\rightarrow T\Ptwo\) d'un fibré en droites holomorphe \(\TF\) au-dessus de \( \Ptwo\) qui s'annule au-dessus d'un ensemble fini de \(\Ptwo\) (\cite{Brunella}). Le fibré \(\TF\) s'appelle le fibré tangent holomorphe de \(\FF\), le lieu où \(m\) s'annule l'ensemble singulier \(\sing\) de \(\FF\) et son complémentaire la partie régulière \(\reg (\FF)\).  Par définition, le degré \(d\) de \(\FF\) est le nombre de tangence de \(m (\TF)\) avec une droite générique de \(\Ptwo\) et on a alors \(\TF \simeq \mathcal{O}(1-d)\).

Nous dirons qu'un champ de vecteurs holomorphe défini sur un ouvert de~\(\Ptwo\) définit \(\FF\) s'il est l'image par \(m\) d'une section de \(\TF\) qui ne s'annule en aucun point. Dans la partie régulière du feuilletage, les extensions analytiques maximales des germes de courbes intégrales d'un champ de vecteurs holomorphe définissant~\(\FF\) forment des courbes holomorphes immergées dans \(\Ptwo\) appelées les feuilles de ~\(\FF\). 

Dans la partie régulière de \(\FF\), les feuilles sont tangentes à la distribution \( \TRF\subset T^\R {\Ptwo}_{|\reg (\FF)}\) définie comme l'image par l'identification naturelle \( T\Ptwo \rightarrow T^\R \Ptwo\) du sous-fibré \( m((\TF )_{| \reg (\FF)} )\subset T {\Ptwo}_{|\reg (\FF)}\) ; observons que \(\TRF\) est naturellement isomorphe à la restriction du fibré tangent holomorphe \(\TF\) à l'ensemble régulier de \(\FF\). 

De même, on définit  le fibré normal \(\NF\) de \(\FF\) comme étant le dual du fibré associé au faisceau des \(1\)-formes holomorphes sur \(\Ptwo\) contenant \(m(\TF)\) dans leur noyau (ce faisceau est localement libre et donc bien associé à un fibré en droites holomorphe \cite{Brunella}). Dans la partie régulière de \(\FF\), ce fibré en droites s'identifie à \( T\Ptwo / m(T\FF)\) et donc au fibré normal réel \(\NRF:= T^\R \Ptwo / T^\R \FF\) via l'identification naturelle  \(T\Ptwo \rightarrow T^\R \Ptwo \).

\subsection{Projectivisation d'un champ homogène et métrique sur \( \TF\)}\label{ss: metrique sur TF}
 
\begin{lemme}\label{l: divergence nulle} Étant donné un feuilletage~\(\FF\) sur \(\Ptwo\) de degré \(d\), il existe un champ de vecteurs \(V\) holomorphe homogène de degré~\(d\) sur \(\C^{3}\), de divergence nulle et tel que la projectivisation du feuilletage \(\FV\) induite par $V$ sur \(\C^{3} \setminus \{0\}\) est le feuilletage \(\FF\). Ce champ est unique modulo multiplication par une constante non nulle et  il est colinéaire au champ radial seulement dans un nombre fini de directions de \(\C^3\) dont les projectivisations sont les singularités de \(\FF\).\end{lemme}

\begin{proof}
L'existence d'un champ de vecteurs \(V\) homogène de degré \(d\), radial uniquement au-dessus du lieu singulier de \(\FF\) et tel que la projectivisation de \(\FF_V\) est égale à \(\FF\), se trouve dans \cite{Lins Neto}. Pour assurer que l'on peut trouver un tel champ \(V\) qui soit de plus à divergence nulle, il suffit de considérer le champ 
\[ V- \frac{\text{div} (V) }{d+2}R ,\]
où \(R= x\frac{\partial}{\partial x} + y \frac{\partial}{\partial y} + z \frac{\partial}{\partial z} \) est le champ radial. L'unicité est évidente.
\end{proof}

On rappelle que le fibré tautologique \(\mathcal{O}(-1)\) s'identifie à \(\C^3 \setminus \{0\}\) en dehors de sa section nulle. Le fibré tangent \(\TF\) de \( \FF\)  s'identifie alors à la puissance \((d-1)\)-ième du fibré tautologique via l'application homogène de degré \((d-1)\)
\begin{equation}\label{eq: isomorphisme TF} p\in \C ^ 3\setminus \{0\} \mapsto D_p\Pi (V(p)) \in T_{\Pi(p)} \mathcal F \end{equation}
définie en dehors du lieu singulier de \(\FF\). Dans cette formule, \( \Pi \) désigne l'application quotient \(\C^{3} \setminus \{0\} \rightarrow \Ptwo\).   

On notera \( \tilde{g} \) la métrique\footnote{Cette métrique est singulière aux points où \(V\) s'annule.} sur \( T\FV\) définie par 
\begin{equation}\label{eq: metrique relevee}  \tilde{g}_p ( V(p) ) = \norm{p}^{2d-2}. \end{equation} 
Cette métrique est invariante par multiplication par les scalaires sur \(\C^3\setminus \{0\}\) et définit donc une métrique hermitienne sur \(\TF\) que l'on note \(g\).


\begin{lemme}
On suppose que \(V\) ne s'annule pas sur \(\C ^3 \setminus \{0\}\) et que \(d \geq 2\). Munies de la métrique hermitienne \(\tilde{g}\), les feuilles de \(\FV\) sont des surfaces complètes à courbure strictement négatives, sauf les feuilles radiales qui sont isométriques à des cylindres bi-infinis euclidiens \( \R^2 / l \Z\) avec \(l>0\). \end{lemme}

\begin{proof}[Démonstration]
  Si \(g\) est une métrique hermitienne sur une surface de Riemann, sa courbure s'exprime par \( -\Delta_g \log \norm{V}_g\) où \(V\) est un champ de vecteurs holomorphe local qui ne s'annule pas. La formule \eqref{eq: metrique relevee}  montre donc que la courbure de la métrique \( \tilde{g} \) le long des feuilles de \(\FF_V\) est donnée par l'expression \( -\Delta_{\tilde{g}} \log \norm{p}^{d-1}  \). Or la fonction \( \log \norm{p}^{d-1}\) est pluri-sous-harmonique, et strictement dans les directions autres que radiales, ce qui montre que la courbure de la restriction de \(\tilde{g}\) aux feuilles de \(\FF_V\) est strictement négative.   

Pour montrer la complétude, il suffit de constater que, si l'on introduit la métrique hermitienne standard \(g_0\) sur \(\C^{3}\), on a l'inégalité \( c||p||^{2d}\leq   g_0 (V)_p \leq d ||p||^{2d} \) valable pour certaines constantes \(c,d>0\) indépendantes de \(p\), et par conséquent on obtient 
\[       c'  \frac{g_0} {\norm{\cdot}^2} \leq \tilde{g} \leq  d' \frac{g_0} {\norm{\cdot}^2}     \]
avec \(c'= 1/d\) et \(d'= 1/c\). La complétude de \(g\) en restriction aux feuilles de \(\FV\) découle de celle de la métrique \(\frac{g_0} {\norm{\cdot}^2} \) sur \( \C^{3}\setminus \{0\} \).

La dernière assertion du lemme vient de ce que les feuilles radiales de \(\FV\) sont topologiquement des cylindres et que toutes les métriques plates complètes sur de telles surfaces sont isométriques à des cylindres euclidiens bi-infinis.   
\end{proof}


\subsection{Définition du champ réel}\label{ss: definition de W}
Introduisons la fonction \(f: \C^{3}\setminus\{0\} \rightarrow \R\) définie par \( f(p) := - \log ||p||^{2} \). Son gradient le long des feuilles de \( \FV\), vis-à-vis de la métrique hermitienne \(\tilde{g}\), est un champ de vecteurs \(\tilde{\W}\) sur \(\C^{3} \setminus \{0\}\). Pour tout scalaire non nul \(\lambda\in \C^*\), on a \( f(\lambda \, \cdot) = - \log |\lambda |^2 + f(\cdot )\), ce qui montre que \( df \) est invariante par multiplication par les scalaires et, par conséquent, qu'il en est de même pour \(\tilde{\W}\). Il existe donc un champ de vecteurs analytique \( \W\) sur \(\Ptwo \) tel que, pour tout \(p\in \C^{3}\setminus\{0\} \), on a \(D\Pi _p ( \tilde{\W}(p)) = \W([p])\): ce champ est appelé le \textit{champ réel}.

\begin{lemme} \label{l: expression de Wtilde} On a  \( \tilde{\W} = \Re \left( \tilde{\rho} V \right)\),  où \(\tilde{\rho} (p)=  -2 \frac{(p\cdot V(p) )}{\norm{p}^{2d}}\) pour tout \(p \in \C^{3} \setminus \{0\}\).  
\end{lemme} 

\begin{proof}[Démonstration]
Introduisons la métrique hermitienne \( g_1\) sur \(T \FV \) telle que \( g_1 (V) =1\). Autrement dit, si l'on paramètre la feuille passant par le point \(p_0\) de \( \C^{3} \setminus \{0\} \) par la courbe intégrale \( t\mapsto p(t)\) de l'équation \( \dot{p} = V(p)\) passant par \( p (0) = p_0\), alors la métrique \(g_1\) est la métrique hermitienne standard \(  |dt|^2 \). Comme \( \tilde{g}(V) = h(p) \), on en déduit \( \tilde{g} = h g_1\) puis la formule 
\begin{equation}\label{eq: comparaison gradient} \tilde{\W} = \nabla _{\tilde{g}}  f = h^{-1}  \nabla_{g_1} f .  \end{equation}
Or dans la coordonnée~\(t \) décrite plus haut, on a 
\[ d \log \norm{p}^{2} = 2 \frac{ \Re \left( (p\cdot  V(p) ) \overline{dt} \right) }{\norm{p}^{2}} , \]
ce qui montre que
\[\nabla _{ g_1} f =- \Re \left( \frac{2 (p\cdot V(p) ) }{ \norm{p}^2} V \right). \] 
Le lemme découle de \eqref{eq: comparaison gradient} et du  fait que \( h = \norm{p}^{2d-2} \). 
 \end{proof} 

\begin{corollaire}
Le champ \(\W\) s'annule sur l'union de  l'ensemble singulier~\(\sing\) de~\(\mathcal F\) et du lieu défini par  
\begin{equation}\label{eq: section transverse} B :=\Pi \big(  \{  R\cdot V = 0\} \big).\end{equation}
\end{corollaire} 

\begin{proof}[Démonstration]
D'après le lemme \ref{l: expression de Wtilde}, \( \{  R\cdot V = 0\} \cup \Pi^{-1} (\sing)\) est le lieu où \(\tilde{\W}\) est tangent à la distribution radiale complexe. 
\end{proof}

Notons que \(B\) est en fait une courbe mixte de bidegré \((1,d)\) au sens de Oka~\cite{Oka}. En particulier, \(B\) est non vide (sauf éventuellement en degré~1 pour certains champs de vecteurs non génériques). 

\subsection{Les singularités de \(\W\) le long des feuilles et la propriété \textbf{\(\Pr\)}}  \label{ss: Pr}
Le résultat suivant donne des informations sur la nature des singularités du champ \(\W\) le long des feuilles, c'est-à-dire en chaque point de l'ensemble \(B\) défini par \eqref{eq: section transverse}. On rappelle qu'une singularité d'un champ de vecteurs sur une variété est une source si toutes les valeurs propres complexes du champ en cette singularité ont une partie réelle strictement positive, une selle si les parties réelles des valeurs propres sont non nulles et certaines de signes opposés, et un puits si toutes les valeurs propres sont de parties réelles strictement négatives. 

\begin{lemme}\label{l: singularites longitudinales de \W}
En restriction à une feuille de \(\mathcal F\), une singularité de \(\W\) 
\begin{enumerate}
\item est un puits  ssi  \(\norm{V}^2 > | DV(V) \cdot R |\);
\item est une selle  ssi \(\norm{V}^2 < | DV(V) \cdot R |\);
\item n'est jamais une source. 
\end{enumerate}
\end{lemme}

Dans ce lemme, on  rappelle que $p \cdot p'= x\overline{x'}+y\overline{y'}+ z\overline{z'}$ est le produit hermitien standard sur $\mathbb C^3$, $R= x\frac{\partial}{\partial x} +y\frac{\partial}{\partial y} + z\frac{\partial}{\partial z}$ le champ radial,
 $D$ la connexion standard sur \( T\C ^3\). Observons alors que si l'on note \( V=\sum_{k\in\{x,y,z\}} V_k \frac{\partial}{\partial k}\),
 $DV(V)$ est le champ de vecteurs défini par  
 $$DV(V)_k = \sum_{l\in\{x,y,z\}} V_l \frac{\partial V_k}{\partial l } \text{ pour } k\in\{x,y,z\}. $$ 
 
\begin{proof} [Démonstration.] 
Le développement de Taylor d'une fonction lisse $\varphi: \C \rightarrow \R$ au voisinage d'un point critique $t_0$ est donné par
\[  \varphi(t) = \varphi(t_0) +\frac{1}{2} \frac{\partial ^2 \varphi }{\partial t^2}\, (t-t_0)^2 +\frac{1}{2} \frac{\partial ^2 \varphi }{\partial \overline{t}^2} \, \overline{(t-t_0)}^2 +  \frac{\partial ^2 \varphi}{\partial t \partial \overline{t} }\, |t-t_0|^2  + o (|t-t_0|^2), \]
ce qui montre que 
\[  D^2 \varphi (t) = \Re \left(\frac{\partial ^2 \varphi }{\partial t^2} \, (t-t_0)^2\right) +  \frac{\partial ^2 \varphi}{\partial t \partial \overline{t} }\, |t-t_0|^2.   \]
Cette forme quadratique est

\begin{itemize} 
	\item dégénérée ssi $\frac{\partial ^2 \varphi}{\partial t \partial \overline{t} } = |\frac{\partial ^2 \varphi }{\partial t^2}|$;
	\item de signature $(0,2)$ ssi $\frac{\partial ^2 \varphi}{\partial t \partial \overline{t} } <- |\frac{\partial ^2 \varphi }{\partial t^2}|$;
	\item de signature $(1,1)$ ssi $|\frac{\partial ^2 \varphi}{\partial t \partial \overline{t} } | < |\frac{\partial ^2 \varphi }{\partial t^2}|$;
	\item de signature $(2,0)$ ssi $\frac{\partial ^2 \varphi}{\partial t \partial \overline{t} } > |\frac{\partial ^2 \varphi }{\partial t^2}|$.
\end{itemize}

On applique alors ce qui précède à la fonction~\(f= - \log ||\cdot ||^{2} \) en restriction à \(\mathcal F_p\), où \(\Pi(p)\) est la singularité de~\(\W\).
En effet, il suffit d'établir les mêmes propriétés pour le champ \(\tilde{\W}\) au point~\(p\) puisque ce dernier est le gradient de la fonction~\(f\).
On obtient alors que  \(p\) est un point critique
\begin{itemize}
	\item dégénéré ssi \( \norm{ V }^2 = | DV(V) \cdot R | \);
	\item non dégénéré d'indice \(2\) ssi \( \norm{ V }^2 > | DV(V) \cdot R | \);
	\item non dégénéré d'indice \(1\) ssi \( \norm{V}^2 < | DV(V) \cdot R | \);
	\item n'est jamais non dégénéré d'indice \(0\),
\end{itemize}
et on conclut la preuve du lemme en utilisant le fait que la fonction \( f\) est pluri-sur-harmonique.  \end{proof}


\begin{definition}[Propriété \textbf{\(\Pr\)}]\label{propriete-PB} Un feuilletage du plan projectif complexe vérifie la propriété \textbf{\(\Pr\)} si le champ \(V\) ne s'annule pas sur \(\C ^{3}\setminus \{0\}\) et si les singularités du champ \(\W\) le long des feuilles de \(\FF\) sont des puits. \end{definition}

\begin{lemme}\label{l: construction de UB}
Supposons que \(\FF\) vérifie la propriété \textbf{\(\Pr\)}. Si~\(B\) est non vide, alors~\(B\) est une section transverse à \(\FF\) qui admet un voisinage tubulaire \( U_B\) tel que, d'une part le feuilletage en restriction à \(U_B\) est une fibré lisse localement trivial en disques et, d'autre part, il existe un difféomorphisme \(  U_B \setminus B \rightarrow \R^{\geq 0} \times \partial U_B\) tel que le champ \(\W\) est égal au champ \(\frac{\partial}{\partial t}\) dans les coordonnées \( (t,q)\in \R^{\geq 0} \times \partial U_B\).
\end{lemme}

\begin{proof}
Sur \(\C^3\setminus \{0\}\), les points critiques de la fonction $f=-\log ||.||$ en restriction à chaque feuille de \( \FF_V\) sont non dégénérés~: $\Pi ^{-1} (B)$ étant le lieu des points où la différentielle de \( f\) le long des feuilles de \(\FF_V\) est nulle, il s'agit d'une section transverse du feuilletage $\FF_V$. En effet, l'application \(\C^3\setminus \{0\}\rightarrow  {\TRF _V}^*\) qui à un point  associe la restriction de la différentielle de \(f\) au fibré tangent \(\TRF _V\) (dans une trivialisation locale de \( {\TRF_V}^*\)) est un difféomorphisme local le long des feuilles au voisinage d'un point critique de \(f\) le long de \(\FF_V\) si et seulement si le point critique est non dégénéré le long de sa feuille. La projection \(B\) de $\Pi ^{-1} (B)$ dans \(\Ptwo\) est donc également une section transverse de \(\FF\). 

On construit \( U_B\) en considérant l'application \(\text{exp}_\FF : (\TRF)_{|B}  \rightarrow \Ptwo\)  qui associe à un vecteur \( v\in T_b ^ \R \FF \)  l'extrémité \(\text{exp} _{\FF}(v) := \gamma(1)\) de la géodésique \(\gamma : [0,1] \rightarrow (\FF(b ),g) \) partant de \(\gamma(0)=b\) dans la direction \(\frac{d\gamma}{dt} (0)=v  \).  Le théorème des fonctions implicites montre que si \(\varepsilon>0 \) est suffisamment petit, la restriction de \(\text{exp} _{\FF} \) à \( \{v\in (\TRF)_{|B}\ |\ \norm{v} \leq \varepsilon \} \) est un difféomorphisme sur son image. De plus, par hypothèse, les valeurs propres de \(\W\) en restriction à \( T_b ^\R \FF\) sont toutes les deux strictement négatives en tout point \(b\in B\), donc par compacité de \(B\), si l'on choisit \(\varepsilon>0\) suffisamment petit, toute trajectoire de \( \W\) partant d'un point de \( U_B \setminus B\) converge vers un point de \(B\) dans le futur et intersecte \(\partial U_B\) en un unique point. En choisissant un tel \(\varepsilon>0\), on conclut la démonstration du lemme en associant à un point \(r\) de \( U_B\setminus B\) l'unique couple \((t,q)\in \R^{\geq 0} \times \partial U_B\) où \(q\) est l'intersection de la trajectoire de \(\W\) passant par \( r\) avec \(\partial U_B\) et où \(t\) est l'unique réel positif tel que \( \Phi _W ^t (q)=r\).      
\end{proof}

\subsection{Étude de \(\W\) au voisinage d'une singularité de \(\FF\) et propriété~\textbf{\(\Qr\)}}

\begin{lemme} \label{l: partie lineaire W} 
En une singularité de \(\FF\) en laquelle \(V\) ne s'annule pas, le champ de vecteurs \(\W\) est égal à la partie réelle d'un champ de vecteurs holomorphe \(Y\) définissant \(\FF\) à l'ordre un, c'est-à-dire que 
\begin{equation} \label{eq: Y} W= \Re (Y ) +\text{termes d'ordre supérieur à } 2.\end{equation} 
\end{lemme}

\begin{proof}[Démonstration]
Soit  \( s\in \Ptwo\) une singularité de \(\FF \). Quitte à permuter les coordonnées, on peut supposer que \(s\) appartient à la carte affine \(\{ z \neq 0\}\), isomorphe à~\( \C ^2\) via l'isomorphisme \( (x,y,z) \mapsto  (u= x/ z , v= y/z)\). Le feuilletage \(\mathcal F \) est défini par le champ de vecteurs \(X = X_u \frac{\partial}{\partial u} + X_v \frac{\partial}{\partial v} \), où
\begin{equation} \label{eq: feuilletage en coordonnees affines}   X=  (V_x(u,v,1) -u V_z(u,v,1) )\frac{\partial} {\partial u}+ (V_y (u,v,1)-vV_z(u,v,1))\frac{\partial} {\partial v}  \end{equation}
Dans les coordonnées \( (u, v)\), on a  également \( \W= \Re \left(\W_u \frac{\partial}{\partial u} + \W_v \frac{\partial}{\partial v} \right)\), où 
\[ \W_u (u, v ) = \tilde{\W}_x (u, v, 1) - u \tilde{\W}_{z} (u, v ,1),\ \  \W_v (u, v ) = \tilde{\W}_y (u, v, 1) - v \tilde{\W}_{z} (u, v ,1)\]
et où les fonctions \(\tilde{\W}_k\) pour \(k\in \{x,y,z\}\) sont définies par l'équation 
\[ \tilde{\W} = \Re \left( \tilde{\W}_x \frac{\partial}{\partial x} +\tilde{\W}_y\frac{\partial}{\partial y}+\tilde{W}_z \frac{\partial}{\partial z}\right).\] 
Or le lemme \ref{l: expression de Wtilde} nous donne l'expression  \( \tilde{\W} _k= \tilde{\rho} V_k \),  où \(\tilde{\rho} (p)=  -2 \frac{(p\cdot V(p) )}{||p||^{2d}}\).  En vertu de \eqref{eq: feuilletage en coordonnees affines}, on en déduit l'expression \( \W= \Re (\rho X) \) où \( \rho (u,v) = \tilde{\rho} (u,v,1)\) et le résultat en découle en posant \(Y= \rho (s) X\) puisque par hypothèse \(\rho (s) \neq 0\). \end{proof}

Nous adopterons la définition suivante : une singularité \(s\) de \(\FF\) est hyperbolique si les valeurs propres d'un champ définissant \(\FF\) au voisinage de \(s\) ne sont pas \(\R\)-colinéaires.  

\begin{definition}[Propriété \textbf{\(\Qr\)}]\label{propriete-PS} Un feuilletage \(\FF\) du plan projectif complexe vérifie la propriété \textbf{\(\Qr\)} si \(V\) ne s'annule pas sur \(\C^3\setminus \{0\}\) et si chaque singularité de \(\FF\) est d'une part hyperbolique et d'autre part une source pour le champ \(\W\).\end{definition}

Si les valeurs propres d'un champ de vecteurs holomorphe \(Y\) en une singularité \(s\) sont \( \lambda\) et \(\mu\), alors celles de sa partie réelle \( W= \Re (Y)\) sont \( \lambda/2, \mu / 2, \overline{\lambda} /2 , \overline{\mu} /2\).  Ainsi,  \(s\) est une source pour le champ \(W\) si et seulement si \( \Re (\lambda)\) et \(\Re (\mu) \) sont strictement positifs.


\begin{lemme} \label{l: construction de US}
Soit \(\FF\) un feuilletage du plan projectif qui vérifie la propriété~\textbf{\(\Qr\)}. Alors pour tout \(s\in \sing\), il existe un voisinage \(U_s\) de \(s\) à bord lisse et un difféomorphisme \( U _s\setminus \{s\} \rightarrow \R^{\leq 0} \times \partial U_s\) qui envoie le champ \(W\) sur le champ horizontal \(\frac{\partial}{\partial t} \) dans les coordonnées \( (t, Q)\in \R^{\leq 0} \times \partial U_s\). 

Le feuilletage \(\FF\) est transverse à \(\partial U_s\) ; son intersection avec \(\partial U_s\) définit donc un feuilletage transversalement holomorphe \(\FF_{\partial U_s}\) par courbes réelles de \(\partial U_s\). Ce dernier admet deux feuilles circulaires et, dans leur complémentaire, il est difféomorphe à un produit de  la courbe elliptique \( E_s= \C / (\Z \lambda+\Z\mu) \) par un intervalle.
De plus, il existe un feuilletage lisse par surfaces \(\mathcal R\) sur \( \partial U_s \) qui est transverse à \(\FF _{\partial U_s}\).
\end{lemme}

\begin{proof} Soit \(Y\) le champ construit au lemme \ref{l: partie lineaire W}. Le théorème de linéarisation de Poincaré montre qu'il existe des coordonnées \((u,v)\) centrées en \(s\) telles que 
\[Y =\lambda u \frac{\partial}{\partial u} + \mu v \frac{\partial}{\partial v}.\]
où \(\lambda\) et \(\mu\) sont les valeurs propres de \(Y\). D'après l'hypothèse \textbf{\(\Qr\)}, ces dernières étant de partie réelle strictement positive,   il existe \(r>0\) tel que en notant \( h = |u|^2+|v|^2\), on a \( d h (W(q) ) >0\) pour tout \(q\neq s\) dans la boule \( U_ s := \{|u|^2+|v|^2\leq r^2\}\). En particulier, toute trajectoire du flot induit par \(\W\) issue d'un point \(q\in U_s \setminus \{s\}\) tend vers \(s\) lorsque le temps tend vers \(-\infty\), et aboutit à un point \(Q(q)\) de \(\partial U_s\) en un certain temps \(t(q)\geq 0\). L'application \( q\in U_s \mapsto ]-t(q), Q(q)[ \in \R^{\leq 0}\times \partial U_s\) est un difféomorphisme qui envoie le champ \(\W\) sur le champ \(\frac{\partial}{\partial t}\).

Les deux feuilles circulaires sont l'intersection des séparatrices avec la sphère \(\partial U_s\). En dehors de ces dernières, on a une intégrale première 
\begin{equation} \label{eq: integrale premiere singularite} \frac{1}{2i\pi} (\lambda \log v - \mu \log u )   \end{equation} 
à valeurs dans \( E\) dont les fibres intersectent \( \partial U_s\) sur des intervalles. C'est cette intégrale première qui confère à la restriction de \(\FF _{\partial U_s}\) la structure de fibré au-dessus de \(E\) par intervalles.  

Construisons à présent le feuilletage \(\mathcal R\). Considérons la fonction \( \rho \) égale à \( \rho= | u | ^2 \) sur \(\partial U_s\). Les niveaux \( \rho = 0 \) et \( \rho = r ^2\) sont les deux feuilles circulaires. Par contre, tous les autres niveaux \( \rho ^{ -1}(\rho_0 )\) pour \( \rho_0\in ]0, r^2[ \)  sont des sections transverses toriques à \(\FF_{\partial U_s}\). Nous allons définir le feuilletage \( \mathcal R \) via \(\rho \) sur \( \rho ^{-1} ( [ \varepsilon, r^2 -\varepsilon ]) \) pour \( \varepsilon >0 \) suffisamment petit et modifier ce fibré dans les deux tores pleins \( \rho ^{-1} ([0, \varepsilon]) \) et \(  \rho ^{-1} ([r^2- \varepsilon , r^2 ]) \) par la méthode du tourbillonnement de Reeb. 

Expliquons cette construction dans le tore plein \( T= \rho ^{-1} ([0, \varepsilon]) \), dans l'autre tore la construction étant similaire. Orientons \( \FF _{\partial U_s}\) en le voyant comme le bord du feuilletage par surfaces de Riemann \( \FF\) sur la boule \(U_s\). Le long des feuilles de \( \FF_{\partial U_s}\) ainsi orientées, la fonction \( \rho \) est croissante si \( \Im  ( \frac {\mu} {\lambda} ) >0\) et  décroissante si \( \Im (\frac{\mu} {\lambda}) < 0\). Introduisons le fibré en disques \(\theta :  T \rightarrow \R / 2\pi \Z \) définie par \( \theta (u,v) =\text{Arg} (v) \), fonction croissante sur les feuilles de \( \FF_{\partial U_s}\). Étant donné une fonction lisse \( \psi :[0,r^2]\rightarrow \R^{\geq 0} \) qui vaut \( 1\) sur un voisinage de l'origine et qui s'annule sur \([\varepsilon , r^2]\), le feuilletage défini par \[ \psi (\rho) d\theta -\text{Im}\left(\frac{\lambda}{\mu}\right) (1-\psi (\rho)) dt =0 \] est alors transverse au flot et coïncide avec le fibré donné par \( \rho\) sur un voisinage de \( \rho^{-1}(\varepsilon)\). \end{proof}

\section{Construction d'un ensemble errant} \label{ss: domaine errant de F} 

Le but de cette partie est d'établir l'existence d'un ensemble errant (\cite{GGMS} et \(\mathsection~\ref{ss: decomposition Fatou Julia}\)) pour les feuilletages vérifiant la propriété~\textbf{\(\Pr\)}.  

\begin{theoreme} \label{t: existence domaine errant} Supposons que $\FF$ vérifie la propriété \( \Pr\). Alors le saturé de $B$ par $\FF$ est un ouvert strict $D \subset \Ptwo$ sur lequel le feuilletage est une fibration sur~\(B\) ayant une structure de fibré lisse localement trivial en disques.
\end{theoreme}

\begin{proof} 
Soit \( \Att _\W (B ) \) l'ensemble d'attraction de \(B\), c'est-à-dire l'ensemble des points \(q\in \Ptwo\) tels que \( \Phi_\W ^t (q) \) tend vers un point de \(B\) lorsque \(t\) tend vers \(+\infty\) ; il s'agit d'un ensemble invariant par \(\Phi_\W\), contenant le voisinage tubulaire  \(U_B\) de \(B\) construit au lemme \ref{l: construction de UB}, et chaque trajectoire de \(\Phi_\W\) contenue dans \(\Att_\W (B)\setminus B\) intersecte \(\partial U_B\) en un unique point. Ainsi, \(\Att_\W (B)\) est un fibré lisse localement trivial en disques au-dessus de \(B\) dont les fibres sont les ensembles d'attraction
\[  \Att_\W (b) = \{ q \in \Ptwo \ |\ \lim _{t\rightarrow +\infty} \Phi_\W ^t ( q ) = b \} \text{ pour } b \in B, \]
et la projection \( \Att_\W (B)\rightarrow B\) est holomorphe.

Remarquons que pour tout \(b\in B\),  \(\Att_\W (b)\) est contenu dans la feuille \(\FF (b)\) du feuilletage \(\FF\) passant par \(b\). Nous allons en fait voir que \(\Att_\W (b)=\FF(b)\) ce qui permettra de conclure la démonstration du théorème.
 
\begin{lemme}\label{l: croissance lineaire}
La croissance de la fonction \( f = -\log \norm{\cdot}^2 \) le long d'une portion de trajectoire de \( \tilde{W}\) contenue dans \( \Pi^{-1} (\Ptwo\setminus \text{Int} (U_B))\) est linéaire avec une constante uniforme. En d'autres termes, il existe un réel \( a >0 \) tel que pour tout \( p \in \Pi^{-1} (\Ptwo\setminus \text{Int} (U_B)) \) et tout \( t\geq 0\) tel que \(\Phi_{\tilde{\W}}^s (p)\in \Pi^{-1} (\Ptwo\setminus \text{Int} (U_B)) \) pour tout \(s\in [0,t]\),  
\[ f (\Phi_{\tilde{\W}}^t (p)) \geq a t + f(p). \]
\end{lemme}

\begin{proof}
Si \( p \in \Pi^{-1} (\Ptwo\setminus \text{Int} (U_B))\), alors toute la trajectoire \(\{ \Phi^t _{\tilde{W}} (p ) \} _{t\leq 0}\) est contenue dans \( \Pi^{-1} (\Ptwo\setminus \text{Int} (U_B)) \). Or, comme \( V\) ne s'annule pas sur \(\C^3 \), le champ \(\tilde{W}\) ne s'annule pas dans \(\Pi^{-1} (\Ptwo\setminus \text{Int} (U_B))\) (on observera que, même si \(W\) s'annule en chaque singularité de \(\FF\), le champ \(\tilde{W}\) ne s'annule pas sur \( \Pi^{-1} (\sing)\)). Ainsi, la quantité \( df ( \tilde{W}) =  \norm{\tilde{W}}_{\tilde{g}} ^2\) est strictement positive en tout point de \(\Pi^{-1} (\Ptwo\setminus \text{Int} (U_B))\). Cette dernière étant invariante par multiplication par les scalaires et \(\Ptwo\setminus \text{Int} (U_B) \) étant compact, elle est minorée par une constante \(a >0\) sur \(\Pi^{-1} (\Ptwo\setminus \text{Int} (U_B))\). Le lemme en résulte immédiatement.\end{proof}

Soit \(b\in B\) et \(p\in \C^3 \setminus \{0\}\)  tel que \(\Pi (p) =b\). La restriction de \(\Pi\) à la feuille \(\FF_V (p)\) de  \(\FF_V\) passant par \(p\) est un revêtement abélien \(\Pi_p : \FF_V(p)\rightarrow \FF(b)\)  et la restriction au domaine d'attraction \( \Att_{\tilde{\W}} (p) = \{ p ' \ | \ \lim  _{t\rightarrow +\infty} \Phi_{\tilde \W} ^t (p ') = p \} \) est un difféomorphisme \( \Att_{\tilde{\W}} (p) \rightarrow \Att_\W (b)\). Or le lemme \ref{l: croissance lineaire} montre que la restriction de \( f\) à \(\Att_{\tilde{\W}} (p)\) est propre, ce qui établit que \(\FF_V (p) = \Att_{\tilde{\W}} (p)\), et par conséquent 
\[ \FF(b)=\Pi_p(\FF_V (p)) =\Pi_p (\Att_{\tilde{\W}} (p))= \Att_\W (b).\] 
\end{proof}

\section{Hyperbolicité de \(\W\) : partie I}  \label{s: hyperbolicite} 

Nous nous donnons un feuilletage algébrique \(\FF\) de \(\Ptwo\) qui satisfait les propriétés \( \Pr\) et \(\Qr\).

\subsection{Hyperbolicité longitudinale} 

On considère dans ce paragraphe le fibré en droites réelles au-dessus de \(  \Ptwostar \) défini par 
\begin{equation} \label{eq: fibre WF} \WF := \TRF / \R \W .\end{equation} 
Comme le flot \( \Phi_\W  \) associé à \(\W\) laisse invariantes les  distributions \(\TRF\) et \( \R\W\), il se relève naturellement en un flot agissant sur \( \WF \) via sa différentielle. La métrique hermitienne \( g\) sur \( T\FF \) que l'on a définie au paragraphe \ref{ss: metrique sur TF} induit une métrique \(g\) sur \( \TRF\simeq T\FF\) dans la partie régulière de \(\FF\) et une métrique \( \gWF \) sur \( \WF \) définie par  
\begin{equation} \label{eq: metrique sur NWF} \gWF ([v]) := | \text{vol}_g (v , \W ) | \end{equation} 
pour tout \( q\in  \Ptwostar \) et tout \([v] \in \WF \). Dans cette formule, \( \text{vol}_g\) désigne la forme volume sur \( \TRF \) induite par la métrique hermitienne \(g\). 

\begin{lemme}\label{l: expansion}
La métrique \( \gWF \) sur \( \WF \) est strictement contractée par le flot \( D\Phi _W \), c'est-à-dire que pour tout \(q \in \Ptwostar \),  tout vecteur non nul \( [v]\in (\WF) _q \) et tout réel \( t>0\) on a 
\[ \gWF \left( \left[ D\Phi_{\W} ^t (v)\right] \right) < \gWF ([v]) .\]
\end{lemme}

\begin{proof}[Démonstration]
On a pour tout \(q \in \Ptwostar \) et tout couple de vecteurs \(v,w \in T_q^{\R} \FF \) 
\[ \frac{d}{d t} _{t=0} \text{vol} _g ( D\Phi_\W ^t (v), D\Phi_\W^t (w) )  = \text{div} ( \W ) (q)\ \text{vol}_g (v,w) ,\] 
où \(\text{div} (\W) \) désigne la divergence de \(\W\) le long des feuilles vis-à-vis de la métrique~\(g\). En posant \(w = \W(q)\), et en remarquant que \( D\Phi_\W^t (\W(q) ) = \W(\Phi_\W ^t (q) ) \), on s'aper\c{c}oit que le résultat est équivalent à montrer que \( \text{div} (\W) < 0\).  Or si \( q=[p]\), on a \( \text{div} (\W) (q) = \text{div} (\tilde{\W})(p) \) et le champ \(\tilde{\W}\) est le gradient de la fonction \(p\mapsto -\log \norm{p} ^2\) qui est strictement pluri-sur-harmonique en dehors des directions radiales de \(V\), ce qui conclut la démonstration du lemme. 
\end{proof} 

\subsection{Calcul de la connexion de Bott de \(\FF \) et hyperbolicité transverse}\label{ss: contraction transverse} 

Le fibré normal \(\NF\) à \(\FF\) est un fibré en droites holomorphe au-dessus de \(\Ptwo\) qui, dans la partie régulière de \(\FF\), s'identifie au quotient  \( \NF = T\Ptwo / m(T\FF ) \) (\cite{Brunella}). Ce fibré est muni d'une connexion feuilletée, c'est-à-dire  d'une connexion définie uniquement dans la direction du feuilletage (\cite[\(\mathsection\)~2.1]{Deroin Guillot}), appelée la connexion de Bott et notée \( \nabla _{Bott}\). Dans la partie régulière où le feuilletage est défini par \( \tau = cst\) dans des coordonnées feuilletées \( (z,\tau) \), la connexion de Bott s'exprime par 
\begin{equation}\label{eq: connexion de Bott} \nabla_{Bott}  \left( \left[u \frac{\partial}{\partial \tau} \right] \right)= d u \cdot  \left[\frac{\partial}{\partial \tau} \right] \end{equation}
pour toute fonction holomorphe \(u\). 

D'un autre coté, le champ \( V \) induit une connexion feuilletée \(\nabla _V\) sur le fibré tautologique \(\mathcal{O}(-1) \) dont les sections plates en restriction à chaque feuille sont les courbes intégrales du champ \(V\). Cette connexion induit une connexion \(\nabla_V^k \) le long des feuilles sur toutes les puissances \(\mathcal{O}(k)\) du fibré tautologique. 

\begin{lemme}\label{l: calcul de la connexion de Bott}
Si \(V\) est de divergence nulle, il existe un isomorphisme entre \( \NF \) et \(\mathcal{O}(d+2) \) qui envoie la connexion de Bott \(\nabla_{Bott}\) sur la connexion \( \nabla_V^{d+2} \). 
\end{lemme}

\begin{proof}[Démonstration]
Soit \(\GG\) le feuilletage singulier de \( \C^3 \setminus \{0\}\) dont la distribution tangente est   
\[  T \mathcal G = \mathbb C R \oplus \mathbb C V, \]
\(R\) désignant le champ radial. En d'autres termes, \(\GG= \Pi^* \FF\). Si \(N\GG\) désigne le fibré normal à \(\GG\), on a donc \( \Pi^* \NF = N\GG\).  
Pour $p\in \mathbb C^3\setminus \{ 0 \}$, la forme linéaire 
\[\varphi _p  (u) = \text{det} (R(p) , V(p) , u) \]
a pour noyau \( T_p \mathcal G\) et définit donc une forme linéaire non nulle sur \(N_p \GG\). Ces formes linéaires  vérifient 
\[ \varphi_{\lambda p}  (\lambda u ) = \text{det} (R(\lambda p) , V(\lambda p), \lambda u )  = \lambda ^{d+2} \varphi _p (u),\]
pour tous \( \lambda \in \mathbb C^*\), \(p\in \mathbb C^3\setminus \{0\}\) et \(u\in N_p\mathcal G\). On construit ainsi une application  \(\mathcal{O}(-1) \rightarrow {\NG }^*\)  
\begin{equation}\label{eq: application homogene} p\in \mathcal{O}(-1)_{[p]}  \mapsto \varphi_p \in {N_p \GG} ^* \simeq {N_p \FF}^*  \end{equation}  
qui est \( (d+2)\)-homogène et qui induit un isomorphisme entre \(\mathcal{O}(d+2) \) et \( \NF \). 

Dans ce qui suit on montre que la connexion de Bott sur \( \NF\) est envoyée sur la connexion \( \nabla_V ^{d+2} \) sur \( \mathcal{O} (d+2) \) par cet isomorphisme. Rappelons que ces connexions sont définies dans la direction des feuilles uniquement. 
Pour calculer la connexion de Bott, observons que le flot local \( \Phi_V \) associé au champ de vecteurs~\( V\) préserve \(\GG\) et, par conséquent, si \([v]\in N_p\GG\), alors la section de \( \NG\) le long de la feuille \(\FV(p)\) passant par \(p\) définie par 
\begin{equation}\label{eq: calcul connexion de Bott} t\in (\C, 0) \mapsto  v(t) =[D\Phi_V^t (v)] \in N_{p(t)}\GG \end{equation}
est plate. Dans cette formule, on paramètre le germe de  feuille \((\FV(p) , p)\) par le germe de surface de Riemann \((\C, 0)\) via \( t\mapsto p(t) =\Phi_V^t (p)\). 
Or \( D\Phi_V ^t (V) = V\) et, puisque \(V\) est homogène, il existe une fonction \( u\) telle que \( D\Phi_V^t (R) = R + u V\). On a alors 
\[ \varphi_{p(t)} (v_t ) = \det ( D\Phi_V ^t (R(p)) , D\Phi_V ^t ( V(p(t)) ),  D\Phi_V ^t (v) ) )  \]
et, puisque \(V\) est de divergence nulle, on en déduit la relation 
\[\varphi_{p(t)} (v_t) = \det ( R(p) , V(p), v) = \varphi_p (v) . \]
C'est exactement cela qui exprime que la connexion induite par le champ \(V\) sur \(\mathcal{O} (-1) \) est envoyée sur la connexion de Bott sur \( \NF^* \) par l'application \eqref{eq: application homogene}. 
\end{proof}

Dans la partie régulière de \(\FF\), définissons le fibré normal réel 
 \[ \NRF := T^\R \Ptwo / T^\R \FF.\] 
L'isomorphisme \( T \Ptwo \simeq T^\R \Ptwo\) induit un isomorphisme \( {\NF}_{| \Ptwo\setminus \sing}  \simeq \NRF \) qui échange les connexions de Bott sur ces deux fibrés. Il fournit également une structure de fibré lisse en droites complexes sur \( \NRF\).

\begin{corollaire} \label{c: decroissance exponentielle des derivees transverses}
Il existe une métrique hermitienne \( g_N \) sur le fibré normal \(\NRF\) qui, en dehors de \( B\cup \sing\), est strictement dilatée par le flot induit par \( W\). Plus précisément, si  \(U_B\) désigne le voisinage de \(B\) construit dans la démonstration du théorème \ref{t: existence domaine errant}, il existe une constante \(b>0\) telle que, pour tout \(\Pi(p) \in \Ptwostar \) et tout réel \(t \geq 0\) tels que la trajectoire \( \Phi_W ^{[0,t]}(\Pi(p))\) est contenue dans \( \Ptwo \setminus \text{Int} (U_B)\), on a pour tout \([v]\in N_{\Pi(p)}^{\R}\FF\), 
\begin{equation}\label{eq: decroissance}  g_N ([ D\Phi_W ^t (v) ] ) \geq \exp (bt) g_N ([v]) . \end{equation}
\end{corollaire} 

\begin{proof}[Démonstration]
On définit la métrique \(g_N\) comme étant l'image de la métrique sur \(\mathcal{O}(d+2) \) induite par la métrique hermitienne standard sur \(\C^3\) par la composition de l'isomorphisme \(\mathcal{O}(d+2) \rightarrow \NF\) défini par \eqref{eq: application homogene} et de l'isomorphisme naturel  \((\NF)_{| \Ptwo\setminus \sing}  \rightarrow \NRF \).  Les sections plates de la connexion induite par \(V\) sur \(\mathcal{O}(-1)\) le long des courbes intégrales de \(W\) sont les courbes intégrales de \(\tilde{W}\). Or le lemme \ref{l: croissance lineaire} montre que 
\begin{equation} \label{eq: croissance} \norm{ \Phi _{\tilde{W} } ^t (p) } ^2  \leq \exp (-a t) \norm{p}^2 \end{equation}
Le corollaire s'ensuit en vertu du lemme \ref{l: calcul de la connexion de Bott}. 
\end{proof}


\section{La décomposition de Fatou/Julia}\label{ss: decomposition Fatou Julia}
 
Dans~\cite{GGMS}, Ghys, Gomez-Mont et Saludes associent à un feuilletage \(\FF\) de \(\Ptwo\) ayant des singularités hyperboliques une partition de \( \Ptwo\) en un ensemble de Fatou \( \Fatou _{\text{GGS}} (\FF) \) et un ensemble de Julia \( \Julia _{\text{GGS}}(\FF )\). Ils donnent également une classification des composantes de \(\Fatou _{\text{GGS}} (\FF)\) en trois familles, l'une d'elles correspondant aux composantes errantes.
Un point de vue alternatif a ensuite été développé par Asuke \cite{asuke} qui associe une décomposition Fatou/Julia à un pseudo-groupe \(\Gamma\) de transformations holomorphes agissant sur une surface de Riemann~\(T\) ayant la propriété de génération compacte.

Rappelons que la propriété de génération compacte, introduite par Haefliger \cite{Haefliger},  stipule qu'il existe un ouvert relativement compact \( T'\subset T\) qui intersecte toutes les orbites de \(\Gamma\) et que, de plus, la restriction \(\Gamma_{|T'}\) de \(\Gamma\) à \(T'\) est engendrée par un nombre fini d'éléments \( \gamma_k \in \Gamma_{|T'}\), chacun se prolongeant en un élément \( \widetilde{\gamma_k}\in \Gamma\) dont le domaine de définition dans \( T\) contient l'adhérence du domaine de définition de \(\gamma_k\) dans \( T'\). Le pseudo-groupe \(\Gamma'= \Gamma_{|T'}\) est appelé une réduction du pseudo-groupe \( \Gamma\). Un ouvert \(U\subset T'\) est un ouvert de Fatou si tout germe de \(\Gamma'\) en un point de \(U\) se prolonge en un élément de \(\Gamma'\) défini sur \(U\).  L'ensemble de Fatou est l'orbite par \(\Gamma\) de l'union des ouverts de Fatou contenu dans une réduction \(\Gamma'\) de \(\Gamma\) et il est indépendant de la réduction choisie. L'ensemble de Julia est son complémentaire.

Le pseudo-groupe d'holonomie d'un feuilletage holomorphe sur une surface complexe compacte, dont les singularités sont hyperboliques, est de génération compacte (\cite{GGMS}). L'ensemble de Fatou \(\Fatou_{\text{A}}(\FF)\) d'un tel feuilletage \(\FF\) au sens d'Asuke est l'union des feuilles correspondant aux points de l'ensemble de Fatou du pseudo-groupe d'holonomie du feuilletage dans sa partie régulière; l'ensemble de Julia \(\Julia_{\text{A}}(\FF) \) de \(\FF\) au sens d'Asuke est son complémentaire. Notons que, d'après \cite[Proposition 4.2]{asuke}, nous avons 
\[\Fatou _{\text{GGS}} (\FF) \subset  \Fatou _{\text{A}} (\FF) .\]

Il nous sera utile dans ce qui suit de conna\^{\i}tre le critère suivant pour qu'un point appartienne à l'ensemble de Julia au sens d'Asuke d'un feuilletage holomorphe \(\FF\) sur une surface complexe compacte avec des singularités hyperboliques : étant donné une métrique hermitienne sur le fibré normal au feuilletage et un compact contenu dans la partie régulière, regardons l'ensemble des applications d'holonomie d'un germe de transversale en notre point vers une transversale en un point de ce compact. Si les dérivées de ces applications d'holonomie en notre point forment un ensemble non borné, alors notre point appartient à l'ensemble de Julia. 

\begin{proposition}\label{p: decomposition Fatou Julia}
Soit \( \FF\) un feuilletage holomorphe sur \(\Ptwo\) satisfaisant les propriétés \(\Pr\) et \(\Qr\). Alors les ensembles de Fatou au sens de Ghys-Gomez-Mont-Saludes et d'Asuke sont tous les deux égaux au domaine errant \(D\) construit au théorème~\ref{t: existence domaine errant}. 
\end{proposition} 

\begin{proof}
Comme  $D\subset \Fatou_{\text{GGS}} (\FF) \subset \Fatou_{\text{A}}(\FF)$, il suffit donc de montrer que le complémentaire de \(D\) est contenu dans \(\Julia _{\text{A}} (\FF)\). Soit \( p \in \Ptwo \setminus D\). Si ce point est singulier, il appartient à  \(\Julia_{\text{GGS}}(\FF) \) et à  \(\Julia_{\text{A}}(\FF) \) par définition. S'il est régulier, son orbite positive \( \{ \Phi_{\W}  ^{ t}(p) \} _{t\geq 0} \) par le flot induit par le champ \( \W \)  ne peut s'approcher ni de \( B \) (car sinon \(p\) appartiendrait au domaine errant), ni de \( \sing\) (car ces dernières sont des sources pour \( \W \)). Ainsi, d'après le corollaire \ref{c: decroissance exponentielle des derivees transverses}, le logarithme de la dérivée de l'holonomie du feuilletage \( \FF\) en \(p\) le long du chemin  \( \{ \Phi_{\W}  ^{t}(p) \} _{0\leq t \leq T} \) cro\^{\i}t linéairement et on en déduit la propriété suivante : l'ensemble des dérivées en \(p\) des applications appartenant au pseudo-groupe d'holonomie de la restriction du feuilletage au complémentaire d'un certain voisinage de l'ensemble singulier de \( \FF \) est non borné. Ainsi \(p\) appartient à \(\Julia _{\text{A}} (\FF)\).  \end{proof}


\section{Hyperbolicité de \(W\) : partie II}

Dans cette partie, nous poursuivons l'étude d'un feuilletage algébrique complexe \(\FF\) de \(\Ptwo\) satisfaisant les propriétés \(\Pr\) et \(\Qr\); en particulier, nous construisons les feuilletages stables et instables, faibles et forts, du champ \(\W\) dans le complémentaire~\(\Ptwostar\) de l'ensemble \(B \cup \sing\). Nous introduisons aussi l'ensemble hyperbolique \(K\) dont l'ensemble d'attraction est l'ensemble de Julia privé de~\(S\).

\subsection{Une métrique riemannienne sur \(\Ptwostar\) adaptée à \(W\)}\label{s: metrique adaptee}

Soit \(U_B\) le voisinage de \(B\)  construit au lemme \ref{l: construction de UB}. à partir d'une métrique riemannienne \( h_{\partial U_B} \) sur \(\partial U_B\), nous construisons une métrique riemannienne \(h\) sur \( U_B\setminus B\) en posant 
 \begin{equation}\label{eq: metrique sur UB}   h = dt^2 + e^{-2 t } h_{\partial U_B, T^{\R} \FF} + e^{2t} h_{\partial U_B, T^{\R}\FF}^\perp ,\end{equation}
 où  les métriques \( h_{\partial U_B, T^{\R}\FF}\) et \( h_{\partial U_B, T^{\R}\FF}^\perp\) ont respectivement pour noyau \( (T^{\R}\FF\cap T\partial U_B)^\perp\) et \( T^{\R}\FF\cap T\partial U_B\), et vérifient \( h_{\partial U_B} = h_{\partial U_B, T^{\R}\FF} +h_{\partial U_B, T^{\R}\FF}^\perp\). Dans l'équation \eqref{eq: metrique sur UB} la coordonnée \(t \) a été définie au lemme \ref{l: construction de UB} et prend ses valeurs dans \([0,+\infty [\).  
 
Soit \(U_S\) le voisinage de \( \sing\) construit au lemme \ref{l: construction de US}.
On définit une métrique~\(h\) dans \(U_S\setminus \sing\) en partant d'une métrique lisse \( h_{\partial U_S}\) sur \(\partial U_S\) pour laquelle \( T^{\R}\FF\cap T\partial U_S\) est orthogonal au feuilletage  \(\mathcal R\) (lemme \ref{l: construction de US}) et on pose 
\begin{equation} \label{eq: metrique sur US}  h= dt^2 + e^{-2t } h_{\partial U _S, T^{\R}\FF} + e^{2t} h_{\partial U_S, T^{\R}\FF}^\perp ,\end{equation}
où les métriques \( h_{\partial U_S, T^{\R}\FF}\) et \( h_{\partial U_S, T^{\R}\FF}^\perp\) ont respectivement pour noyau \( (T^{\R}\FF\cap T\partial U_S)^\perp= T\mathcal R\) et \( T^{\R}\FF\cap T\partial U_S\), et vérifient \( h_{\partial U_S} = h_{\partial U_S, T^{\R}\FF} +h_{\partial U_S, T^{\R}\FF}^\perp\). Dans l'équation \eqref{eq: metrique sur US} la coordonnée \(t \) a été définie au lemme \ref{l: construction de US} et prend ses valeurs dans \(]-\infty , 0] \).

Pour terminer, nous étendons la métrique \(h \) en une métrique riemannienne lisse sur \( \Ptwostar\) de fa\c{c}on arbitraire et nous notons \( \hWF \) et \( h_N \) les métriques induites par \(h\) sur les fibrés \( \WF \) et \( N_{\FF}^\R \) respectivement. 

\begin{lemme} \label{l: hyperbolicite}
Pour tout \( t\in \R\), on a 
\[ \sup _{p\in X,\ t\in [-1,1]} \norm{D\Phi _{\W} ^ t(p)} <+\infty .\]
D'autre part, il existe des constantes \( a,b,c,d  >0\) telles que, pour tout \( t\in \R\) et pour tout \( v\in  \WF\), 
\[            c\exp (-a t ) \hWF  ( v)  \leq \hWF  ( D\Phi_W^t (v) ) \leq d \exp (-b t)  \hWF  ( v),\]
et, pour tout \( w\in N^{\R}_{\FF} \), 
\[    c \exp ( b t ) h_{\FF} (w) \leq h_{N} ( D\Phi_W^t (w) ) \leq d \exp (a t) h_{\FF} (w) .\]
\end{lemme} 

\begin{proof} [Démonstration]
La proposition est satisfaite pour une trajectoire qui reste dans un des voisinages \( U_B\setminus B \) ou \( U_S \setminus \sing\) (avec constantes \( c,d=1\) et \(a = b = 2\) par construction de la métrique \( g\)). Elle est également satisfaite pour une trajectoire restant dans le compact \( \Ptwo\setminus (U_B\cup U_S )\), d'après le lemme \ref{l: expansion} et le corollaire \ref{c: decroissance exponentielle des derivees transverses}. Comme le long d'une trajectoire quelconque, il ne peut y avoir que deux transitions entre ces deux régimes, le résultat en découle.
\end{proof}

\subsection{Un lemme classique} 
Le résultat suivant est bien connu mais, n'ayant pas trouvé l'énoncé sous cette forme dans la littérature, nous le redémontrons. 

\begin{lemme} \label{l: construction de distributions}
Soit \(\mathcal E \rightarrow X\) un fibré vectoriel réel de dimension finie au-dessus d'un espace topologique \(X\), muni 
\begin{itemize} 
\item d'un flot continu d'automorphismes \( \hat{\Psi}=\{\hat{\Psi}_t \}_{t\in \R} \) induisant un flot \( \Psi= \{ \Psi _t \}_{t\in \R} \) agissant sur \(X\)
\item et d'une métrique continue et définie positive \( |\cdot| \). 
\end{itemize} 
Supposons qu'il existe un sous-fibré continu \(\hat{\Psi} \)-invariant \( E\subset \mathcal E\) et écrivons  
\begin{equation} \hat{\Psi}_t = \left( \begin{array}{cc} \alpha_t & \gamma_t\\  0 & \delta_t  \end{array}\right) \end{equation} 
relativement à la décomposition \(\mathcal E = E\oplus E^\perp\).  Enfin, supposons que  
\begin{equation} \label{eq: sup} \sup _{p\in X,\ t\in [-1,1]} \norm{\hat{\Psi}_t(p)} <+\infty ,\end{equation} 
et qu'il existe des constantes \(a , c >0\) telles que, pour tout \(p \in X\) et tout \(t\in \R\), on ait 
\begin{equation} \label{eq: majoration}  \norm{\alpha_t(p)^{-1} } \leq c  \text{ et } \norm{\delta_t(p)} \leq c\exp (-a t),\end{equation} 
pour la norme d'opérateur associée à  \(|\cdot |\). Alors, l'ensemble \(F\) formé des éléments \( f \in \mathcal E \) tels que 
\begin{equation} \label{eq: stabilite}  \lim_{t\rightarrow +\infty}  \norm{\hat{\Psi}_t (f) } =0 \end{equation} 
est un sous-fibré vectoriel continu de \(\mathcal E\) tel que 
\begin{itemize} 
\item  \(\mathcal E= E\oplus F \) et l'angle entre \(E\) et \(F\) est uniformément minoré, 
\item \(F\) est \(\hat{\Psi}\)-invariant,
\item  et la quantité \(\sup _{p\in X} \norm{\left( \hat{\Psi}_t(p) \right)_{|F}} \) tend vers \(0\) exponentiellement vite lorsque \(t\) tend vers \(+\infty\).
\end{itemize} 
\end{lemme}

\begin{proof}[Démonstration] 
Choisissons \( f\in F\) au-dessus d'un point \(p_0\) de \(X\) et, pour tout \(t\in \R\), notons \(p_t = \Psi (p_0)\), \(f_t = \hat{\Psi}_t (f) \), et décomposons \( f_t= e_t+e^\perp_t\) avec \( e_t\in E\) et \(e^\perp_t \in E^\perp\). On a alors
\begin{equation}\label{eq: caracterisation F 1} e_0 = - \alpha_t (p_0) ^{-1} \gamma_t (p_0) e_0^\perp + \alpha_t(p_0)^{-1} e_t .\end{equation}
Le second membre du terme de droite tend vers \(0\) lorsque \(t\) tend vers \(+\infty\).  Étudions plus en détail le premier membre. 
En posant 
\begin{equation} \label{eq: definition ut} u_t  (p_0) := \alpha_t (p_0)^{-1} \gamma_t(p_0)\in \text{Hom} (E_{p_0} ^\perp, E_{p_0} ) , \end{equation}
la relation de cocycle \(\hat{\Psi}_t (p_s)= \hat{\Psi}_{t-1}(p_{s+1}) \circ \hat{\Psi}_1(p_s)\)  nous donne
\begin{equation}\label{eq: caracterisation F 3} u_t(p_0) = u_1(p_0) + \alpha_1(p_0)^{-1} u_1(p_1) \delta_1(p_0) + \ldots + \alpha_{t-1} (p_0)^{-1} u_1(p_{t-1}) \delta_{t-1}(p_0) . \end{equation}
En utilisant \eqref{eq: sup} et \eqref{eq: majoration}, on obtient que le terme de droite de cette expression est la somme partielle d'une série normalement convergente, ce qui montre que \(u_t(p_0) \) admet une limite \( u_{\infty} (p_0)\) lorsque \(t\) tend vers \(+\infty\), qui dépend contin\^ument de \(p_0\). Comme d'après \eqref{eq: caracterisation F 1}, on a \( e_0 = - u_t (p_0) e_0^\perp + \alpha_t(p_0 )^{-1} e_t\), on obtient l'expression 
\begin{equation}\label{eq: caracterisation F 4} e_0 = - u_{\infty} (p_0) e_0^\perp .\end{equation}
Réciproquement, supposons que \(f = e_0 + e_0^\perp \), où \(e_0\) et \(e_0^\perp\) sont des éléments de \(E\) et \(E^\perp\) respectivement, qui vérifient \eqref{eq: caracterisation F 4}. On a alors en manipulant \eqref{eq: caracterisation F 1}
\begin{equation} \label{eq: caracterisation F 5}  e_t = \alpha_t(p_0) \left(e_0 + \alpha_t(p_0)^{-1} \gamma_t(p_0) e_0^\perp \right) = \alpha_t(p_0) \left( u_t(p_0) - u_\infty (p_0) \right) e_0^\perp,\end{equation}
Or 
\begin{equation} \label{eq: caracterisation F 6} \alpha_t(p_0) \left( u_\infty (p_0) -u_t(p_0)  \right) =  \left( \sum_{k\geq 0} \alpha_k(p_t) ^{-1} u(p_{t+k})\delta_k (p_t) \right) \delta_t(p_0), \end{equation}
donc d'après \eqref{eq: caracterisation F 5} et \eqref{eq: caracterisation F 6}, on obtient 
\[ \norm{e_t } \leq  c' \exp (-a t) \norm{e_0^\perp} \] 
pour une constante \(c'\) indépendante de \(p_0\), \(t\) ou encore  \(f\).  On a aussi  
\[ \norm{e_t^\perp } =\norm{ \delta_t(p_0)  e_0^\perp } \leq c \exp (-a t) \norm{ e_0^\perp }, \]
ce qui montre bien que \( f_t = e_t + e_t ^\perp \) converge vers \(0\) (et de plus exponentiellement vite). On a donc bien montré que \(F\) est le graphe du morphisme continu de sous-fibrés de \(\mathcal E\) 
\[  e^\perp \in E_p ^\perp \mapsto - u_{\infty} (p) e^\perp \in E_p,\]
ce qui achève la démonstration du lemme.
\end{proof}

\subsection{Les distributions stables et instables du flot \(\W\)} \label{ss: distribution stables}


\begin{proposition} \label{p: distributions stables et instables}
Il existe une décomposition \(D\Phi\)-invariante  et continue  
\begin{equation} \label{eq: decomposition} T^\R \Ptwostar  = T\FSFo _\W \oplus \R \W \oplus  T\FIFo_W \end{equation}
où \(\TRF_{|\Ptwostar} = T\FSFo _\W \oplus \R \W\). Les angles entre les facteurs de la décomposition \eqref{eq: decomposition} sont uniformément minorés pour la métrique \(h\). De plus, il existe des constantes \(a,b,c,d\) strictement positives telles que pour tout \(v\in T\FSFo _\W\)
\begin{equation} \label{eq: estimation distribution stable forte}            c\exp (-a t ) h  ( v)  \leq h  ( D\Phi_W^t (v) ) \leq d \exp (-b t)  h ( v),\end{equation} 
et, pour tout \( w\in T\FIFo_W \), 
\begin{equation} \label{eq: estimation distribution instable forte}    c \exp ( b t ) h (w) \leq h ( D\Phi_W^t (w) ) \leq d \exp (a t) h (w) .\end{equation}
\end{proposition}

\begin{proof}
On applique le lemme \ref{l: hyperbolicite} (première estimée) ainsi que le lemme \ref{l: construction de distributions} au fibré \(\mathcal E = \TRF\), au flot \( \hat{\Psi}:= D\Phi_\W  \) et au sous-fibré \( E = \R W\). On obtient l'existence d'un sous-fibré continu \( T\FSFo_\W \subset \TRF\) de dimension réelle \( 1\), qui est uniformément exponentiellement contracté par le flot \( D\Phi_\W \).

On applique maintenant le lemme \ref{l: hyperbolicite} (deuxième estimée) ainsi que le lemme \ref{l: construction de distributions} au fibré \(\mathcal E = T^{\R} \Ptwo\), au flot \(\hat{\Psi}\) défini par \( \hat{\Psi}_t := D\Phi^{-t} _W \) pour tout \(t\in \R\) et au sous-fibré \( E = \TRF\), pour obtenir l'existence d'un sous-fibré continu  \( T\FIFo_W \subset T^{\R} \Ptwo\) de dimension réelle \( 2\), transverse à \(\TRF\), qui est uniformément exponentiellement contracté par le flot \(\hat{\Psi} \).
\end{proof}

\begin{remarque} 
Par construction de la métrique  \(h\) au voisinage des singularités de \(\FF\), la distribution \( T\FSFo _\W  \) est l'intersection \( \TRF \cap T\partial U_B\) en restriction à \(\partial U_B\). De même  la distribution  \(T\FIFo_W \) est le fibré tangent du feuilletage de Reeb \(\mathcal R\) en restriction à \( U_S\) .
\end{remarque}

\subsection{Ensemble hyperbolique maximal} \label{ss: ensemble hyperbolique}

Pour toute singularité \(s \in \sing\), on note \( \Rep_W (s) \) le bassin de répulsion de \(s\), à savoir l'ensemble des points dont la trajectoire par le flot \( \Phi_{\W} \)  dans le passé converge vers \(s\). Il s'agit d'un ouvert de \(\Ptwo\), puisque chaque singularité de \(\FF\) est une source (propriété \(\Qr\)). On notera \(\Rep_W (\sing)\) l'union des bassins de répulsion des singularités de \(\FF\).  

Rappelons que l'ensemble errant \( D\), l'ensemble de Fatou \(\Fatou (\FF)\) et l'ensemble d'attraction \(\text{Att}_W  (B)\) de \(B\) sont tous les trois égaux (théorème \ref{t: existence domaine errant} et proposition \ref{p: decomposition Fatou Julia}). Introduisons l'ensemble 
\begin{equation}\label{eq: ensemble hyperbolique} K:= \Ptwo\setminus \left( \text{Att}_W  (B) \cup \Rep _W (\sing)  \right). \end{equation} 
Il s'agit d'un compact de \(\Ptwostar\) qui est invariant par le flot \(\Phi_{\W}\).  On remarque que l'ensemble de \(W\)-attraction de \(K\) 
défini par 
\begin{equation} \label{eq: W repulsion de K} \Att_{\W} (K):= \{ p\in \Ptwo \ |\ \lim_{t\rightarrow +\infty} d(\Phi_{\W}^t (p), K) = 0\}\end{equation}
 est l'ensemble de Julia de \(\mathcal F\) privé des singularités et que l'ensemble de \(W\)-répulsion de \(K\) défini par 
 \begin{equation} \label{eq: W attraction de K} \Rep_\W (K) :=  \{ p\in \Ptwo \ |\ \lim_{t\rightarrow -\infty} d(\Phi_{\W}^t (p), K) = 0\} \end{equation}
 est le complémentaire dans \(\Ptwostar\) de l'ensemble de répulsion de \(\sing\).

\begin{proposition} \label{p: K est hyperbolique} L'ensemble \(K\) est un compact \(\Phi_\W\)-invariant  hyperbolique maximal. \end{proposition}

\begin{proof}[Démonstration] Ceci découle de l'hyperbolicité de \( \W\) sur \(\Ptwostar \) vis-à-vis de la métrique \(h\) construite au paragraphe \ref{s: metrique adaptee}. La maximalité découle de ce que toute orbite qui n'est pas dans \(K\) tend vers l'infini dans \(\Ptwostar\) par définition même de~\(K\).
\end{proof}

\begin{corollaire} \label{c: union varietes stables faibles} 
Les ensembles de \(\W\)-attraction (resp. de \(\W\)-répulsion) de \(K\) sont des unions de variétés stables faibles (resp. instables faibles) de points de~\(K\). 
\end{corollaire}

\begin{proof} 
C'est une conséquence de la proposition \ref{p: K est hyperbolique} et de~\cite[Theorem 5.3.25] {FH}.
\end{proof}

\begin{corollaire} \label{cor:nombre-denombrable-orbites-periodiques}
L'ensemble des orbites périodiques de \(\W\) est dénombrable.
\end{corollaire} 

\begin{proof} 
Comme \(K\) est hyperbolique, pour tout \(T>0\), les orbites périodiques de périodes bornées par \(T\) sont isolées. Par compacité de \(K\), il n'y en a qu'un nombre fini. \end{proof} 

\subsection{Les feuilletages stables et instables de \(\W\)}\label{ss: feuilletages stables et instables}

Un feuilletage de classe \( C^{1,0} \) de dimension \(k\) d'une variété \(M\) de classe \(C^1\) est un feuilletage topologique de dimension \(k\) de \(M\) dont les feuilles sont des sous-variétés immergées de classe \(C^1\) de \(M\) qui dépendent de fa\c{c}on continue du paramètre transverse dans la topologie~\(C^1\). Un feuilletage de classe \( C^{1,0} \) admet une distribution tangente qui est une distribution continue de dimension \(k\) de \(TM\).


\begin{proposition}\label{p: feuilletages stables et instables}
Il existe des uniques feuilletages  \( \FSFo_W, \FSFa_W , \FIFo_W, \FIFa_W\) sur \( \Ptwostar \) de classe \(C^{1,0}\) et ayant pour distributions tangentes 
\[ T\FSFo_W,\ \  T\FSFa_W= T\FSFo_W +\R W ,\ \  T\FIFo_W, \ \ T \FIFa_W= T\FIFo_W +\R W . \]
\end{proposition}

\begin{proof} Nous utiliserons le critère suivant. 

\begin{lemme} \label{l: critere integrabilite distribution continue}
Soit  \( D \subset TM\) une distribution continue de dimension \(k\) d'une variété \(M\) de classe \(C^1\). 
Supposons qu'il existe un recouvrement de \(M\) par des ouverts \(U_i \) sur lesquels on a des applications 
\( F_i : \R ^k \times U_i \rightarrow M \) telles que pour  \(p\in U_i\), \(F_i (p)=p\) et \(F_i (\cdot, p) \) est une immersion de classe \(C^1\) tangente à \(D\) qui dépend contin\^ument de \(p\) dans la topologie \(C^1\). 

 Supposons de plus que, pour tout point \(p\in M\), le germe de sous-variété immergée de classe \(C^1\) tangente à \(D\) en \(p\) soit unique. Alors il existe un unique feuilletage de dimension \(k\) et de classe \( C^{1,0}\) dont la distribution tangente est \(D\). 
\end{lemme} 

\begin{proof} 
Pour tout \(i\) et tout point \( p\in U_i \), considérons une variété \(T_{i,p}\) transverse à \(D\) passant par \(p\). La restriction \(f_{i,p} \) de \(F_i\) à \( \R ^k \times T_{i,p}\rightarrow M \) induit un homéomorphisme d'un voisinage \( V_{i,p} \) de \( (0, p)\) dans \( \R ^k \times T_{i,p} \) sur un voisinage \( U_{i,p} \) de \(p\). Ces homéomorphismes envoient localement les plaques \(( \R^k \times t_{i,p})\cap V_{i,p} \) sur des variétés de classes \(C^1\) immergées tangentes à \(D\). La propriété d'unicité des germes de sous-variétés intégrales montre que les inverses \( f^{-1}_{i,p} : U_{i,p} \rightarrow V_{i,p}\) fournissent un atlas de cartes feuilletées et définissent un feuilletage de classe \( C^{1,0} \) dont la distribution tangente est \(D\). 
\end{proof} 

Commen\c{c}ons par construire le feuilletage stable fort \(\FSFo_W\).  Le lemme \ref{l: construction de distributions}, ainsi que la construction de la métrique \( h\) dans \(U_B\), permettent d'affirmer que la distribution \( T\FSFo _\W  \) est l'intersection \( \TRF \cap T\partial U_B\) en restriction à \(\partial U_B\). Étant \(\Phi_W\)-invariante, on en déduit que \( T\FSFo _\W  \) est lisse sur le \(\Phi_W\)-saturé de \(\partial U_B\), c'est-à-dire sur l'ensemble de Fatou privé de \(B\). En particulier, sur l'ensemble de Fatou privé de \(B\) la distribution \( T\FSFo _\W  \) s'intègre en un unique feuilletage lisse, qui se trouve être une fibré en cercles au-dessus du produit \( \R \times B\). Dans l'ensemble de Julia privé de \(\sing\), les trajectoires de \(\Phi_W\) convergent dans le futur vers le compact hyperbolique \(K\) : le théorème de la variété stable \cite[Thm 6.1.1]{FH} montre qu'il existe une famille \( \{ I_p \}_{p\in K}\) d'intervalles plongés \(I_p\),  de classe \( C^\infty\), passant par~\(p\) et tangents à \( T\FSFo_W\). De plus ces variétés dépendent continûment de~\(p\) dans la topologie \(C^1\).   Ces variétés se prolongent en des variétés immergées \( \FSFo_ W (p)\) passant par \(p\) qui vérifient l'équivariante \( \Phi^t_W ( \FSFo_ W (p)) =  \FSFo_ W (\Phi_W ^t(p))\) et qui sont caractérisées par 
\begin{equation} \FSFo_ W (p) = \{q \in \Ptwostar \ |\ d(\Phi_W^t (q), \Phi_W^t(p) ) \rightarrow_{t\rightarrow +\infty} 0 \}. \end{equation}  
Les estimations uniformes de la proposition \ref{p: distributions stables et instables} montrent que tout germe de variété tangente à \(D\) passant par un point \(q\in \FSFo_ W (p)\) avec \( p\in K\) est contenu dans \( \FSFo_ W (p)\) et est donc unique. De plus,  lorsque qu'un point \(p'\) de l'ensemble de Fatou converge vers un point \(p\) de l'ensemble de Julia, les variétés \( \FSFo_W (p')\) convergent vers \(\FSFo_W(p)\) et, étant tangentes à la distribution \(T\FSFo_W\) qui est continue, la convergence a lieu dans la topologie \(C^1\). On peut donc appliquer le lemme   \ref{l: critere integrabilite distribution continue} pour conclure à l'existence d'un feuilletage stable fort \( \FSFo_W\) globalement défini sur \( \Ptwostar\) de classe \( C^{1,0}\) qui de plus est unique.  


Le feuilletage stable faible de \(\W\) vis-à-vis de la métrique riemannienne \(h\) est alors par construction le feuilletage  
\[ \FSFa _W = \FF\] 
(voir le lemme \ref{l: hyperbolicite}). Un tel feuilletage est unique car ses feuilles doivent à la fois être saturées par celles de \( \FSFo_W\) et par les orbites de \(\Phi_W\).


Construisons à présent le feuilletage insta\-ble fort.  Le lemme \ref{l: construction de distributions}, ainsi que la construction de la métrique \( h\) dans \(U_S\), permettent d'affirmer que la distribution \( T\FIFo _\W  \) est la distribution tangente du  feuilletage de Reeb \(\mathcal R\) en restriction à \(\partial U_S\) (lemme \ref{l: construction de US}). Étant \(\Phi_W\)-invariante, on en déduit que \( T\FIFo _\W  \) s'intègre sur le \(\Phi_W\)-saturé de \(\partial U_S\), c'est-à-dire sur  l'ensemble \( \left(\Rep _W \sing\right)\setminus \sing\) en un unique feuilletage de dimension \(2\). Dans le complémentaire de \( \left(\Rep _W \sing\right)\setminus \sing\) privé de \(B\), les trajectoires de \(\Phi_W\) convergent dans le passé vers le compact hyperbolique \(K\), c'est-à-dire qu'elles appartiennent à \(\Rep _W(K)\) : on peut alors appliquer le théorème de la variété instable forte et raisonner de manière analogue au cas du feuilletage stable fort pour construire un unique feuilletage instable fort \(\FIFo_W \) de classe \( C^{1,0}\) dont la distribution tangente est \( T\FIFo _\W  \).  

L'unique feuilletage tangent à \( T\FIFa_W\)  est le feuilletage \(\FIFa_W\) de dimension \(3\) de \(\Ptwostar\) dont les feuilles sont les saturés des feuilles du feuilletage instable fort \(\FIFo_W\) par le flot \(\Phi_W\).   \end{proof}

\begin{remarque}\label{r: reparametrage}
La construction des feuilletages stables et instables que nous avons décrite pour le champ \(\W\) fonctionne de fa\c{c}on similaire pour toute reparamétrisation \( \exp (\varphi) W\) du champ \( W\), où \(\varphi: \Ptwostar\rightarrow \R \) est une fonction lisse constante en dehors d'un compact.   
\end{remarque}

\begin{remarque} Par construction, toutes les feuilles de \(\FIFo_W\) sont des sections transverses ; ce sont des courbes entières pour la structure holomorphe induite par la structure transverse holomorphe de \(\FF\) (nous n'utiliserons pas ce fait dans ce qui suit).
\end{remarque}

\begin{remarque} Il est intéressant de noter qu'au voisinage des singularités de~\(\FF\) le feuilletage \(\FIFa_W\) est localement le produit du feuilletage \(\mathcal R\) sur \(\partial U_S\) par une demi-droite réelle, tandis qu'au voisinage de la section transverse \(B\), il possède localement la structure d'un livre ouvert, même si cette structure ne se globalise pas à \( B\) tout entier :  il y a un phénomène de monodromie le long des chemins fermés de \(B\).
\end{remarque}

\begin{corollaire}
Les ensembles de \(\W\)-attraction (resp. de \(\W\)-répulsion) de \(K\) sont contenus dans \(\Ptwostar\) et ont une structure de lamination par variétés stables faibles (resp. instables). 
\end{corollaire} 

\begin{proof} Cela résulte du fait que ces ensembles sont fermés dans \( \Ptwostar\), du corollaire \ref{c: union varietes stables faibles}  et de la proposition \ref{p: feuilletages stables et instables}.\end{proof}


\section{Groupe affine, ensemble hyperbolique et conjecture d'Anosov} \label{ss: structure hyperbolique feuilletee} 

Le but de ce paragraphe est de vérifier la conjecture d'Anosov pour un feuilletage algébrique du plan projectif complexe satisfaisant les propriétés~\( \Pr\) et~\( \Qr\). En particulier, nous construisons une action localement libre du groupe affine qui nous servira pour établir la stabilité structurelle de ces feuilletages. 

\subsection{Action du groupe affine} 
Nous noterons \(\AF\) le groupe \(\R^2\) muni de la loi \((x,t) \cdot (x',t') = (e^{t'} x + x', t+t')\). 

\begin{lemme}\label{l: action groupe affine} Soit \(\FF\) un feuilletage algébrique de \(\Ptwo\) satisfaisant les propriétés~\( \Pr\) et~\( \Qr\). Il existe une action localement libre et continue du groupe \(\AF\)  sur \(\Ptwostar \) dont les orbites sont les feuilles de la restriction de \(\FF\) à \(\Ptwostar\).
\end{lemme} 


\begin{proof} Soit \(T >0 \) un nombre réel et \(\hWF ^T\) la métrique sur le fibré \( \WF\)  définie par 
\[ \hWF ^T (\cdot):= \int _0^T \hWF ( D\Phi_{\W}^{-s} \cdot) ds , \]
où \(\hWF\) est la métrique sur \( \WF\) induite par \(h\) (voir paragraphe \ref{s: metrique adaptee}).
Définissons la fonction 
\[ u(p) := \frac{d \log  \hWF ^T(D \Phi_{\W}^t v )}{dt} _{|t=0} \text{ pour } v\in \WF (p)\setminus 0 ,\]
qui ne dépend pas du choix de \(v\). Comme 
\[ u(p) = \frac{ -\hWF (D \Phi_{\W}^{-T} v ) + \hWF  (v)}{\hWF^T (v) } ,\]
le lemme \ref{l: hyperbolicite} montre que si \(T\) est choisi suffisamment grand,  la fonction \(u\) 
vérifie 
\begin{equation} \label{eq: u bornee}  -a<  u < -b\end{equation}  
pour certaines constantes uniformes \(a,b>0\). Nous définissons le reparamétrage de \(\W\) par  
\begin{equation}\label{eq: reparametrisation de W} W_R := \frac{-2} {u} W .\end{equation} 
Comme la restriction de \(\W\) à \(\Ptwostar \) est complet, l'estimée \eqref{eq: u bornee} montre que \( W_R\) est complet également. On a de plus la formule
\begin{equation}\label{eq: expansion apres reparametrage} h_{\W,\FF}^T(D\Phi_{\W_R}^t v )= \exp (-2t) h_{\W,\FF}^T( v ) \text{ pour tout } v\in \WF .\end{equation}
Notons que, vu la forme particulière du couple \( (h, \W)\) dans \( U_B\cup U_S\), la fonction \(u\) est constante dans \( \Phi _{\W} ^T (U_B) \cup U_S\) ; en particulier toutes ses dérivées sont bornées. Le lemme \ref{l: hyperbolicite} est donc satisfait pour le champ \(\W_R\) à la place du champ \(\W\). 
Les propositions \ref{p: distributions stables et instables} et \ref{p: feuilletages stables et instables} sont donc également valables pour le champ \(\W_R\) (remarque \ref{r: reparametrage}) et fournissent l'existence d'un feuilletage stable fort \( \FSFo_{\W_R}\) sur  \( \Ptwostar\) de dimension réelle \(1\), de classe \( C^{1,0}\) dont la distribution tangente \( T\FSFo_{\W_R}\) est contenue dans \( \TRF\) et fait un angle avec \( \R \W\) minoré par une constante strictement positive.


Nous notons \(X\) le champ de vecteurs continu sur \( \Ptwostar\) qui est tangent à \( \FSFo_{\W_R}\),  dont la projection dans \(\WF= \TRF / \R W_R\)  est de norme \(1\) vis-à-vis de la métrique  \(\hWF ^T\) et orienté de sorte que le couple \((X, \W_R) \) forme une base directe de \( \TRF\). Nous introduisons la métrique continue \( \hyp\) sur \( \TRF \) comme étant l'unique métrique rendant la base \( X, \W_R\) orthonormale~:
\begin{equation} \label{eq: metrique hyperbolique} \hyp (X,X)= \hyp(\W_R, \W_R) =1 \text{ et }  \hyp (X,\W_R)=0.\end{equation}

\begin{lemme} \label{l: completude hyp} La restriction de \( \hyp \) à toute feuille de \(\FF_{\Ptwostar }\) est complète. 
\end{lemme}

\begin{proof} Comme les métriques \(\hWF ^T\) et \( \hWF\) sont bornées l'une par rapport à l'autre à des constantes multiplicatives près et que l'angle entre \( X\) et \( W_R\) est minoré par une constante strictement positive uniforme, la métrique \(h^T\) et la restriction de \(h\) au feuilletage sont majorées l'une par rapport à l'autre à des constantes multiplicatives strictement positives près. Le résultat découle de ce que la restriction de \(h\) aux feuilles de \(\FF_{\Ptwostar }\) est complète. \end{proof}

Nous définissons le flot \( \Phi_X =\{\Phi_X ^t\} _{t\in \R} \) de la fa\c{c}on suivante : pour tout \(p\in \Ptwostar\),  \( t\mapsto \Phi _ X ^t (p) \) est le paramétrage de classe \(C^1\) de la feuille de \( \FSFo_{\W_R} \) passant par~\(p\), isométrique vis-à-vis de la métrique \(\hyp\), respectant l'orientation et envoyant \(t=0\) sur \(\Phi _X^0 (p)=p\). Il est bien défini pour tout temps d'après le lemme \ref{l: completude hyp}. Par définition de \(\Phi_X\), on a les relations \(\Phi _X ^{t+t'} = \Phi_X^t \circ \Phi_X^{t'} \) pour tous \(t,t'\in \R\), ainsi que  
\begin{equation} \label{eq: derivee du flot} \frac{d\Phi_X^t(p)} {\partial t} _{t=0} = X(p)\text{  pour tout }p .\end{equation} 
D'autre part, \(\Phi_X\) est un flot continu puisque la distance entre \( p\) et \( \Phi^t _X (p)\) est majorée par \(|t|\) pour la métrique \(\hyp\), et donc par une constante fois \(|t|\) pour la métrique \(h\), ce qui montre que \(\Phi _X^t\) converge uniformément vers l'identité sur tout compact. 

La construction de \(\Phi _X\) ainsi que \eqref{eq: expansion apres reparametrage} montrent que les transformations \( \Phi_{W_R} ^t \) échangent les feuilles de \(\FSFo_{\W_R}\) en leur appliquant une contraction (vis-à-vis de \(\hyp\)) d'un facteur \(e^{-t}\) et en préservant leur orientation donnée par \(X\). Nous avons donc les relations
\begin{equation} \label{eq: groupe affine}  \Phi _{W_R} ^t \circ \Phi _X ^ x = \Phi _X ^{e^{-t} x} \circ \Phi _{W_R}  ^t \end{equation}
pour tout \((x,t) \in \R^2\). Nous obtenons alors une action \(\pi : \AF\times \Ptwostar \rightarrow \Ptwostar\) par la formule 
\begin{equation} \label{eq: def pi} \pi ( x, t, p ) = \Phi _{\W_R} ^t \circ \Phi _X ^x (p) \text{ pour tous } (x,t)\in \AF, \ p\in  \Ptwostar.\end{equation} 
Cette action est localement libre d'après \eqref{eq: derivee du flot} (ainsi que la relation analogue pour \( \W_R\) qui est lisse) et \(X\) et \(W_R\) sont linéairement indépendants en tout point.  Les orbites étant contenues dans des feuilles de la restriction de \(\FF\) à \(\Ptwostar\), ce sont donc des ouverts dans ces dernières qui sont connexes. 
\end{proof}




\subsection{Structure hyperbolique, structure affine et topologie des feuilles}

Étant donné un point \( p\in \Ptwostar\), la feuille de la restriction de \(\FF\) à \(\Ptwostar\) passant par \(p\) est \(\FF(p) \setminus B\). Notons \(\pi_p: \R ^2 \rightarrow \FF (p)\setminus B\) le paramétrage défini par 
\begin{equation} \label{eq: coordonnees R^2} \pi_p (x,t) := \pi (x, t, p) . \end{equation}

\begin{lemme}\label{l: revetement}
Les applications \(\pi_p\) sont des immersions de classe \(C^1\) qui induisent des revêtements de \(\R^2\) dans \(\FF (p)\setminus B\). \end{lemme} 

\begin{proof} En vertu de \eqref{eq: derivee du flot} et de l'identité analogue pour \(\W_R\) qui est lisse, ainsi que des relations \eqref{eq: groupe affine}, l'application \(\pi_p\) admet des dérivées partielles par rapport à \(x\) et \(t\) égales à 
\[ \frac{\partial \pi_p (x,t)}{\partial x} = e^{-t} X( \pi_p (x,t))  \quad \text{ et } \quad \frac{\partial \pi_p (x,t)}{\partial t}= \W_R (\pi_p (x,t)) .\] 
Ces dernières étant continues et linéairement indépendantes, les applications \(\pi_p \) sont des immersions de classe \(C^1\). Par construction, on a 
\begin{equation} \label{eq: metrique hyperbolique} (\pi _p )^* \hyp = e^{-2t} dx^2 + dt^2 , \end{equation}
et comme la restriction de \(\hyp\) à \(\FF (p) \setminus B\) est complète, \(\pi_p\) est un revêtement.  
\end{proof} 

Notons que si \( q= \Phi_{W_R} ^{t_0} \circ \Phi_X^{x_0} (p)\) est un autre point dans \(\FF (p) \setminus B\), la relation \eqref{eq: groupe affine} montre que 
\[ \pi_q (x,t) =  \pi _p (e^{t_0} x+ x_0, t+t_0).\]
Les transformations du plan \(\R^2\) définies par 
\begin{equation} \label{eq: transformations affines} (x,t) \mapsto (e^{t_0} x + x_0 , t+t_0) \end{equation} 
forment un groupe qui préserve une triple structure géométrique :  la métrique riemannienne \( e^{-2 t} dx^2 + dt^2 \)  
qui est complète et de courbure constante \(-1\), la structure affine de \(\R^2\) ainsi que le champ de vecteur vertical \(\frac{\partial} {\partial t}\). La famille d'applications \(\pi_p\) confèrent donc aux feuilles de \(\FF\) cette triple structure géométrique.

\begin{proposition} \label{p: orbite periodique}
Les feuilles non simplement connexes de \(\FF\)  sont des anneaux qui contiennent une unique trajectoire périodique de \(\W\) dans l'ensemble hyperbolique \(K\). En particulier, les séparatrices de \( \FF\) sont des anneaux. Réciproquement, toute trajectoire périodique  de \(\W\) dans \(K\)  est contenue dans une feuille annulaire de~\(\FF\). 
\end{proposition}

\begin{proof} 
Comme l'ensemble de Fatou est égal à l'ensemble errant \(D\), dans lequel toutes les feuilles de \(\FF\) sont simplement connexes, une feuille non simplement connexe est contenue dans l'ensemble de Julia. 

Notons que le rayon d'injectivité des feuilles de la restriction de \(\FF\) à l'ensemble de Julia munies de la métrique \(\hyp\) est uniformément minoré. En effet, il est minoré sur toute région compacte et la relation \eqref{eq: groupe affine} permet de voir que ce rayon tend vers l'infini lorsque l'on se rapproche de l'ensemble singulier. 

Soit \(p\) un point régulier de l'ensemble de Julia. La feuille \(\FF(p)\) de \(\FF\) passant par \(p\) ne rencontre pas la surface \(B\). Le lemme \ref{l: revetement} montre que \(\FF(p)\) est isomorphe (munie de sa triple structure géométrique) au quotient de \( \R^2 \) par un sous-groupe du groupe des transformations \eqref{eq: transformations affines} qui agit librement, proprement et  discontin\^ument. Un tel groupe est nécessairement cyclique, engendré par une transformation \( (x,t) \mapsto (e^{t_0} x+ x_0, t+t_0)\). Observons que \( t_0\neq 0\) car le rayon d'injectivité des feuilles pour la métrique \( \hyp \) est uniformément minoré par une constante strictement positive.  Si l'on désigne par \(x_1\in \R\) le point fixe de \( x\mapsto e^{t_0} x+x_0\), la verticale \( x_1 \times \R\) se projette via \(\pi_p\) sur une orbite périodique de \( \Phi_{W_R}\) ; de plus, toute autre trajectoire périodique est de cette forme. Ainsi, toute feuille non simplement connexe est annulaire et contient une unique orbite périodique non constante de \(\W_R\). 

Réciproquement, toute orbite périodique de \( \W_R\) passant par un point \(p\) se relève via le revêtement \(\pi_p\) en la courbe verticale \( 0\times \R\). L'application \(\pi_p\) est donc non injective et la feuille passant par \(p\) est non simplement connexe, donc annulaire comme nous l'avons vu.  

La proposition \eqref{p: orbite periodique} découle alors du fait que  \(\W\) et \(\W_R\) ont les mêmes orbites. \end{proof}

\subsection{Conjecture d'Anosov} \label{ss: conjecture Anosov}

Nous terminons cette partie en établissant que les feuilletages satisfaisant les propriétés \( \Pr\) et \( \Qr\) vérifient la conjecture d'Anosov.

\begin{theoreme}\label{t: conjecture Anosov}
Soit \( \FF\) un feuilletage algébrique de \(\Ptwo\) satisfaisant les propriétés~\( \Pr\) et~\( \Qr\). Alors, toutes les feuilles de \(\FF\) sont simplement connexes, sauf un nombre dénombrable qui sont des anneaux d'holonomie hyperbolique. 
\end{theoreme}

\begin{proof}
Comme l'ensemble hyperbolique \(K\) ne contient qu'un nombre dénombrable d'orbites périodiques pour \(\W\) (corollaire~\ref{cor:nombre-denombrable-orbites-periodiques}), la proposition \ref{p: orbite periodique} montre qu'il n'y a qu'un nombre dénombrable de feuilles non simplement connexes. D'autre part, chaque feuille annulaire contient une orbite périodique non constante du champ \(\W\) qui est contenue dans \(K\). D'après le corollaire \ref{c: decroissance exponentielle des derivees transverses}, l'holonomie de~\(\FF\) le long de cette orbite périodique est un  germe de transformation dilatante, donc hyperbolique.   
\end{proof}


\section{Stabilité structurelle} \label{s: stabilite structurelle} 

Dans cette partie nous démontrons le théorème de stabilité structurelle.

\begin{theoreme} \label{t:stabilite-structurelle-globale}
Soit \(\FF\) un feuilletage algébrique complexe de \(\Ptwo\) de degré \(d\) satisfaisant les propriétés \(\Pr\) et \(\Qr\). Alors, tout feuilletage algébrique complexe de degré \(d\) suffisamment proche de \(\FF\) lui est topologiquement conjugué. 
\end{theoreme} 

Nous utiliserons la topologie sur l'espace des feuilletages algébriques complexes de \(\Ptwo\) définie par le quotient de la topologie naturelle sur l'espace des champs de vecteurs homogènes (non nuls) de degré \(d\) sur \(\C^3\) : deux feuilletages \(\FF\) et \(\FF'\) sont proches s'ils peuvent être définis par des champs de vecteurs homogènes de degré \(d\) sur \(\C^3\) qui sont proches.

\subsection{Stabilité structurelle de l'ensemble de Julia}  \label{ss: stabilite Julia}
Nous établissons que l'ensemble de Julia de \(\FF\) est structurellement stable dans le sens suivant : si \(\FF'\) est un feuilletage algébrique complexe de degré \(d\) suffisamment proche de \(\FF\), son ensemble de Julia est topologiquement conjugué à celui de \( \FF\). 

On rappelle que \(V\) désigne un champ de vecteurs homogène sur \({\bf C}^3 \setminus \{0\}\) de degré \(d\) qui ne s'annule pas, de divergence nulle tel que la projectivisation de \(\FF _V\) est le feuilletage \(\FF\). De plus, si \( V'\) est un autre champ (de degré~\(d\) qui ne s'annule pas et de divergence nulle), on notera \(\FF'\) le feuilletage de \(\Ptwo\) induit par~\(V'\),  \(\W' \) le champ défini au paragraphe \ref{ss: definition de W} qui est associé à \(V'\), \(B'\) l'ensemble des singularités de \(W'\) le long des feuilles de \(\FF'\), etc.

\begin{proposition}\label{p: stabilite structurelle W} Il existe un voisinage de \(V\) dans l'espace des champs de vecteurs homogènes de degré \(d\) qui ne s'annulent pas et de divergence nulle sur~\(\C^3\) tel que pour tout~\(V'\) dans ce voisinage, il existe un homéomorphisme \( \psi : \Ptwo\setminus B \rightarrow \Ptwo\setminus B'\) qui conjugue orbitalement les flots \(\Phi_{\W}\) et \(\Phi_{\W'} \). De plus, \(\psi \) converge vers l'identité (pour la topologie compacte ouverte) lorsque \(V'\) tend vers~\(V\).
\end{proposition} 

\begin{proof} Rappelons que nous avons noté \(U_B\) (resp. \(U_S\)) un voisinage de~\(B\) (resp. de \(\sing\)) dont le bord est transverse à \(\W\) et sur lequel \(\FF\) est un fibré lisse localement trivial en disques au-dessus de \(B\) (resp. un voisinage de linéarisation de \(\sing\), lemmes \ref{l: construction de UB} et \ref{l: construction de US}). Si \(V'\) est choisi suffisamment proche de \(V\), on pourra supposer que \( U_{B'}= U_B\) et que \( U_{S'} = U_S\) ; on aura en particulier les inclusions \( B' \subset U_B \) et \( \sing ' \subset U_S\), et le fait que \(\W '\) est transverse sortant à \(\partial U_S\) et transverse rentrant à \(\partial U_B\).  

Nous allons appliquer le théorème \cite[Theorem C, p. 3]{Robinson} de Robinson pour montrer que la restriction de \(\Phi_{\W}\) à la variété à bord \(M= \Ptwo \setminus \text{Int}( U_B \cup U_S) \) est structurellement stable. Comme le champ \(W\) est transverse à \(\partial M\), il suffit de montrer que l'ensemble récurrent par cha\^{\i}ne est hyperbolique et que les variétés instables fortes de ce dernier intersectent les variétés stables faibles transversalement. Or, ou bien les trajectoires de \(\Phi_{\W}\) dans \(M \) sont contenues dans \(K\), ou bien elles intersectent le bord de \(M\). Donc l'ensemble récurrent par cha\^{\i}ne est contenu dans \(K\) ; il est ainsi hyperbolique d'après la proposition \ref{p: K est hyperbolique}. De plus,  les feuilles stables faibles des \(\W\)-trajectoires sont les feuilles de \(\FF\) puisque \(\FF\) est le feuilletage stable faible pour la métrique \(h\) (\(\mathsection~\ref{ss: feuilletages stables et instables}\)). Quant aux feuilles instables fortes de \(\W\)-trajectoires, ce sont des sections transverses du feuilletage~\(\FF\) (\(\mathsection~\ref{ss: feuilletages stables et instables}\)) : en particulier, l'hypothèse de transversalité est donc bien satisfaite et le flot \(\Phi_{\W} \) est orbitalement structurellement stable sur \(M\). On pourra donc trouver un homéomorphisme \(\psi : M\rightarrow M\) qui envoie les trajectoires de \(\Phi_{\W}\) sur celles de \(\Phi_{\W'}\). On étend \(\psi\) à un homéomorphisme de \(\Ptwostar \) dans \(\Ptwostarprime \) en posant pour tout \(t\geq 0\) 
\[ \psi (\Phi_{\W}^t p) = \Phi_{\W'}^t( \psi(p) ) \quad \text{ si } \ p\in \partial U_B \] 
et 
\[ \psi (\Phi_{\W}^{-t} p) = \Phi_{\W'}^{-t}( \psi(p) ) \quad \text{ si } \ p\in \partial U_S .\]
L'homéomorphisme \(\psi\) défini de cette fa\c{c}on se prolonge par continuité à l'ensemble~\(\sing\) en un homéomorphisme \( \psi : \Ptwo\setminus B \rightarrow \Ptwo \setminus B'\) qui conjugue orbitalement \(\Phi_{\W} \) et \(\Phi_{\W'}\). Cela est d\^u au fait que les singularités de \(\FF\) (resp. de \(\FF'\)) sont des sources pour \(W\) (resp. \(W'\)). 
\end{proof}

\begin{lemme}\label{l: feuilles intersectant K}
Les feuilles de \(\FF\) qui intersectent \(K\) sont exactement les feuilles de l'ensemble de Julia de \( \FF\).
\end{lemme}

 \begin{proof} D'abord, par définition de \(K\), une feuille de l'ensemble de Fatou n'intersecte pas \(K\). D'autre part, soit  \(F\) une feuille de \(\FF\) contenue dans  \(\Julia(\FF)\). Nous allons montrer que \(F\) intersecte \(K\). En effet, si \(F\) ne contient pas de singularité dans son adhérence, alors elle est contenue dans \(K\). Sinon,  \(F\) intersecte le domaine de linéarisation de la singularité et son adhérence contient donc une séparatrice. Cette séparatrice n'étant pas simplement connexe, elle contient une unique orbite périodique \(o\) du flot \(\Phi_{\W} \) (proposition \ref{p: orbite periodique}). La variété instable faible de cette orbite périodique intersecte la feuille \(F\) suivant une orbite de \(\W\) qui converge en temps positif vers \(K\) (car elle est contenue dans \(\Julia (\FF)\)) et en temps négatif vers \(o\subset K\). Cette orbite est nécessairement dans \(K\).\end{proof}

\begin{corollaire}\label{c: stabilite structurelle ensemble de Julia}
L'ensemble de Julia de \(\FF\) est structurellement stable. Plus précisément, l'homéomorphisme \( \psi \) construit à la proposition \ref{p: stabilite structurelle W} induit par restriction un homéomorphisme \( \psi : (\Julia (\FF) , \FF) \rightarrow (\Julia (\FF'), \FF')\).
\end{corollaire}

\begin{proof}
Le saturé de \(\partial U_B\) par \(\Phi_{\W} \) (resp. \(\Phi_{\W'} \)) est l'ensemble de Fatou de \(\FF\) (resp. celui de \(\FF'\)). Comme l'homéomorphisme \(\psi \) envoie \(\partial U_B\) sur lui-même et qu'il conjugue orbitalement \(\Phi_{\W}\) à \(\Phi_{\W'}\), il envoie l'ensemble de Fatou de \(\FF\) sur celui de \(\FF'\) et, par conséquent, l'ensemble de Julia de \(\FF\) sur celui de \(\FF'\).

De même, \(\psi \) envoie \(\Rep  _\W(\sing) \) sur \(\Rep _\W(\sing ' )\). En particulier, il envoie~\(K\) sur~\(K'\). Comme il conjugue orbitalement \(\Phi_\W \) à \(\Phi_{\W'}\) en préservant l'orientation des flots, il envoie également l'ensemble de \(\W\)-attraction de \(K\) (éq. \ref{eq: W repulsion de K}) sur l'ensemble de \(\W'\)-attraction  de \(K'\). D'après le corollaire \ref{c: union varietes stables faibles}, les ensembles  \(\Att _{\W} (K) \) et \(\Att_{\W'}(K') \) sont les unions disjointes des feuilles stables faibles des points de \(K\) et de \(K'\) pour \(\W \) et \(\W'\) respectivement.  Les feuilletages stables faibles de \(\W\) et de \(\W'\) étant les feuilletages  \(\FF\) et \(\FF'\) respectivement d'après la proposition \ref{p: feuilletages stables et instables}, la conjugaison \(\psi \) envoie chaque feuille de \(\FF\) qui intersecte \(K\) sur une feuille de \(\FF'\) qui intersecte \(K'\). Le lemme \ref{l: feuilles intersectant K} achève la démonstation du corollaire. \end{proof}

\subsection{Intermède : les bandes} \label{ss: bandes}

Pour construire une conjugaison globale, nous aurons besoin d'étudier la restriction du feuilletage \(\FF\) (ainsi que des structures géométriques sur ces feuilles induites par les coordonnées \(\pi_p\), éq.~\ref{eq: coordonnees R^2}) à l'ensemble de \(\W\)-répulsion du lieu singulier de \(\FF\) (ou ce qui revient au même, au complémentaire de l'ensemble de \(\W\)-répulsion de l'ensemble hyperbolique \(K\)) : nous appellerons \textit{bande} une feuille de la restriction de \(\FF\) à cet ensemble.

\begin{lemme} \label{l: topologie des bandes}
Soit \(s\in \sing\). Les deux feuilles annulaires de la restriction de \(\FF\) à \(\text{Int} (U_s)\)  (lemme \ref{l: construction de US}) sont contenues dans des bandes annulaires. 
Les autres bandes de \(\Rep_W (s)\) sont simplement connexes et accumulent sur les deux bandes annulaires. L'union des bandes simplement connexes contenues dans \(\Rep_W (s)\) est un ouvert sur lequel le feuilletage est un fibré lisse topologiquement trivial par disques au-dessus de la courbe elliptique \( E_s\).
\end{lemme}

\begin{proof}
Comme les trajectoires de \(\W\) dans \(U_s\) tendent en \(-\infty\) vers un point de \(\sing\) et sortent transversalement à \(\partial U_s\) en temps fini, la restriction du feuilletage \(\FF\) à l'ensemble \(\Rep_W (s)\) est le produit de l'intersection du feuilletage~\( \FF\) avec \(\partial U_s\) avec la droite réelle. Le résultat découle alors du lemme \ref{l: construction de US}.
\end{proof}

Étant donné \(s\in \sing\), nous notons \( \beta ^ \pm (s) \) les deux bandes annulaires contenues dans \(\Rep _W (s)\). Le signe est déterminé par la propriété suivante : l'holonomie de~\(\FF\) le long d'un lacet de \(\beta ^+ (s) \) (resp. \(\beta^-(s) \)) d'indice positif vis-à-vis de \(s\) est dilatante (resp. contractante).

\begin{definition} Soit \(\beta\) une bande et \(p\in \beta\). La composante connexe de \(\pi_p ^{-1} (\beta)\subset \R^2\) contenant le point \((0,0)\) est un ouvert connexe invariant par le champ \(\frac{\partial}{\partial t}\) dans les coordonnées \((x,t)\in \R^2\) (car l'image de ce dernier par la différentielle de \(\pi_p\) est le champ \(\W_R\) qui laisse invariant \(\beta\)) ; elle est donc de la forme \( I_p \times \R\) où \(I_p\) est un invervalle ouvert contenant \(0\). \end{definition}
 
\begin{lemme}\label{l: bandes annulaires}
Étant donné \(s\in \sing\), les bords de \(\beta^\pm(s) \) dans leurs feuilles respectives sont des orbites périodiques du flot \(\W\). De plus, pour tout \(p\in \beta^+(s)\) (resp. \(p\in \beta^-(s) \)), l'intervalle \(I_p \) est de la forme \(]-\infty, x_1[ \) (resp. \(]x_1, +\infty[\)) avec \(x_1\in \R\). 
\end{lemme}

\begin{proof} Le sous-ensemble \(I_p \times \R\subset \R^2\) est invariant par un automorphisme non trivial du revêtement \(\pi_p\) de la forme \( ( e^{t_0} x + x_0  , t+t_0)\) avec \(t_0\) non nul (proposition \ref{p: orbite periodique}). L'intervalle \(I_p\) est alors invariant par  \(x\mapsto e^{t_0} x + x_0\), et ne contient pas son point fixe \(x_1 = x_0/(1- e^{t_0})\), sans quoi la bande contiendrait une orbite périodique de \(\W\) (ce qui est impossible car une telle orbite n'intersecte pas l'ensemble de répulsion de \(\Rep_\W (\sing\)). Il n'y a que deux intervalles non vides de ce type: \(]x_1,+\infty[\) ou \(]-\infty, x_1[\). Le bord de \(\beta^\pm(s)\) dans la feuille de \(\FF\) dans laquelle elle est contenue est la \(\W\)-orbite \( \pi_p (x_1\times \R )\) qui est périodique. Cette orbite périodique est d'holonomie dilatante, donc la bande \(\beta^+(s)\) se situe à gauche de cette orbite et la bande \(\beta^-(s)\) à droite. \end{proof}

\begin{lemme} \label{l: point T(p)} 
Soit \(s\in \sing\). Pour tout point \(p\) appartenant à une bande simplement connexe \(\beta \subset \Rep_W (s)\), l'intervalle \(I_p\) est borné. Il existe un unique point \( \sigma (p )\in \beta\) tel que \( I_{\sigma (p)} = ]-1,1[\). L'application \( p\in \Rep _\W (s) \setminus (\beta^+ (s) \cup \beta^- (s) ) \mapsto \sigma(p) \in \Rep _\W (s) \setminus (\beta^+ (s) \cup \beta^- (s) )\) est continue, constante le long des bandes et son image~\(\Sigma_s\) est une section transverse torique continue de \( \FF\) biholomorphe à la courbe elliptique \(E_s\) (lemme \ref{l: construction de US}).   
\end{lemme} 

\begin{proof} 
Les bandes non annulaires, qui accumulent vers chacune des bandes annulaires \(\beta^+(s)\) et \(\beta^-(s)\) (lemme \ref{l: topologie des bandes}), contiennent dans leur bord une trajectoire de \(W\) située dans la variété instable faible des \(\W\)-orbites périodiques  \(\partial \beta^\pm(s)\) (lemme \ref{l: bandes annulaires}). Ceci montre que \(I_p\) est un intervalle borné et que les applications qui à un point \(p\) appartenant à \(\Rep _\W (s) \setminus (\beta^+ (s) \cup \beta^- (s) ) \) associent les extrémités positives et négatives de l'intervalle \( I_p\) sont continues. L'existence et l'unicité du point \(\sigma(p)\)  découle alors de la formule suivante~: pour tout \( (x, t) \in I_p \times \R\), on a \(I_ {\pi_p (x,t)} = e^{-t} (I_p-x) \). La continuité de \(\sigma\) est une conséquence de la transversalité de la lamination \( \Rep _W (K)\) avec le feuilletage \(\FF\), ainsi que de la continuité de l'action \(\pi\) (lemme \ref{l: action groupe affine}). L'application \(\sigma\) est par construction le quotient de l'ouvert \(\Rep _\W (s) \setminus (\beta^+ (s) \cup \beta^- (s) ) \) par le feuilletage: l'image de \(\sigma\) est donc biholomorphe  à \(E_s\). \end{proof}

\begin{definition} \label{def: section transverse torique}
Soient \( \Sigma_s ^\pm \) les projections de \( \Sigma_s\) sur l'ensemble de répulsion \(\Rep _\W ( K) \) définis par les images respectives des applications 
\[ p  \in \Sigma_s \mapsto \pi_p ( (\pm 1, 0 ) ) \in \Rep _\W ( K)  . \]
On note \( \Sigma\) (resp. \(\Sigma^\pm\)) l'union des \( \Sigma_s\) (resp. \(\Sigma_s^\pm\)) pour \(s\in \sing\). 
\end{definition} 

Une bande simplement connexe \(\beta\) admet deux composantes de bord \(\partial^\pm \beta\) dans la feuille dans laquelle elle est contenue : \(\partial ^+\beta\) est la composante de bord dont  l'orientation co\"{\i}ncide avec celle de \(\W\), \(\partial ^- \beta\) celle dont l'orientation est contraire à celle de \(\W\). Si \(p\in \beta \cap \Sigma\), alors \( \partial ^+ \beta = \pi_p (1 \times \R)\) et \(\partial ^ - \beta = \pi_p (-1\times \R) \). Les points de \( \Sigma^+\) se situent donc sur les bords positifs des bandes simplement connexes et ceux de \(\Sigma^-\) sur les bords négatifs de ces dernières.

\subsection{Stabilité structurelle globale}\label{ss: stabilite structurelle globale}

Dans ce paragraphe nous terminons la démonstration du théorème \ref{t:stabilite-structurelle-globale} en construisant une conjugaison entre \(\FF\) et \(\FF'\) qui co\"{\i}ncide avec n'importe quel homéomorphisme \( l : B\rightarrow B' \) tel que 
\begin{equation} \label{eq: homeomorphisme entre B et B'}  \sup _{p\in B} d(p, l(p)) \end{equation} 
soit suffisamment petit. La stratégie consiste à modifier l'homéomorphisme \(\psi\) construit à la proposition \ref{p: stabilite structurelle W} sur le complémentaire  de l'ensemble \(K\). Cette modification a lieu en deux temps. 

On modifie dans un premier temps \(\psi\) sur l'ensemble \(\Rep _{\W} (K)\) de fa\c{c}on à construire un homéomorphisme \( \Psi : \Rep_{\W} (K) \rightarrow \Rep_{\W'} (K') \).  Puis, dans un second temps, on étend \(\Psi \) au complémentaire de \(\Rep_{\W} (K)\) qui est égal à  l'ensemble de répulsion \(\Rep_W(\sing)\) pour construire un homéomorphisme global de \(\Ptwo\) qui conjugue \(\FF\) à \(\FF'\) via une  formule explicite dans chaque bande. 

\subsubsection{Construction de \(\Psi\)}

\begin{proposition} \label{p: construction de Psi} Il existe un homéomorphisme \( \Psi : \Rep_{\W} (K) \rightarrow \Rep_{\W'} (K') \) satisfaisant les conditions suivantes : 
\begin{enumerate}
\item pour toute feuille \(F\) de \(\FF\), il existe une feuille \(F'\) de \(\FF'\) telle que \[ \Psi  (F\cap \Rep_{\W} (K) ) = F'\cap \Rep_{\W'} (K')\, ;\]
\item pour toute bande simplement connexe \(\beta \) de \( \FF \), il existe une bande simplement connexe \(\beta' \) de \(\FF'\) telle que 
\[\Psi  (\partial ^\pm  \beta ) =\partial ^\pm \beta ' .\]
De plus, il existe une bijection \( s\in \sing\mapsto s' \in \text{sing} (\FF')\) telle que 
\[\Psi (\partial \beta ^{\pm} (s) ) = \partial {\beta '} ^{\pm} (s')\, ; \]
\item il existe des constantes \( a,b>0\) telles que  \( \Psi\) induit une \((a,b)\)-quasi-isométrie entre le revêtement universel d'une \(\W\)-trajectoire et celui de son image par \(\Psi\) (dans les paramétrages donnés par les temps des champs \(\W\) et \(\W'\) respectivement); 
\item il existe une constante \(c >0\) telle que pour tout point \(p\in  \Sigma ^\pm \) (voir définition \ref{def: section transverse torique}) il existe \( p'\in {\Sigma'}^\pm\) et \(t\in \R\) tels que \(|t|\leq c\) et \( \Psi (p) = \Phi_\W ^t (p')\). 
\end{enumerate}
\end{proposition}

\begin{proof} Nous choisirons des perturbations \(V'\) suffisamment petites de \(V\) de fa\c{c}on à ce que l'on puisse construire les voisinages \( U_{B'} \) et \(U_{S'}\) de \( B'\) et \(S'\) respectivement égaux à \(U_B\) et \(U_S\) comme dans la preuve de la proposition \ref{p: stabilite structurelle W}.

Notons \(\pi_B: \partial U_B \rightarrow B\) la projection parallèle au feuilletage \(\FF\). Il s'agit d'un fibré localement trivial en cercles au-dessus de \(B\). L'ensemble \(\Rep_{\W}(K)\) possède une structure de lamination par variétés instables faibles du flot \(\Phi_\W\) qui sont des variétés lisses de dimension \(3\) : cette lamination intersecte le fibré \(\pi_B\) transversalement. Ainsi, pour tout \(b\in B\), si l'on note \( C_b:= \pi_B^{- 1}(b) \cap \Rep_{\W} (K) \), il existe un voisinage \(\mathcal V_b\) de \(b\) dans \(B\)  ainsi qu'un homéomorphisme 
\begin{equation}\label{eq: coordonnees fibration} F_b : \pi _{B} ^{-1}(\mathcal V_b) \rightarrow \pi_B^{-1} (b)\times \mathcal V_b \end{equation}
tel que \( \pi_B = \text{pr}_{\mathcal V_b} \circ F_b\) et \( F_b (\Rep_{\W} (K)\cap \pi_B^{-1} (\mathcal V_b) ) = C_b \times \mathcal V_b\) (les plaques instables faibles étant envoyées sur les verticales \( p\times \mathcal V_b\)). La restriction de \(F_b\) à \( \pi _{B} ^{-1}(\mathcal V_b) \cap \Rep_{\W} (K) \) est unique. 

De fa\c{c}on analogue, nous définissons un fibré en cercles \(\pi_{B'}:\partial U_{B'} \rightarrow B'\),  le point \( b'= l(b)\), l'ensemble \(C'_{b'}:= \pi_{B'}^{-1}(b') \cap \text{Att}_{\W '} (K')\), le voisinage \(\mathcal V '_{b'} := l (\mathcal V_b ) \) de \(b'\in B'\) et l'homéomorphisme \( F'_{b'}: \pi _{B'} ^{-1}(\mathcal V'_{b'})\rightarrow \pi_{B'} ^{-1} (b')\times \mathcal V'_{b'}\).

L'homéomorphisme \(\psi : \Ptwostar \rightarrow \Ptwostarprime\) applique \( \Rep _{\W} (K) \) sur \(\Rep_{\W'} (K')\), en envoyant une feuille instable faible de \(\W\) sur une feuille instable faible de \(\W'\) (corollaire \ref{c: stabilite structurelle ensemble de Julia}). De plus, \(\psi\) envoie \( \partial U_B\) sur \(\partial U_{B'}=\partial U_B\) et converge vers l'identité dans la topologie compacte ouverte lorsque \(V'\) tend vers \(V\). Si le supremum \eqref{eq: homeomorphisme entre B et B'} est suffisamment petit, l'image de la courbe \(\pi_B ^{-1} (b) \) par \(\psi\) est donc contenue dans le voisinage \( \pi _{B'} ^{-1}(\mathcal V'_{b'}) \). Dans les cartes \(F_b\) et \(F'_{b'=l(b)}\), \(\psi\)  prend alors la forme 
\begin{equation}\label{eq: expression de psi en coordonnees} F_{b'} \circ \psi \circ F_b^{-1} (p, c) = (p'(p), c'(p,c) ) ,\end{equation}
sur un voisinage suffisamment petit de \(C_b\times \{b\} \) dans \(C_b \times \mathcal V_b\), tout en s'étendant en un homéomorphisme d'un voisinage de \(\pi_B^{-1} (b) \times \{b\} \) dans \(\pi_B^{-1} (b) \times \mathcal V_b \). Ceci nous permet de définir une transformation
\[ \Psi _0 : \partial U_B \cap \Rep_{\W} (K) \rightarrow \partial U_{B'} \cap \Rep_{\W'} (K') \] 
par 
\begin{equation}\label{eq: expression de Psi en coordonnees} F_{b'} \circ \Psi_0  \circ F_b^{-1} (p, c) = (p'(p), c'=l(c) ) .\end{equation}
Échanger le rôle de \(\psi\) et \(\psi^{-1} \) mène à la construction de l'inverse de \(\Psi_0 \). Ce dernier est donc un homéomorphisme qui vérifie les propriétés suivantes.

\begin{lemme}\label{l: homeo sur le bord de UB}
L'application \(\Psi _0 \) définie par \eqref{eq: expression de Psi en coordonnees} induit un homéomorphisme \[ \partial U_B \cap \Rep_{\W} (K) \rightarrow \partial U_{B'} \cap \Rep_{\W'} (K')\] vérifiant : 
\begin{itemize}
\item pour tout \(p\in \partial U_B \cap \Rep_{\W} (K)\), \(\Psi _0 (p)\) appartient à \( \FIFa _{\W'} (\psi(p))\);
\item pour tout \(b\in B\), \( \Psi _0  ( C_b ) = C' _{l(b)}\);
\item pour tout \(b\in B\), et toute composante connexe \( I\) de \(\pi_B^{-1} (b) \setminus \Rep_{\W} (K)\) d'extrémités positives et négatives \( \partial ^{\pm} I\), il existe une composante connexe \( I' \) de \(\pi_{B'}^{-1} (l(b)) \setminus \Rep_{\W'} (K')\) d'extrémités positives et négatives \( \partial ^{\pm} I'\) telle que 
\[ \Psi _0 ( \partial ^{\pm} I ) = \partial ^{\pm} I'.\]  
\end{itemize}  
\end{lemme}

Dans cet énoncé, il est important de noter que les fibrés en cercles \(\pi_B\) et \(\pi_{B'}\) sont orientées par l'orientation des feuilletages \(\FF\) et \(\FF'\). Il y a donc bien un sens à parler des extrémités positives et négatives d'un intervalle contenu dans une fibre.  

Pour définir l'application \(\Psi : \Rep_{\W}(K) \rightarrow \Rep_{\W '} (K')\), il nous faut introduire le cocycle instable faible sur les feuilles instables faibles du flot \(\W'\) : étant donné deux points \(x,y\in \Rep_{\W'}(K')\) qui appartiennent à la même feuille instable faible du flot \( \Phi_{\W'}\), il existe un unique réel \(c' (x,y)\in \R\) tel que \(\Phi^{c'(x,y)}_{\W'} (x) \in \FIFo_{\W'} (y)\) ou, en d'autres termes, tel que 
\begin{equation} \label{eq: cocycle} d (\Phi^{\tau+c'(x,y)}_{ \W'} (x) , \Phi^{\tau}_{\W'} (y) )  
\xrightarrow[\tau \rightarrow -\infty] {} 0. \end{equation}
Ce cocyle est continu comme fonction de \(x\) et \(y\), du fait de la structure de lamination par variétés instables fortes de l'ensemble de \(\W\)-répulsion de \(K\) (\(\mathsection\) \ref{ss: feuilletages stables et instables}). Notons que la convergence \eqref{eq: cocycle} est exponentielle et uniforme pour \( (x,y)\) dans un compact. En particulier, la fonction \(p\in \Rep_{\W}(K) \cap \partial U_B\mapsto c' (\psi(p), \Psi _0 (p)) \in \R\) est continue.   

Tout point \(q\in  \Fatou(\FF)\cap \Rep_{\W}(K) \) s'écrit de fa\c{c}on unique sous la forme  \( q= \Phi^t_{\W}  (p)\) avec \(p\in \partial U_B\cap \Rep_{\W} (K)\) et \(t\in \R \). Écrivons 
\[ \psi (q) = \Phi ^{\tau} _{\W'} ( \psi(p) ) \]
pour un certain \(\tau = \tau (q) \) qui tend uniformément vers \(-\infty\) (resp. \(+\infty \)) lorsque \(q\) tend vers \(J(\FF)\) (resp. \(B\)). Définissons  \(\Psi\) sur \( \Fatou(\FF)\cap \Rep_{\W} (K)\) par la formule
\begin{equation} \label{eq: definition extension} \Psi (q):= \Phi _{\W'} ^{\tau +c'(\Psi_0(p), \psi(p))}(\Psi _0 (p) ) \end{equation}
Observons que la limite \eqref{eq: cocycle} donne 
\begin{equation} \label{eq: convergence psi Psi} \lim\limits _{q\in (\Fatou(\FF)\cap \Rep_{\W}(K) ) \setminus \text{Int}(U_B) \rightarrow \infty} d( \Psi (q)  , \psi(q) ) =0 .
\end{equation}
Le lemme suivant termine la démonstration de la proposition \ref{p: construction de Psi}.

\begin{lemme} \label{l: dem de la proposition construction de Psi}
La transformation \[ \Psi : \Rep_{\W} (K)  \rightarrow \Rep_{\W'} (K')  \] définie par \eqref{eq: definition extension} sur \( \Fatou(\FF)\cap \Rep_{\W}(K) \) et par \(\psi _{|K}\) sur \(K= \Julia(\FF)\cap \Rep_{\W}(K)\) est un homéomorphisme qui vérifie les propriétés de la conclusion de la proposition~\ref{p: construction de Psi}.
\end{lemme}

\begin{proof}
Les propriétés (1) et (2) de la proposition \ref{p: construction de Psi} découlent immédiatement du lemme \ref{l: homeo sur le bord de UB} en restriction à chaque feuille du domaine de Fatou de~\(\FF\). L'homéomorphisme \(\psi\) envoie les feuilles de l'ensemble de Julia de \(\FF\) sur celles de l'ensemble de Julia de \(\FF '\), et l'ensemble \(\Rep_W (K) \) sur l'ensemble \(\Rep _{W'} (K')\) : en particulier, il envoie les bandes de \(\FF\) contenues dans l'ensemble de Julia de \(\FF\) sur des bandes de \(\FF'\) contenues dans l'ensemble de Julia de \(\FF'\). Comme par définition \(\Psi\) est égal à \(\psi \) sur l'intersection de l'ensemble de Julia de~\(\FF\) avec \( \Rep_W (K)\), les propriétés (1) et (2) de la proposition \ref{p: construction de Psi} sont également satisfaites sur les feuilles de l'ensemble de Julia de \(\FF\).

Montrons que \(\Psi\) est un homéomorphisme. Par construction, c'est un homéomorphisme en restriction à \(\Rep_{\W} (K) \cap \Fatou (\FF) \) (lemme \ref{l: homeo sur le bord de UB}). 
Pour montrer que~\(\Psi\) définit un homéomorphisme globalement, l'unique point délicat est la continuité de \(\Psi\) en un point de l'ensemble de Julia de \(\FF\). Or si \((p_n)_n\) est une suite de points de \( \text{dom} (\Psi) =\Rep_{\W} (K) \) qui converge vers un point \(p_\infty \in \Julia(\FF) \), on peut décomposer la suite \( (p_n)\) en deux sous-suites \( (p_{n_k})_k\) et \( (p_{m_l})_l\) avec \( p_{n_k} \in \Fatou(\FF) \cap \Rep_{\W} (K) \) et \( p_{m_l}\in\Julia(\FF)\). Il est clair que, à supposer que la sous-suite \( (p_{m_l})_l\) soit infinie, on a \(\lim_l \Psi (p_{m_l})=\Psi (p_\infty)\) puisque \(\Psi\) et \(\psi\) co\"{\i}ncident sur \(J(\FF)\) et que \(\psi\) est continue. Quant à la suite \( (p_{n_k})_k\), à supposer qu'elle soit infinie, la limite \eqref{eq: convergence psi Psi} nous donne 
\[ d (\Psi (p_{n_k}), \Psi(p_\infty) = \psi (p_\infty) ) \leq d(\Psi (p_{n_k}), \psi(p_{n_k})) + d(\psi(p_{n_k}), \psi(p_{\infty}))\xrightarrow[k\to \infty] {} 0,\]
ce qui conclut la preuve de la continuité de \(\Psi\) le long de la suite \((p_n)\). 

 Par construction, l'homéomorphisme \(\Psi: \Rep _\W (K) \rightarrow \Rep _{\W'} (K') \) est en dehors d'un compact de \(\Rep_W (K)\) une conjugaison entre les flots induit par \(\W\) et \(\W'\). Le troisième item de la proposition \ref{p: construction de Psi} est donc satisfait par \(\Psi\). 
  
Enfin, en vertu du troisième item du lemme \ref{l: homeo sur le bord de UB} ainsi que de la compacité de \( T^\pm\),  le quatrième item de la proposition \ref{p: construction de Psi} est également satisfait pour \(\Psi\). La démonstration du lemme \ref{l: dem de la proposition construction de Psi} est achevée. \end{proof}

\end{proof} 

\subsubsection{Extension de \(\Psi\) sur \(\Rep_W(\sing)\)}
Nous définissons dans ce paragraphe une application 
\[ \widetilde{\Psi } : \Rep_\W (\sing) \rightarrow  \Rep_{\W'} (\text{sing} (\FF') ) , \]
qui est, en restriction à chaque bande \(\beta\) de \(\FF\), une extension de \(\Psi\). Il y a deux cas :

\vspace{0.3cm}

\textit{Premier cas : \(\beta\) est simplement connexe.}  Considérons la bande \(\beta '\) (pour le champ \(W'\)) associée à \(\beta\) (proposition \ref{p: construction de Psi}). Soit \( p\in \beta \) (resp. \(p'\in \beta'\)) le point tel que \( I_p = ]-1, 1[\) (resp. \( I_{p'} = ]-1, 1[\)) donné par le lemme \ref{l: point T(p)}. 

L'homéomorphisme \(\Psi\) induit un homéomorphisme \( \Psi_ \beta  \) de \(\partial \left([-1,1]  \times \R\right)\) défini par 
\[ \pi_ {p'}  \circ \Psi _\beta  = \Psi  \circ \pi_ p .\]
Étant donné \(t\in \R\), considérons les points \( P^\pm   =  (\pm 1, t) \in \partial \left([-1,1]  \times \R\right)\) et \( Q^\pm := \Psi_\beta (P^\pm)\in \partial \left([-1,1]  \times \R\right)\). On subdivise l'intervalle orienté \( [P^- , P^+] = [-1,1] \times \{t\} \subset \R^2\) en trois intervalles \( I^-, I^m, I^+\), ordonnés par ordre croissant, avec \(I^-\) et \( I^+\) de longueur 
\begin{equation} \label{eq: longueurs} l(I^\pm) := \frac{1}{3} \inf ( l([P^- , P^+] ) , l ([Q^- , Q^+]), \end{equation}
où \([P^- , P^+]\) et \([Q^-, Q^+]\) sont les segments affines entre les points \(P^\pm\) et \(Q^\pm\) respectivement vis-à-vis de la structure affine naturelle de \( \R^2\) et la longueur \(l\) étant mesurée relativement à la métrique hyperbolique \( e^{-2t} dx^2 + dt^2\).

Définissons l'extension \(\widetilde{\Psi}_\beta\) de \(\Psi_\beta\) à \([-1,1]\times \R\)  par les conditions suivantes. Pour tout \(t\in \R\) : 
\begin{itemize}
\item l'image par \(\Psi_\beta\) du segment \( [P^-, P^+]\) est le segment affine \([Q^-, Q^+]\);
\item la restriction de \(\Psi_\beta\) aux segments \(I^\pm\) respecte la longueur d'arc (pour la métrique hyperbolique \( e^{-2t} dx^2 +dt^2\));
\item la restriction de \(\Psi_\beta\)  au segment \(I^m\) dilate les longueurs par multiplication par une certaine constante (qui dépend de \(t\)).
\end{itemize}  
On définit l'extension \(\widetilde{\Psi} \) de \(\Psi \) à \(\beta\) via la formule
\begin{equation} \label{eq: extension Psi} \pi_ {p'}  \circ \widetilde{\Psi} _{\beta} ^\pm = \widetilde{\Psi} \circ \pi_ p .\end{equation}

\vspace{0.3cm}

\textit{Deuxième cas : \(\beta\) n'est pas simplement connexe.} Supposons que \(\beta\) soit une bande annulaire positive. Il lui correspond alors une bande positive \( \beta '\) de \(\FF'\) (voir le deuxième item de la proposition \ref{p: construction de Psi}). Soient  \(p\in \beta\), \(p'\in \beta'\). En reprenant les notations du lemme \ref{l: bandes annulaires}, les applications \(\pi_p, \pi_{p'} \) induisent des homéomorphismes  \( (e^{t_0} x+x_0, t+t_0)\backslash (]-\infty, x_1[ \times \R)  \rightarrow \beta\) et \(  (e^{t_0'} x+x_0', t+t_0')\backslash  (]-\infty, x_1 '[ \times \R)  \rightarrow \beta'\). La restriction de l'homéomorphisme \(\Psi \) à \(\partial \beta\) induit donc un homéomorphisme \( \Psi_\beta :  (e^{t_0} x+x_0, t+t_0)\backslash (x_1 \times \R)  \rightarrow (e^{t_0'} x+x_0', t+t_0')\backslash (x_1' \times \R) \) satisfaisant
\[ \Psi \circ \pi_p = \pi_{p'} \circ \Psi_\beta .\] 
L'application \( \Psi_\beta\) se relève à un homéomorphisme de \(x_1 \times \R\) de la forme \(x_1 \times h_\beta\) où \(h_\beta\) est un homéomorphisme de \(\R\). On étend alors \(\Psi_ \beta\) à l'homéomorphisme \(\widetilde{\Psi} _\beta :   (e^{t_0} x+x_0, t+t_0)\backslash (]-\infty, x_1] \times \R)  \rightarrow (e^{t_0'} x+x_0', t+t_0')\backslash (]-\infty, x_1'] \times \R) \) défini par 
\[ \widetilde{ \Psi} _\beta ( x , t )\text{ mod }   (e^{t_0} x+x_0, t+t_0) := (x_1 ' + e^{h_\beta (t) - t}(x-x_1 ') , h_\beta (t) )  \text{ mod }   (e^{t_0'} x+x_0', t+t_0') \]
Il s'agit du prolongement de \( \Psi_\beta\) qui envoie le le feuilletage horizontal de \( ]-\infty, x_1] \times \R\) sur celui de \( ]-\infty, x_1'] \times \R\) en préservant la longueur d'arc sur les feuilles de ce dernier relativement à la métrique hyperbolique \( e^{-2t} dx^2 + dt^2\). Enfin, on définit l'extension \(\widetilde{\Psi} \) de \(\Psi \) à \(\beta\) via la formule \eqref{eq: extension Psi}. Si \(\beta\) est une bande annulaire négative, on procède de fa\c{c}on analogue.

\begin{lemme}\label{toto}
\(\widetilde{\Psi} \) est continue. 
\end{lemme} 

\begin{proof} 
Le lemme \ref{l: point T(p)}  ainsi que la construction de \(\widetilde{\Psi}\) dans les bandes simplement connexes montrent que \(\widetilde{\Psi} \) est continue en dehors des bandes annulaires. Elle est de plus continue en restriction à chaque bande annulaire. Il nous suffit donc de montrer qu'elle admet une limite le long de toute suite de points \(p_n\) appartenant à des bandes simplement connexes \(\beta_n\) et convergeant vers un point~\( p \) appartenant à une bande annulaire \(\beta \) et que cette limite est égale à \(\widetilde{\Psi} (p)\). 

Nous supposerons que \(\beta \) est une bande annulaire positive, le cas d'une bande négative se traitant de fa\c{c}on similaire. 
Soient \(\pi_n:= \pi'_{T(p_n)}: ]-1,1[ \times \R \rightarrow \beta_n\) et \(\pi_n ' = \pi_{T'(p_n')} : ]-1,1[ \times \R\rightarrow \beta_n' \) (voir lemme \ref{l: point T(p)}). Notons \( P_n= (x_n, t_n )=\pi_n^{-1} ( p_n) \)  les coordonnées des points \( p_n \) et introduisons les points 
\[ Q_n = \widetilde{\Psi}_{\beta_n} (P_n)=(x_n', t_n') ,  \ \  P_n ^\pm = (\pm 1,t_n) \text{ et } Q_n^\pm = \widetilde{\Psi}_{\beta_n} (P_n^\pm)  =(\pm 1, {t_n'} ^\pm) ,\]
et les points correspondant dans \(\Ptwo\)
\[ p_n ^+ = \pi_n (P_n^+), \ \ p^+ = \pi (P^+) \ \text{ et } \ q_n^\pm = \pi _n ' (Q_n ^\pm) . \]
Enfin, si \( I_p = ] -\infty, l[\) pour un certain nombre \(l>0\) (car \(\beta\) est une bande annulaire positive), notons \(p_+ = \pi_p  (l,0) \). 

On a \( \pi _{p_n} (x,t)= ( e^{-t_n} (x-x_n) , t - t_n) \), et donc \( I_{p_n} = ]-  e^{-t_n} (1+x_n) , e^{-t_n} (1-x_n)[\). Comme \(I_{p_n} \) converge vers \( I_p\), on obtient  
\[ t_n   \xrightarrow[n\rightarrow \infty] {} -\infty
\quad \text{et} \quad 
e^{-t_n} (1-x_n)  \xrightarrow[n\rightarrow \infty] {}  l .\]   

Les troisième et quatrième items de la proposition \ref{p: construction de Psi} montrent que, en posant \(d= b+c\) :  
\begin{equation} \label{eq: tnprime}        a^{-1} t_ n - d \leq  {t_n ' } ^\pm \leq  a t_n + d .\end{equation}
On en déduit que les longueurs des segments  \([P_n^-, P_n ^+]\) et \([Q_n^-, Q_n ^+] \) (mesurées avec la métrique hyperbolique \( e^{-2t} dx^2 + dt^2\)) tendent vers l'infini lorsque \(n\) tend vers l'infini, et que \( Q_n \) est le point situé sur le segment \([Q_n^- , Q_n ^+]\) à distance \( l_n = l ([P_n, P_n^+]) \) car cette dernière tend vers \(l\) et reste donc bornée. 

On a, par définition de l'extension \(\widetilde{\psi}\) : \( \psi (p_n^+) = q_n^+\), par suite \( q_n^+ \) converge vers \( \psi (p^+)\). D'autre part, la longueur du segment \([q_n, q_n^+]\) converge vers \(l\). Pour conclure, il nous suffit donc de montrer que l'angle entre le segment \([q_n^-, q_n^+]\) et le bord droit de \(\beta_n '\) tend vers \(\pi /2\) pour la métrique hyperbolique \( e^{-2t} dx^2 + dt^2\). Pour cela, pla\c{c}ons-nous dans les coordonnées données par \( \pi'_{q_n^+}\)~: les points \(q_n^+\) et \(q_n^- \) ont alors pour coordonnées respectives \((\pi'_{q_n^+})^{-1}(q_n^+) =(0,0)\) et \( (\pi'_{q_n^+})^{-1}(q_n^-)  =(-2 e^{-{t_n'}^+}, {t_n'}^- - {t_n'} ^+ ) \) et le bord gauche de \(\beta_n '\) est l'axe vertical \((\pi'_{q_n^+})^{-1} (\partial ^+ \beta_n' ) = \{0\} \times \R\subset \R^2\).  On a donc bien le résultat en vertu de l'estimation \eqref{eq: tnprime} qui donne
   \[ {t_n'}^- - {t_n'} ^+ = o ( e^{-{t_n'}^+} ) .\]
\end{proof} 


\begin{lemme}
L'application \( \Ptwostar \rightarrow \Ptwostarprime \) définie par \(\Psi \) dans \( \Rep_W (K) \) et par \(\widetilde{\Psi} \) dans \( \Rep_\W (\sing) \)  s'étend en un homéomorphisme de \(\Ptwo\) dans lui-même qui conjugue topologiquement les feuilletages \(\FF\) et \(\FF'\).
\end{lemme}

\begin{proof}
Considérons une suite \(p_n\) de points de \(\Ptwostar \) qui converge vers un point \(p_\infty\) appartenant à \( \Rep_W (K)\). Pour voir que \( \Psi (p_n) \) converge vers \(\Psi(p_\infty)\), il suffit de supposer que \( p_n \notin \Rep _W (K)\) ou, de fa\c{c}on équivalente, que \( p_n \in \Rep_W(\sing) \). Comme \( \Rep _W (K)\) a la structure d'une lamination de dimension \(3\) réelle  transverse au feuilletage \(\FF\), on pourra trouver une suite de points \(q_n\), qui sont des extrémités du segment horizontal  contenant \(p_n\) dans sa bande et tels que \( d(p_n, q_n) \rightarrow 0\). La suite \(q_n\) converge donc vers \(p_\infty\). Par définition de \(\Psi\) sur \(\Rep_W(K)\), on déduit \(d (\Psi (p_n), \Psi(q_n) ) \rightarrow 0\) puis \(\lim \Psi (p_n) =\lim \Psi (q_n) = \Psi (p_\infty)\). La continuité de \(\Psi\) est ainsi démontrée.

Tout homéomorphisme de \(\Ptwostar \) dans \(\Ptwostarprime \) qui envoie les feuilles de \(\FF\) sur celles de \(\FF'\) s'étend en un homéomorphisme de \( \Ptwo\) qui conjugue topologiquement les feuilletages~\(\FF\) et~\(\FF'\).
\end{proof}

\section{Étude des sections transverses} \label{s: sections transverses}

Étant donné un feuilletage algébrique complexe \(\FF\) de \(\Ptwo\), on rappelle qu'une section transverse de \(\FF\) est une surface réelle \( Z\subset \Ptwo\setminus \sing\) de classe~\(C^1\) et transverse à \(\FF\) en tout point. Une telle surface hérite naturellement des structures transverses du feuilletage \(\FF\) et, en particulier, d'une structure de surface de Riemann. Dans la suite, on munit les sections transverses de l'orientation induite par leur structure de surface de Riemann. 

\subsection{Des sections transverses toriques}
Commen\c{c}ons par donner des exem\-ples de sections transverses dans les voisinages arbitrairement petits des singularités hyperboliques. 

\begin{lemme}\label{l: section transverse torique singularite}
Étant donné un feuilletage algébrique \(\FF\) de \(\Ptwo\),  au voisinage de toute singularité hyperbolique de \(\FF\),  il existe une section transverse arbitrairement proche de la singularité qui est biholomorphe à la courbe elliptique \(\C/ ( \Z \lambda + \Z \mu) \), où \(\lambda\) et \(\mu\) sont les valeurs propres d'un champ de vecteurs définissant \(\FF\). 
\end{lemme}

\begin{proof}
On reprend les notations de la démonstration du lemme \ref{l: construction de US}. La restriction de l'intégrale première \eqref{eq: integrale premiere singularite} à la surface torique d'équation \(T=\{|u|= |v|= r_1 \}\) avec \(0<r_1<r\) induit un biholomorphisme entre  \(T\) et \(\C/ ( \Z \lambda + \Z \mu) \). \end{proof}

Une généralisation de cette méthode permet de construire une section transverse torique dans le voisinage de n'importe quel lacet  d'holonomie hyperbolique. Rappelons qu'un lacet \(\gamma: \S ^1\rightarrow F\) de classe \( C^1\) contenu dans une feuille d'un feuilletage algébrique complexe \(\FF\) sur \(\Ptwo\) est dit d'holonomie hyperbolique si la dérivée de l'application d'holonomie associée à \(\gamma\) est un nombre complexe de module différent de \(1\). 

\begin{lemme}\label{l: section transverse torique lacet hyperbolique}
Étant donné un lacet d'holonomie hyperbolique \(\gamma\) dans une feuille d'un feuilletage transversalement holomorphe de codimension complexe un, il existe une section transverse torique dans un voisinage arbitrairement petit autour de \(\gamma\). 
\end{lemme}

\begin{proof}[Démonstration] 
Quitte à considérer  une sous-variété de dimension \(3\) transverse au feuilletage et contenant \(\gamma\), on peut supposer que la variété  est  de dimension réelle~\(3\), et que le feuilletage est de dimension réelle \(1\) et transversalement holomorphe. Il est donc défini par un champ de vecteurs \(X\) non singulier, au moins dans un voisinage suffisamment petit de \(\gamma\).  Quitte à changer l'orientation de ce dernier, on peut supposer également que l'holonomie de \(\gamma\) dans le sens défini par~\(X\) est contractante. Dans ce cas, il est bien connu qu'il existe une métrique adaptée \(g_N\)  sur le fibré normal \(\NRF\) de \(\FF\) qui est strictement contractée par le flot induit par \(X\) en tout temps strictement positif. En considérant une métrique sur~\(M\) qui induit la métrique \(g_N\) sur \(\NRF\), pour \(\varepsilon >0\) suffisamment petit, la surface constituée des points à distance \(\varepsilon\) de \(\gamma\) est transverse à \(\FF\). \end{proof}


\subsection{Caractéristique d'Euler}

Notons que \( H_2(\Ptwo,\Z)\) est cyclique infini engendré par la classe d'homologie d'une droite projective complexe \(\PUC\): ainsi, pour toute section transverse~\(Z\)  compacte, il existe un entier \( d_Z\) tel que la classe d'homologie \([Z]\) de \( Z\) soit \( [Z]= d_Z [\PUC]\). Cet entier est appelé le degré de \( Z\). 

\begin{lemme} \label{l: caracteristique section transverse}
Étant donné un feuilletage algébrique complexe $\FF$  de degré $d$ de $\Ptwo$, toute section transverse compacte $Z$ est de degré $d_Z=0 $ ou $d_Z= 1-d$. Si~\(Z\) a une composante connexe de caractéristique d'Euler non nulle, alors celle-ci est unique et de genre \(\frac{d(d+1)} {2}\). 
\end{lemme}

\begin{proof} [Démonstration.]  Cela résulte des isomorphismes naturels
\[ T^\R Z \simeq ( \NRF ) _{|Z} \quad \text{ et } \quad N^\R Z \simeq (\TRF ) _{|Z} ,\]
ce qui donne, étant donné que dans la partie régulière on a \( \NRF \simeq \NF\) et \(\TRF \simeq \TF\) 
\[ \chi (Z) = [\NF] \cdot [Z]  \quad \text{ et } \quad    [Z]^2 = [\TF] \cdot [Z].\]
Dans ces formules, \(\chi\) est la caractéristique d'Euler, les crochets désignent les classes de Chern des fibrés en droites complexes correspondant et on note \(\cdot\) le produit d'intersection, ou la dualité selon le point de vue.
Comme $\NF = \mathcal{O} (d+2)$ et $\TF = \mathcal{O} (1-d)$ pour un feuilletage de degré $d$, on obtient les formules 
\[  \chi (Z)= (d+2) d_Z \quad \text{   et   } \quad d_Z ^2 = (1-d) d_Z. \]
On a donc \(d_Z= 0\) ou \(d_Z=1-d\). Par ailleurs, ces formules montrent qu'une composante connexe de \(Z\) de caractéristique  d'Euler non nulle a un degré non nul. Puisque la forme d'intersection sur \(H_2(\Ptwo,\Z)\) est non dégénérée, une telle composante est unique. \end{proof}

Le degré d'une courbe de \(\Ptwo\) étant strictement positif, une section transverse compacte n'est donc jamais une courbe algébrique de \(\Ptwo\), sauf si le degré du feuilletage est nul. 

 Certains feuilletages algébriques complexes de $\Ptwo$ ne possèdent pas de section transverse compacte: c'est par exemple le cas des pinceaux de Lefschetz (autres que les pinceaux de droites). En effet, si pour un tel pinceau  $f : \Ptwo\dashrightarrow \PUC $ il existait une section transverse $Z$, la restriction de $f$ à $Z$ serait un revêtement de $Z$ sur $\PUC $, ce qui montrerait que $Z$ est une union disjointe de sphères et contredirait le lemme~\ref{l: caracteristique section transverse}. 

Pour d'autres feuilletages algébriques complexes du plan projectif, il existe des sections transverses toriques mais pas de section transverse de caractéristique d'Euler non nulle: c'est par exemple le cas des feuilletages qui admettent une courbe algébrique invariante et qui ont quelque part un lacet d'holonomie hyperbolique. 

\begin{lemme}\label{l: incompatibilite sections transverses courbes algebriques}
Soit \(\FF\) un feuilletage algébrique de degré \(d \geq 1\) du plan projectif complexe ayant une courbe algébrique invariante \(A\). Alors toute section transverse compacte de \(\FF\) est un tore qui ne rencontre pas \(A\). 
\end{lemme}

\begin{proof}[Démonstration] Soit \(Z\) une section transverse de \(\FF\). Comme les indices d'intersection d'une section transverse avec une feuille sont égaux à \(1\), on a \(d_A \cdot d_Z = [A]\cdot [Z] = |A\cap Z|\). Le degré d'une courbe algébrique étant strictement positif, on en déduit que \(d_Z\) est positif ou nul, et est nul si et seulement si \(A\) n'intersecte pas \(Z\). Or nous avons vu que la caractéristique d'Euler d'une section transverse est soit nulle soit strictement négative et que, dans ce dernier cas, son degré est strictement négatif (voir la démonstration du lemme \ref{l: caracteristique section transverse}). \end{proof}

Un problème  intéressant serait d'étudier les feuilletages algébriques de $\Ptwo$ admettant une section transverse compacte de caractéristique d'Euler non nulle: il nous semble plausible qu'une condition nécessaire et suffisante soit l'absence d'une courbe algébrique invariante (le lemme \ref{l: incompatibilite sections transverses courbes algebriques} montre qu'il s'agit d'une condition nécessaire). 

\subsection{Connexité}

\begin{proposition}\label{p: connexite}
Soit \(\FF\) un feuilletage algébrique complexe de \(\Ptwo\) de degré~\(d\) satisfaisant les propriétés \(\Pr\) et \(\Qr\). Alors la section transverse \(B\) formée par l'ensemble des points critiques de \( W\) le long des feuilles régulières de \(\FF\) est une surface de Riemann compacte connexe de genre \(\frac{d(d+1)}{2} \). 
\end{proposition}

\begin{proof} [Démonstration]
D'après le lemme \ref{l: caracteristique section transverse}, il suffit d'établir que \(B\) n'a pas de composante connexe de genre \(1\). Par l'absurde, supposons que ce soit le cas et notons \(C\subset B\) une telle composante ; \(C\) est alors une courbe elliptique. Le saturé de \(C\) par \(\FF\) est un ouvert que l'on notera \( D_C\). 

\begin{lemme}\label{l: construction d'une section transverse torique appropriee}
Il existe une section transverse torique \(T\) qui intersecte à la fois \( D_C\) et \(D_C^c\). 
\end{lemme}

\begin{proof} [Démonstration]
La frontière \(\partial D_C\) de \( D_C\) est saturée par les feuilles de \(\FF\) et n'est pas contenue dans l'ensemble singulier de \(\FF\). En effet, si tel était le cas, \(\partial D_C\) étant  connexe, il serait réduit à une singularité \(s\) de \(\FF\). Il en résulterait que les deux  séparatrices de \(s\) seraient contenues dans \(D_C\). Ces dernières seraient alors simplement connexes d'après le théorème \ref{t: existence domaine errant} donc des courbes rationnelles passant par la singularité. Or le feuilletage ne peut posséder de courbe algébrique invariante puisqu'il admet une section transverse de caractéristique d'Euler non nulle (lemme \ref{l: incompatibilite sections transverses courbes algebriques}). 

Considérons alors l'unique ensemble pseudo-minimal \( M\subset\Ptwo\) contenu dans  \(\partial D_C\). Il s'agit d'un ensemble fermé, saturé par le feuilletage, non réduit à un sous-ensemble de singularités et minimal pour ces propriétés. En particulier, toute feuille \(F\) de \(\FF\) contenue dans \(M\) est dense dans \(M\). Un tel ensemble existe, il suffit de considérer un ensemble minimal pour la restriction du feuilletage à \( \Ptwo \setminus \text{Int}(U_S)\) qui est contenu dans \(\partial D_C\) (ce dernier intersecte le complémentaire de \(U_S\) car il est saturé par \(\FF\) et n'est pas contenu dans l'ensemble singulier) puis de saturer l'ensemble obtenu par \(\FF\).  Comme \(M\) n'est pas réduit à une courbe algébrique invariante, un point régulier de \(M\) n'est jamais transversalement isolé dans \(M\). L'unicité résulte du fait que le complémentaire d'un ensemble pseudo-minimal est un ouvert de Stein \cite{Takeuchi} mais nous n'utiliserons pas ce point.

Dans le cas où \(M\) contient une singularité \(s\) de \(\sing\), la section transverse torique construite au lemme \ref{l: section transverse torique singularite} convient. En effet, puisque \(M\) ne contient pas de point transversalement isolé dans sa partie régulière, l'intersection de \(M\) avec une petite boule autour de \(s\) ne peut être réduit aux deux séparatrices de \(\FF\) en \(s\), et donc intersecte \(T\). Aussi, l'ouvert  \(D_C\) accumule sur \(s\) et, en particulier, intersecte \(T\). Comme \(M\) et \(D_C\) sont disjoints, cela prouve bien que \( T\) intersecte à la fois \(D_C\) et \(D_C^c\). 

Supposons maintenant que \( M\) ne contient aucune singularité de \(\FF\). Dans ce cas, un théorème de Bonatti, Langevin et Moussu \cite{BLM}  (voir également \cite{Deroin Kleptsyn} ainsi que le théorème \ref{t: conjecture Anosov} qui offrent des démonstrations alternatives)  montre qu'il existe un lacet d'holonomie hyperbolique \(\gamma\) contenu dans \(M\). Dans ce cas, la section transverse torique associée à \(\gamma\), dont la construction est expliquée au lemme \ref{l: section transverse torique lacet hyperbolique}, convient. En effet, \(T\) intersecte \(D_C\) puisque \(D_C\) est ouvert et accumule sur \(\gamma\). De plus,  \(T\) intersecte \(M\), donc \(D_C\) puisque \(\gamma\) n'est pas transversalement isolé dans~\(M\). 
\end{proof}

Nous sommes maintenant en mesure d'achever la démonstration de la proposition \ref{p: connexite}. Soit \(T\) la section transverse torique construite au lemme \ref{l: construction d'une section transverse torique appropriee} et soit \(U\)  l'ouvert de \(T\) défini par \( U := T \cap D_C\). Ce dernier est un ouvert non vide strict de \(T\). Considérons l'application \( P : D_C \rightarrow C\) qui, à un point de \(D_C\), associe l'unique point de \(C\) situé dans la même feuille de \(\FF\). Cette application est bien définie en vertu du théorème \ref{t: existence domaine errant} et sa restriction à toute section transverse est un biholomorphisme local. La restriction \(P_U\) de \(P\) à \(U\) est donc un biholomorphisme local. 
Or la norme de la dérivée de \(P_U\) (mesurée vis-à-vis de la métrique hermitienne \(g_N\) construite dans le parapraphe \ref{ss: contraction transverse}) tend vers l'infini lorsqu'elle est évaluée en un point qui tend vers \(\partial U\). Cela découle du corollaire \ref{c: decroissance exponentielle des derivees transverses}. En parallélisant les courbes elliptiques \(T\) et \(C\), la dérivée  de \(P_U\) définit une fonction holomorphe \( P'_U : U \rightarrow \C^* \) dont la norme tend vers \(+\infty\) lorsque l'on tend vers \(\partial U\). Or ceci est impossible en vertu du principe du maximum appliqué à la fonction \(1/ P'_U\). 
\end{proof}



\section{Démonstration du théorème principal} \label{s: numerique} 

Nous allons déduire la stabilité structurelle du feuilletage de Jouanolou \(\FJ2\) en montrant qu'il satisfait les propriétés \textbf{\(\Qr\)} et \textbf{\(\Pr\)} (théorème \ref{t:stabilite-structurelle-globale}) puis nous allons montrer que \(B\) est biholomorphe à  la quartique de Klein. 

\subsection{Le groupe de symétrie de \(\FJ2\)}

Le groupe \(\text{Aut} (\FJ2)\) est le groupe des transformations projectives qui préservent \(\FJ2\). Il  a été calculé par Jouanolou dans \cite{Jouanolou} : il s'agit du groupe engendré par les transformations 
\begin{equation} \label{eq: symetries s et t}  s = [y:z:x] \quad \text{ et } \quad t = [ \zeta x : \zeta^2 y : \zeta ^4 z ] \end{equation} 
où \(\zeta\) est une racine \(7\)-ième primitive de l'unité. La structure algébrique du groupe  \(\text{Aut} (\FJ2)\) peut être complètement décrite par les relations entre \(s\) et \(t\) : toutes se déduisent des trois relations suivantes 
\[ s^3 = 1, \quad t^7 = 1, \quad sts^{-1} = t^2 .\]   
On en déduit que \(\text{Aut} (\FJ2)\) se relève à un sous-groupe du groupe unitaire \( \text{U}(3) \subset \text{GL} (3,\C)\), via le morphisme \begin{equation} \label{eq: relevement} m : \text{Aut} (\FJ2)\rightarrow \text{U} (3) \end{equation}  
qui à \(s\) associe \( m(s) = (y,z,x)\) et à \(t\) associe \( m(t) = (\zeta x , \zeta^2 y , \zeta ^4 z)\). 

\begin{lemme}\label{l: preservation de W et B}
Le groupe \(\text{Aut} (\FJ2)\) préserve le champ \(\W\) et la courbe mixte \(B\).
\end{lemme}

\begin{proof}
Cela est d\^u au fait que les transformations \( m(s) \) et \(m(t)\) préservent la classe projective du champ de Jouanolou \(\J2\) ainsi que la forme hermitienne standard sur \(\C^3\) et par suite la fonction \( f(p) := - \log \norm{p}^{2} \).
\end{proof}

\subsection{La propriété \textbf{\(\Qr\)}}

\begin{lemme} \label{l: les singularites de FJ2 sont des sources}
Le feuilletage de Jouanolou \(\FJ2\) vérifie la propriété \textbf{\(\Qr\)}. \end{lemme}

\begin{proof}[Démonstration]
Le point $[1:1:1]$ est une singularité de $\J2$ en laquelle $\J2 (p)=p$.  Dans la carte affine $\{z\neq 0\}$, en notant \(u,v\) les coordonnées définies par $x=u z$ et $y= vz$, le feuilletage est donné par le champ de vecteurs 
\[ X =  (v^2 - u^{3}) \frac{\partial} {\partial u} + (1- u^2 v) \frac{\partial} {\partial v} \]
dont les valeurs propres sont $\lambda _\pm = - 2 \pm i \sqrt{3}$. Elles sont linéairement indépendantes sur \(\R\) et ont chacune une partie réelle strictement négative. Ainsi  la singularité \([1:1:1]\) est hyperbolique et, d'après le lemme \ref{l: partie lineaire W} et sa démonstration, le champ $\W$ est égal à  \( \frac{-2}{3} \Re (X) \)  à l'ordre \(1\).  En particulier, $[1:1:1]$ est une source pour \(W\). 
Le lemme \ref{l: preservation de W et B} nous permet de déduire cette propriété aux autres singularités de \(\FJ2\) en utilisant le groupe de symétries de l'équation de Jouanolou. \end{proof}

\subsection{La propriété \textbf{\(\Pr\)}}

\begin{proposition}
\label{transversality-Btwotwo}
La courbe mixte \(B\) (corollaire \ref{eq: section transverse}) d'équation 
\begin{equation}\label{eq: B Jouanolou2}  x\bar{y}^2 + y\bar{z}^2 + z\bar{x}^2 = 0 \end{equation}
est transverse à \(\FJ2\). 
De plus, le feuilletage de Jouanolou \(\FJ2\) vérifie la propriété \textbf{\(\Pr\)}, c'est-à-dire 
\[[x: y: z] \in B \ \Longrightarrow \ 2 |\bar{x}yz^2 + \bar{y}z x^2 + \bar{z}x y^2| < |x|^4 + |y|^4 + |z|^4.\]
\end{proposition}

Le problème étant invariant par la symétrie~\(s\) donnée par \eqref{eq: symetries s et t}, on peut supposer~\(z\) de module maximal.
Il suffit dans ce cas de démontrer l'inégalité ci-dessus en tout point \((x,y,z)\) satisfaisant l'équation \(x\bar{y}^2 + y\bar{z}^2 + z\bar{x}^2 = 0\) et vérifiant \(|z| \neq 0\), ainsi que les deux inégalités \(|x| \leq |z|\) et \(|y| \leq |z|\).
Posons alors \(u = x / z\) et \(v = y / z\) et considérons les deux fonctions
\[F(u,v) = u\bar{v}^2 + v +\bar{u}^2 \quad \text{et} \quad G(u,v) = \frac{\bar{u} v+\bar{v} u^2 + uv^2 } {|u|^4 +|v|^4 + 1}.\]
En notant~\(\D\) le disque unité de~\(\C\), la proposition~\ref{transversality-Btwotwo} suit du lemme suivant.

\begin{lemme}
\label{transversality-proposition-bi-disk}
Pour tout \(p\) de \(\D \times \D\), si \(F(p) = 0\) alors \(|G(p)| < 1/2\).
\end{lemme}

Nous allons à présent voir comment ramener le lemme~\ref{transversality-proposition-bi-disk} à la vérification d'un nombre fini d'inéquations sur les entiers.

Pour tout entier naturel \(N\) non nul, on désigne par \(\Gamma_N = \frac{1}{N} \Z[i]\), où \(\Z[i]\) est l'anneau des entiers de Gauss et par \(\D_N = \rho_N \D\) le disque de rayon \(\rho_N = 1+ \frac{1}{N}\).
On note également
\[C_F(N) = \sup_{\D_N \times \D_N} \norm{\text{d}F} \qquad \text{et} \qquad C_G(N) = \sup_{\D_N \times \D_N} \norm{\text{d}G}\]
les bornes supérieures sur le bi-disque \(\D_N \times \D_N\) (muni de sa structure hermitienne standard) des normes d'opérateurs des différentielles \(\text{d}F\) et \(\text{d}G\).
On définit enfin l'ensemble
\begin{equation}
\mathbb{E}_N = \left\{ (U, V)\in \Z[i]^2 \ ; \
\left\{
    \begin{array}{lll}
         |U| \leq N + 1 \ \ \text{et} \ \ |V| \leq N + 1 \\
         |U\bar{V}^2 +N^2 V+N \bar{U}^2| \leq  N^2 C_F(N)\\
    \end{array}
\right. \right\}
\end{equation}
et la condition \(\mathbb{C}_N\)
\begin{equation}
\forall Q \in \mathbb{E}_N, \ \ N |N\bar{U} V+ \bar{V} U^2 + UV^2 | < \left(\frac{1}{2} - \frac{C_G(N)}{N}\right) \left(|U|^4 + |V|^4 + N^4 \right).
\end{equation} 

\begin{lemme}
\label{condition-CN-implies-transversality}
S'il existe un entier naturel \(N\) non nul tel que la condition~\(\mathbb{C}_N\) soit satisfaite, alors le lemme~\ref{transversality-proposition-bi-disk} est également satisfait.
\end{lemme}

\begin{proof}
Soit \(N\) un entier naturel non nul tel que la condition~\(\mathbb{C}_N\) soit satisfaite.
Tout nombre complexe \(w\) est à une distance inférieure ou égale à \(1/\sqrt{2}N\) de \(\Gamma_N\) puisque \(Nw\) est à distance au plus de \(1/\sqrt{2}\) d'un entier de Gauss.
Soit \(p = (u,v)\) un point de \(\D \times \D\) tel que \(F(p) = 0\) et \(q\) un point de \(\Gamma_N \times \Gamma_N\) dont chacune des coordonnées est à une distance inférieure ou égale à \(1/\sqrt{2}N\) de celles de \(p\).
On a donc \(\norm{p-q}_2 \leq 1/N\) et, d'après le théorème des accroissements finis,
\[|F(q)| = |F(q)-F(p)| \leq \frac{C_F(N)}{N}.\] 
En notant \(Q = Nq = (U,V) \in \Z[i]^2\), l'inégalité précédente se réécrit
\[\lvert U \bar{V}^2 + N^2 V + N \bar{U}^2 \rvert \leq N^2 C_F(N).\]
Le point~\(Q\) appartient donc à~\(\mathbb{E}_N\) puisque, d'après l'inégalité triangulaire, \(|U|\) et \(|V|\) sont inférieurs ou égaux à \(N + 1/\sqrt{2}\), donc à \(N+1\).
On déduit alors de la condition~\(\mathbb{C}_N\) que 
\[|G(q)| < \frac{1}{2} - \frac{C_G(N)}{N}\]
et, toujours d'après le théorème des accroissements finis,
\[|G(p)| \leq |G(q)| + \frac{C_G(N)}{N} < \frac{1}{2}.\]
\end{proof}

Nous allons à présent donner des estimations des constantes~\(C_F(N)\) et~\(C_G(N)\).
Pour cela, en tout point \((u,v)\) du bi-disque \(\D_N \times \D_N\), la différentielle de \(F\) est donnée par
\[\text{d}F(u,v) =\bar{v}^2 \text{d}u +2\bar{u}\text{d}\bar{u} + \text{d}v + 2u\bar{v}\text{d}\bar{v},\]
ce qui fournit l'estimation suivante
\[|\text{d}F(u,v)| \leq 3\rho_N^2 \left(|\text{d}u|+|\text{d}v|\right) \leq 3\rho_N^2 \sqrt{2} \sqrt{|\text{d}u|^2+|\text{d}v|^2},\]
d'où finalement \(C_F(N) \leq 3\sqrt{2}\rho_N^2\).
Notons~\(P\) et~\(Q\) respectivement le numérateur et le dénominateur de la fonction~\(G = P / Q\); on a alors \(\text{d}G = \frac{\text{d}P}{Q} - \frac{P\text{d}Q}{Q^2}\).
Les différentielles de~\(P\) et~\(Q\) sont données par
\[
\left\{
\begin{array}{r @{\ = \ } l}
    \text{d}P(u,v) & (v^2+2u\bar{v}) \text{d}u +v\text{d}\bar{u} + (\bar{u} +2uv)\text{d}v+u^2\text{d}\bar{v} \\
    \text{d}Q(u,v) & 2u\bar{u}^2 \text{d}u +2u^2 \bar{u}\text{d}\bar{u} +2v\bar{v}^2 \text{d}v +2v^2 \bar{v}\text{d}\bar{v}, \\
\end{array}
\right.
\]
d'où les estimations
\[|\text{d}P(u,v)| \leq 4 \rho_N^2\sqrt{2} \sqrt{|\text{d}u|^2+|\text{d}v|^2} \quad \text{et} \quad |\text{d}Q(u,v)| \leq 4\rho_N^3\sqrt{2} \sqrt{|\text{d}u|^2+|\text{d}v|^2}.\]
On en déduit alors
\[|\text{d}G(u,v)| \leq 4\sqrt{2}\rho_N^2 + 12\sqrt{2}\rho_N^6 \quad \text{et donc} \quad C_G(N) \leq 4\sqrt{2}\rho_N^2(1+3\rho_N^4).\]

\vspace{0.1cm}

Dans une dernière étape, nous allons ramener la démonstration du lemme~\ref{transversality-proposition-bi-disk}, et par suite celle de la proposition~\ref{transversality-Btwotwo}, à des calculs \textit{uniquement} sur des entiers naturels.
Pour cela, nous définissons l'ensemble
\begin{equation}
\mathbf{E}_N = \left\{ (U, V)\in \Z[i]^2 \ ; \
\left\{
    \begin{array}{lll}
         |U|^2\leq (N+1)^2 \ \ \text{et} \ \ |V|^2\leq (N+1)^2 \\
         |U\bar{V}^2 +N^2 V+N \bar{U}^2|^2 \leq  18 (N+1)^4\\
    \end{array}
\right. \right\}
\end{equation}
et la condition
\begin{equation}
\mathbf{C}_N = \forall Q \in \mathbf{E}_N, \ 10 N^2 |N\bar{U} V+ \bar{V} U^2 + UV^2 |^2 < \left(|U|^4 + |V|^4 + N^4\right)^2.
\end{equation} 

\begin{lemme}
S'il existe un entier naturel \(N \geq 54\) tel que la condition~\(\mathbf{C}_N\) soit satisfaite, alors la condition~\(\mathbb{C}_N\) est également satisfaite.
\end{lemme}

\begin{proof}
Soit \(N\) un entier naturel non nul tel que la condition~\(\mathbf{C}_N\) soit satisfaite.
Soit \(Q = (U, V)\) un élément de l'ensemble~\(\mathbb{E}_N\).
Nous avons calculé, à la suite de la démonstration du lemme~\ref{condition-CN-implies-transversality}, la majoration \(C_F(N) \leq 3\sqrt{2}\rho_N^2\) de laquelle on déduit que \((N^2 C_F(N))^2 \leq 18 (N+1)^4\), et finalement que \(Q\) appartient à l'ensemble~\(\mathbf{E}_N\).
On déduit alors de la condition~\(\mathbf{C}_N\) que
\[|N\bar{U} V+ \bar{V} U^2 + UV^2 | < \frac{1}{\sqrt{10}N} \left(N^4 + |U|^4 +|V|^4\right).\]
En utilisant à présent la majoration \(C_G(N) \leq 6 \rho_N^2 (1 + 3\rho_N^4)\) calculée précédemment et le fait que le membre de droite de cette inégalité est une fonction décroissante de~\(N\) tendant vers~24, on laisse au lecteur le soin de vérifier que
\[\frac{1}{\sqrt{10}N} \leq \frac{1}{2} - \frac{C_G(N)}{N},\]
dès lors que \(N\) est au moins égal à~54, ce qui termine la démonstration.
\end{proof}

\begin{lemme}
La condition \(\mathbf{C}_{145}\) est vérifiée.
\end{lemme}

\begin{proof}
La démonstration du lemme se réduit à vérifier un ensemble fini d'inégalités sur les entiers naturels.
\end{proof}

Pour conclure cette section, nous indiquons quelques données numériques à propos de la condition~\(\mathbf{C}_{145}\) et de la condition~\(\mathbf{C}_{150}\) que nous avons également vérifiée.
Rappelons qu'il s'agit d'un ensemble d'inégalités de la forme \(\iota(U,V) \leq \kappa(U,V)\).

\vspace{0.2cm}

\noindent
\textbf{Condition \(\mathbf{C}_{145}:\)}
\begin{itemize}
\item nombres d'inégalités à vérifier \(= 3 \ 329 \ 227\)
\item \(\text{max}_{\mathbf{E}_{145}} \{\iota(U,V)\} = 371 \ 887 \ 603 \ 636 \ 416 \ 250 \ (< 2^{59})\)
\item \(\text{max}_{\mathbf{E}_{145}} \{\kappa(U,V)\} = 1 \ 815 \ 000 \ 825 \ 762 \ 712 \ 225 \ (< 2^{61})\)
\item \(\text{min}_{\mathbf{E}_{145}} \{\kappa(U,V) - \iota(U,V)\} = 22 \ 039 \ 448 \ 963 \ 015 \ 524 \ (< 2^{55})\)
\item \(\text{max}_{\mathbf{E}_{145}} \left\{ \frac{\iota(U,V)}{\kappa(U,V)} \right\} \simeq 0.935442548319427\)
\end{itemize}
Les inégalités étant indépendantes les unes des autres, il est possible de distribuer le calcul et effectuer les vérifications en quelques dizaines de secondes. 

\vspace{0.2cm}

\noindent
\textbf{Condition \(\mathbf{C}_{150}:\)}
\begin{itemize}
\item nombres d'inégalités à vérifier \(= 3 \ 558 \ 612\)
\item \(\text{max}_{\mathbf{E}_{150}} \{\iota(U,V)\} = 486 \ 516 \ 872 \ 673 \ 000 \ 000 \ (< 2^{59})\)
\item \(\text{max}_{\mathbf{E}_{150}} \{\kappa(U,V)\} = 2 \ 385 \ 954 \ 719 \ 278 \ 502 \ 464 \ (< 2^{62})\)
\item \(\text{min}_{\mathbf{E}_{150}} \{\kappa(U,V) - \iota(U,V)\} = 29 \ 998 \ 373 \ 501 \ 278 \ 096 \ (< 2^{55})\)
\item \(\text{max}_{\mathbf{E}_{150}} \left\{ \frac{\iota(U,V)}{\kappa(U,V)} \right\} \simeq 0.932602375568296\)
\end{itemize}

\noindent
\textit{Remarque:}
Notons que tous les entiers naturels en jeu ici sont strictement inférieurs à~\(2^{64}\) et qu'il est donc possible d'implémenter un algorithme de vérification de ces inégalités sur une architecture informatique classique de~64 bits dans un langage de programmation standard.
En particulier, le recours à des langages de programmation permettant d'écrire des programmes manipulant des \og grands entiers \fg\ n'est pas nécessaire ici.

\subsection{La section transverse \(B\) est biholomorphe à la quartique de Klein} 


\begin{proposition} \label{p: caracterisation de la quartique de Klein} Toute surface de Riemann de genre $3$ qui contient une copie de $\text{Aut} (\FJ2)$ dans son groupe d'automorphismes est biholomorphe à la quartique de Klein
\begin{equation}\label{eq: quartique de Klein} xy^3+yz^3+zx^3=0. \end{equation}
 \end{proposition}

\begin{proof} 
Notons \(C\) une telle surface de Riemann et soit $i : \text{Aut} (\FJ2)\rightarrow \text{Aut} (C)$ un morphisme injectif. Soit \( L =\Omega^1 (C) ^*\) et \( c: C \rightarrow \mathbb P (L ) \) l'application canonique qui, à un point \(p\in C\), associe la classe projective de l'évaluation d'une \(1\)-forme holomorphe en un vecteur tangent non nul en \(p\).  Cette application est équivariante vis-à-vis d'une représentation linéaire \(\rho : \text{Aut} (C) \rightarrow \text{GL} (L)\). 

\begin{lemme}\label{l: injectivite}  \(\rho\) est injective.\end{lemme}

\begin{proof}
Il est bien connu que \(c\) est un plongement si \(C\) n'est pas hyperelliptique et un revêtement ramifié double sur une conique de \(\mathbb P ( L )\) si c'est une courbe hyperelliptique (\cite{ACGH}). En particulier, étant donné que le groupe \(\text{Aut} (\FJ2)\) ne possède aucun élément d'ordre \(2\), la représentation linéaire \(\rho\) est injective. 
\end{proof}

\begin{lemme}\label{l: conjugaison}  \( \rho \) est conjuguée à la représentation \(\gamma\circ m\) où \(\gamma\) est l'automorphisme de \(\text{GL} (3,\C)\) induit par un automorphisme de Galois.\end{lemme}

\begin{proof} La relation \( sts^{-1} = t^2\) montre que l'application \(\rho (s^{-1})\) permute les espaces propres de \(\rho(t)\) en envoyant \( \text{Ker} (\rho (t) - \lambda)\) sur \( \text{Ker} (\rho (t) - \lambda ^2) \). Ainsi, le spectre de \(\rho (t)\) est formé de trois valeurs propres distinctes de la forme \( \lambda , \lambda^2 , \lambda ^4\) pour une certaine racine \(7\)-ième de l'unité \(\lambda\). Dans certaines coordonnées de \(L\), on a donc 
\[\rho (s) = (  y,z,x) \text{ et } \rho (t) = (\lambda x , \lambda^2 y, \lambda^4 z) .\]
Ainsi \(\rho\) est conjugué à la représentation induite par l'automorphisme de Galois qui envoie \(\zeta\) sur \(\lambda\).
\end{proof}

\begin{lemme} \label{l: quartique} Les équations quartiques projectivement invariantes par le groupe \(\rho \circ i (\text{Aut} (\FJ2) ) \) sont dans certaines coordonnées sur \(L\) données par  
\begin{equation}\label{eq: equation quartique de Klein}  xy^3 + y z^3 +  z x^3 = 0 .\end{equation}
 \end{lemme}
 
\begin{proof} 
À cause du lemme \ref{l: conjugaison}, il suffit de montrer que les équations quartiques projectivement invariantes par le groupe \(m(\text{Aut} (\FJ2))\) sont à multiplication par une constante près de la forme  
\begin{equation} \label{eq: equation P} x y^3 + \eta y z^3 + \eta^2 z x^3 =0 \end{equation}
où $\eta$ est une racine cubique primitive de l'unité. En effet, elles sont toutes équivalentes à \eqref{eq: equation quartique de Klein} par la transformation linéaire \((x, \eta y, \eta^2 z)\).

Soit \(P\in \C [x,y,z] \) un polynôme non nul homogène de degré~4 projectivement invariant par le groupe \(m( \text{Aut} (\FJ2)) \), c'est-à-dire que \( P\circ m(g) = u(g) P \) pour un certain morphisme \( u  : \text{Aut} (\FJ2) \rightarrow \C^*\). Comme \(t\) appartient au groupe dérivé de \(\text{Aut} (\FJ2)\), on a \( u(t) =1\). 
Décomposons le polynôme \(P\) en une combinaison linéaire de monômes \(P = \sum_{\alpha, \beta, \gamma} P_{\alpha, \beta, \gamma} x^\alpha y^\beta z^\gamma \), où \(\alpha, \beta, \gamma\) décrivent toutes les partitions de quatre.   Comme \(P \circ m(t) = P\), on déduit que \( P_{\alpha,\beta,\gamma} = 0 \) dès que \( \alpha + 2 \beta + 4\gamma \) n'est pas un multiple de \(7\). De plus, comme \( P\circ m(s) = u(s) P \) avec \(u(s) \neq 0\), on déduit que l'annulation de \( P_{\alpha,\beta, \gamma}\) est invariante par permutation cyclique de \(\alpha , \beta , \gamma\). Ces deux observations montrent que les seuls monômes apparaissant dans la décomposition de \(P\) sont associés aux partitions \( (\alpha, \beta, \gamma) \in \{ (1,3,0) , (0, 1, 3), (3,0,1)\} \), donc \(P\) est projectivement équivalent à \eqref{eq: equation P} avec \( c= u(s)\).  
\end{proof}  
 
Nous sommes maintenant en mesure de terminer la démonstration de la proposition~\ref{p: caracterisation de la quartique de Klein}. En effet, l'image de \(C\) par l'application canonique est une courbe (comptée avec multiplicité) de degré~4 dont l'équation quartique est projectivement invariante par \(\rho \circ i (\text{Aut} (\FJ2))\). Dans certaines coordonnées, cette équation est \eqref{eq: equation quartique de Klein} d'après le lemme \ref{l: quartique}, ce qui montre que \( c(C)\) est la quartique de Klein \eqref{eq: quartique de Klein} et que \(c\) induit un biholomorphisme entre \(C\) et cette dernière.  \end{proof}

\begin{corollaire} \label{c: quartique de Klein}
Le quotient de l'ensemble de Fatou \(\Fatou(\FJ2)\) par \(\FJ2\)  est biholomorphe à la quartique de Klein.
\end{corollaire}

\begin{proof} [Démonstration]
Ce quotient est biholomorphe à la courbe mixte \(B\) d'équation \eqref{eq: B Jouanolou2} munie de la structure complexe induite par \(\FJ2\). Comme \(B\)  est une surface analytique réelle lisse non vide et non holomorphe, elle est Zariski dense sur \(\C\) dans \(\Ptwo\) et par suite le groupe \(\text{Aut} (\FJ2) \) agit fidèlement sur \(B\). Son action sur \(B\) étant holomorphe, le corollaire est une conséquence de la proposition \ref{p: caracterisation de la quartique de Klein}. 
\end{proof}

\section{Expérimentations numériques et images} 
\label{sec:images-exp-num}

\subsection{Stabilité structurelle en degré \(d>2\)} 

Par un calcul analogue à celui détaillé dans le lemme~\ref{l: les singularites de FJ2 sont des sources}, on démontre que le feuilletage de Jouanolou de degré \(d > 2\) vérifie la propriété~\textbf{\(\Qr\)}.
Quant à la propriété~\textbf{\(\Pr\)}, les choses sont plus délicates puisque celle-ci dépend du choix d'une fonction mesurant la distance infinitésimale entre les feuilles.
La fonction \(-\log ||\cdot||\) ne convient pas pour \(d \geq 3\) mais nous avons pu vérifier numériquement pour les degrés \(3 \leq d \leq 5\) que d'autres fonctions conviennent, en l'occurrence les fonctions \(-\log ||\cdot||_p\) pour des~\(p\) bien choisis dépendant de~\(d\).
Plus précisément, les dessins de la figure~\ref{img:transversality} montrent que, pour \(2 \leq d \leq 5\), il existe un intervalle de normes \(\ell^p\) pour lesquelles l'inégalité de transversalité est satisfaite.
Sur chaque dessin de la figure~\ref{img:transversality}, on lit en abscisse la valeur de \(p\) et en ordonnée le maximum sur la surface \(B_p\) (analogue de l'équation~\ref{eq: section transverse} pour la fonction \(-\log ||\cdot||_p\)) du rapport \(R_p={\left|\frac{\partial^2||\cdot||^p_p}{\partial t^2}\right| }/{\frac{\partial^2||\cdot||^p_p}{\partial t \partial \bar{t}}}\).
D'après la démonstration du lemme~\ref{l: singularites longitudinales de \W}, la propriété~\textbf{\(\Pr\)} est donc satisfaite pour la fonction \(-\log ||\cdot||_p\) si et seulement si ce rapport~\(R_p\) est strictement inférieur à~1.

La proposition~\ref{eq: B Jouanolou2} donne une démonstration formelle de la propriété~\textbf{\(\Pr\)} pour le couple degré-norme \((d=2, p=2)\) et on constate numériquement que les couples \((d=3,p=3)\), \((d=4,p=4)\) et \((d=5,p=5)\) fonctionnent.
Ce sont ces expérimentations numériques qui nous permettent de conjecturer la stabilité structurelle du feuilletage de Jouanolou en degré \(d > 2\).

\begin{figure}[ht]
	\centering
	\begin{subfigure}{0.45\textwidth}
		\centering
		\includegraphics[width=\textwidth]{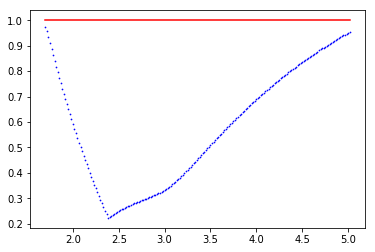}
		\caption{\(d=2\)}
		\label{img:transversality-d2}
	\end{subfigure}
	\begin{subfigure}{0.45\textwidth}
		\centering
		\includegraphics[width=\textwidth]{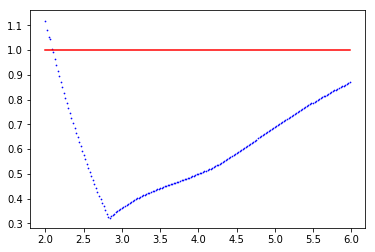}
		\caption{\(d=3\)}
		\label{img:transversality-d3}
	\end{subfigure}
	\begin{subfigure}{0.45\textwidth}
		\centering
		\includegraphics[width=\textwidth]{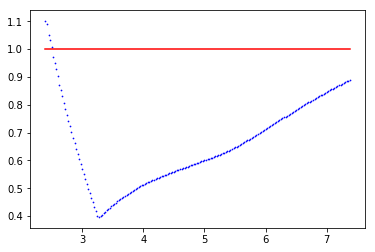}
		\caption{\(d=4\)}
		\label{img:transversality-d4}
	\end{subfigure}
	\begin{subfigure}{0.45\textwidth}
		\centering
		\includegraphics[width=\textwidth]{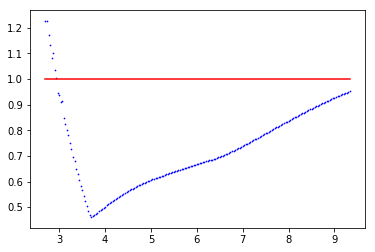}
		\caption{\(d=5\)}
		\label{img:transversality-d5}
	\end{subfigure}
	\caption{Maximum de~\(R_p\) sur la surface~\(B_p\) en fonction de~\(p\)}
	\label{img:transversality}
\end{figure}

\subsection{Ensemble de Julia transversalement Cantor} 

Les dessins de la figure~\ref{img:deg} (page~\pageref{img:deg}) montrent l'intersection de l'ensemble de Julia du feuilletage de Jouanolou de degré \(2 \leq d \leq 5\) avec une sphère bordant le voisinage de l'une des singularités et invariante par la symétrie~\(s\) d'ordre~3 (éq.~\ref{eq: symetries s et t}).
Nous avons produit des dessins analogues jusqu'en degré~\(d=9\) (ainsi que des visualisations en 3D pour en apprécier plus finement les détails).
Ce sont ces images et d'autres travaux en cours qui nous permettent de conjecturer que l'ensemble de Julia est transversalement un ensemble de Cantor en tout degré \(d \geq 2\).

\begin{figure}[ht]
	\centering
	\begin{subfigure}{0.45\textwidth}
		\centering
		\includegraphics[width=\textwidth]{d2.pdf}
		\caption{\(d=2\)}
		\label{img:deg2}
	\end{subfigure}
	\begin{subfigure}{0.45\textwidth}
		\centering
		\includegraphics[width=\textwidth]{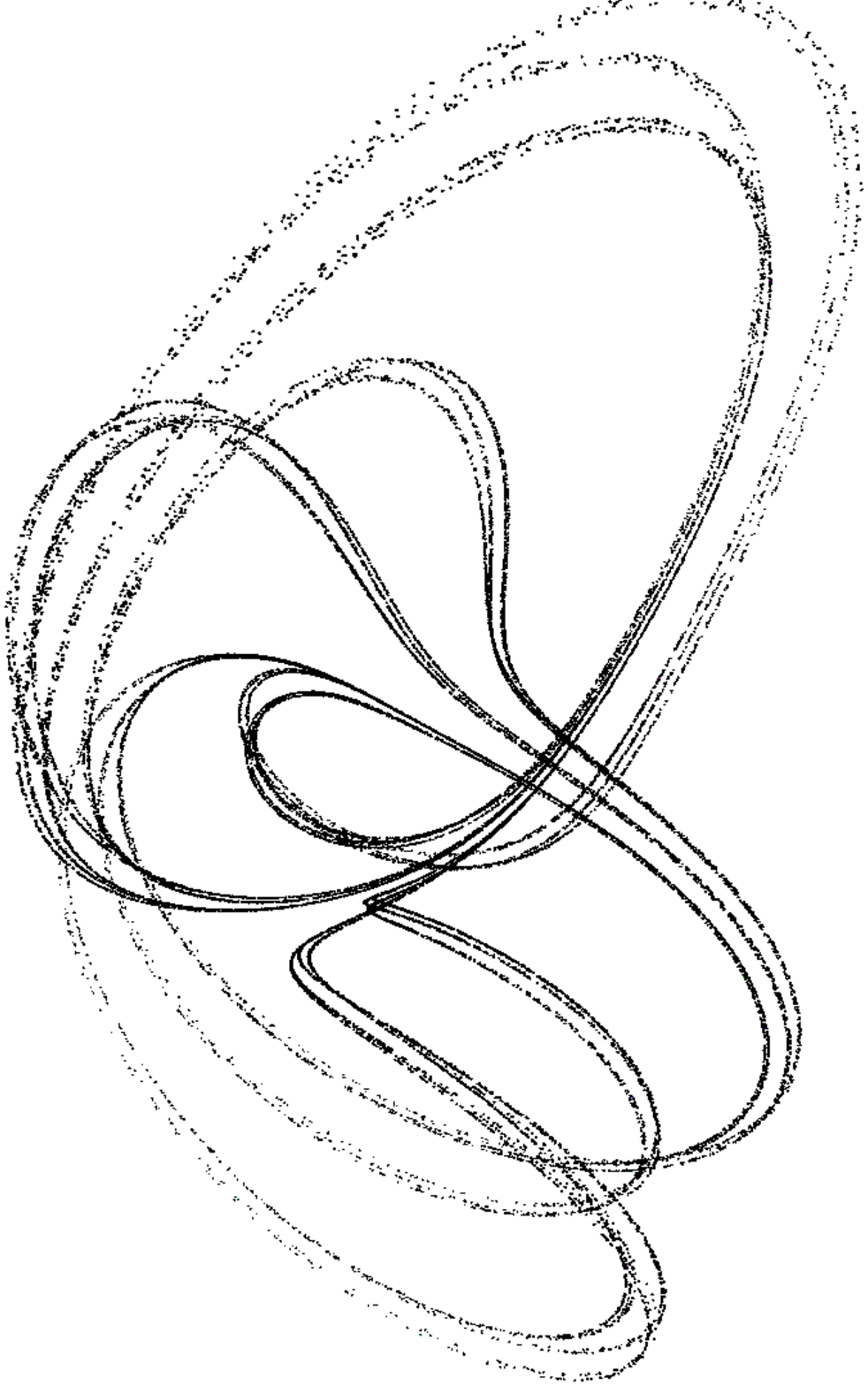}
		\caption{\(d=3\)}
		\label{img:deg3}
	\end{subfigure}
	\begin{subfigure}{0.45\textwidth}
		\centering
		\includegraphics[width=\textwidth]{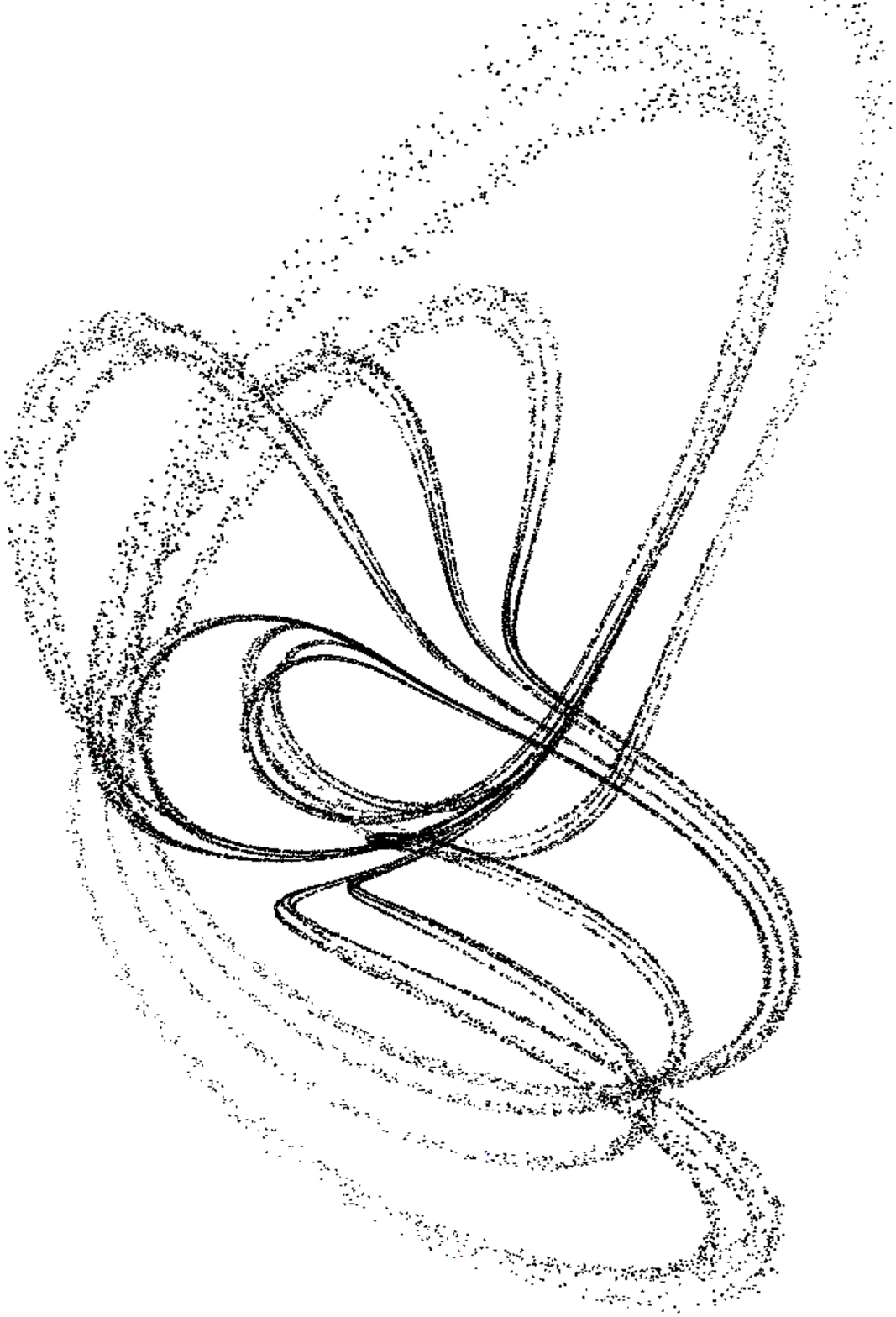}
		\caption{\(d=4\)}
		\label{img:deg4}
	\end{subfigure}
	\begin{subfigure}{0.45\textwidth}
		\centering
		\includegraphics[width=\textwidth]{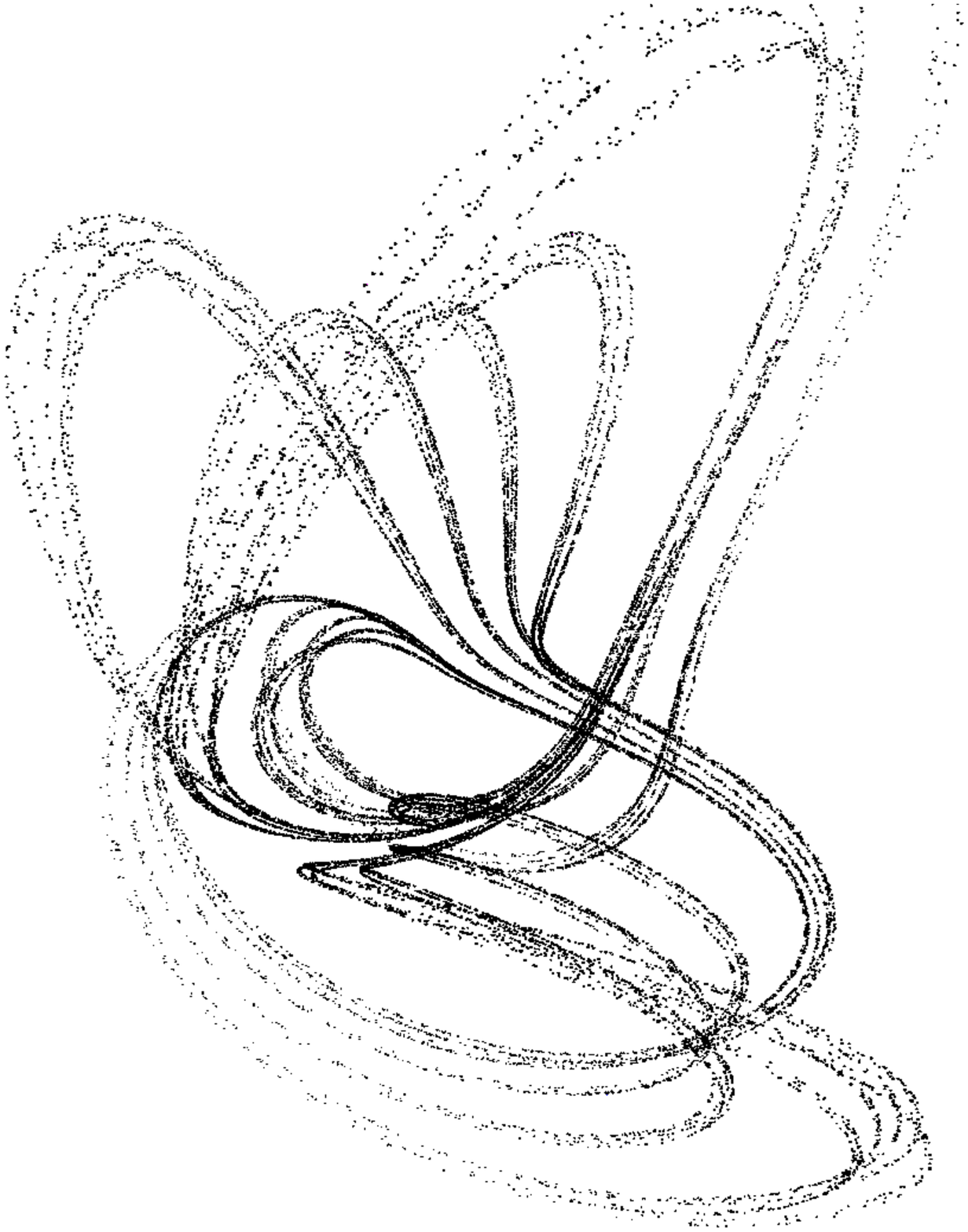}
		\caption{\(d=5\)}
		\label{img:deg5}
	\end{subfigure}
		\caption{Intersection de l'ensemble de Julia du feuilletage de Jouanolou de degré~\(d\) avec une sphère entourant une singularité}
		\label{img:deg}
\end{figure}



\begin{thebibliography}{GGMS}


\bibitem{ACGH}
E. Arbarello, M. Cornalba, P.A. Griffiths, J.E. Harris.
\newblock{\em Geometry of algebraic curves}. 
\newline
\newblock{Springer} (1984).

 \bibitem{asuke} 
 T. Asuke.
 \newblock{\em A Fatou-Julia decomposition of transversally holomorphic foliations.}
 \newline
 \newblock {Ann. Inst. Fourier} (2010), 60-3, p.~1057-1104.

\bibitem{Bogomolov}
F. Bogomolov.
\newblock{\em Complex manifolds and algebraic foliations.}
\newline
\newblock {RIMS-1084} (1996).

\bibitem{BLM} 
C. Bonatti, R. Langevin, R. Moussu.
\newblock {\em Feuilletages de \( \PnC \) : de l'holonomie hyperbolique pour les minimaux exceptionnels.}
\newline 
\newblock {Publ. Math. IHES} (1992), 75-1, p.~123-134.

\bibitem{Brunella}
M. Brunella.
\newblock{\em Birational geometry of foliations.}
\newline
\newblock{Springer} IMPA monograph (2015).

\bibitem{CdF}
C. Camacho, L.H. de Figueiredo. 
\newblock{\em  The dynamics of the Jouanolou foliation on the complex projective $2$-space.}
\newline
\newblock {Ergod. Th. Dynam. Syst.} (2001), 21, p.~757-766.

\bibitem{Cerveau}
D. Cerveau.
\newblock {\em Densité des feuilles de certaines équations de Pfaff à 2 variables. }
\newline 
\newblock {Ann. Inst. Fourier} (1983), 33-1, p.~185-194. 

\bibitem{Deroin Guillot}
B. Deroin, A. Guillot.
\newblock {\em Foliated affine and projective structures.}
\newline 
\newblock {Compositio Math.} À paraître.

\bibitem{Deroin Kleptsyn} 
B. Deroin, V. Kleptsyn. 
\newblock{\em Random Conformal Dynamical Systems.}
\newline
\newblock{Geometric and Functional Analysis} (2007), 17-4, p.~1043-1105.

\bibitem{FH}
T. Fisher, B. Hasselblatt.
\newblock{\em Hyperbolic flows.}
\newline 
\newblock{EMS Zurich Lectures in Advanced Mathematics} (2020) 25.

\bibitem{Ghys}
É. Ghys.
\newblock{\em Sur les groupes engendrés par des difféomorphismes proches de l'identité.}
\newline
\newblock{Bol. Soc. Brasil. Mat.} (1993), 24-2, p.~137-178. 

\bibitem{GGMS}
É.~Ghys, X.~Gomez-Mont, J.~Saludes.
\newblock{\em Fatou and Julia components of transversely holomorphic foliations.}
\newline
\newblock {Essays on geometry and related topics}. Monogr. Enseign. Math. 38 (2001), 1-2, p.~287-319.

\bibitem{HV} 
M.O. Hudaj-Verenov. 
\newblock {\em A property of the solutions of a differential equation.}
\newline
\newblock{ Math. USSR Sb.} (1962), 56-3, p.~301-308 (en russe).

\bibitem{Ilyashenko}
  Y. Ilyashenko. 
  \newblock{\em Global and local aspects of the theory of complex differential equations. }
  \newline 
  \newblock{Proc. Int. Cong. Math. Helsinski} (1978), p.~821-828.

\bibitem{Jouanolou}
J.-P. Jouanolou.
\newblock{\em Équations de Pfaff algébriques.} 
\newline
\newblock{Lecture Notes in Mathematics} (1979), Springer, 708.

\bibitem{Haefliger}
A. Haefliger.
 \newblock {\em Foliations and compactly generated pseudogroups} 
 \newline 
 \newblock {Foliations: geometry and dynamics (2000)}, World Sci. Publ. (2002), p.~275-295.

\bibitem{Lins Neto} 
A. Lins Neto. 
\newblock{\em Algebraic solutions of polynomial differential equations and foliations in dimension two.}
\newline 
\newblock{Holomorphic Dynamics (1986) (Lecture Notes in Mathematics, 1345).} Springer (1988), p.~192–232.

\bibitem{Loray-Rebelo}
F.~Loray, J.~Rebelo.
\newblock{\em Minimal, rigid foliations by curves on \(\PnC\).}
\newline
\newblock {J. Eur. Math. Soc.} (2003), 5, p.~147-201.

\bibitem{Mjuller}
B. Mjuller.
\newblock{\em On the density of solutions of an equation in $\Ptwo$. }
\newline 
\newblock {Mat. Sb.} (1975), 98 (140), p.~325-338.

\bibitem{Nakai}
I. Nakai. 
\newblock{\em Separatrices for nonsolvable dynamics on C,0.}
\newline
\newblock{Ann. Inst. Fourier} (1994), 44-2, p.~569-599. 

\bibitem{Oka} 
M. Oka.
\newblock{\em On mixed projective curves.}
\newline
\newblock {Singularities in geometry and topology (2009). Proceedings of the 5th Franco-Japanese symposium on singularities} (EMS) IRMA Lect. in Math. and Theoretical Physics, (2012), 20, p.~133-147. 

\bibitem{Robinson}
C.~Robinson.
\newblock{\em Structural Stability Manifolds with boundary.}
\newline
\newblock{Journal of differential equations} (1980), 37, p.~1-11.

\bibitem{Scherbakov}
A.A. Scherbakov.
\newblock{\em On the density of an orbit of a pseudogroup of conformal mappings and a generalization of the Hudai-Verenov theorem}
\newline
\newblock{Vestnik Movskovskogo Universiteta. Mathematika} (1982), 31-4, p.~10-15.

\bibitem{Takeuchi}
A. Takeuchi.
\newblock{\em Domaines Pseudo-Convexes dans les variétés Kählériennes.} 
\newline 
\newblock {Jour. Math. Kyoto University} (1967), 6-3 p.~323-357.

\end{thebibliography}
\end{document}